\documentclass[a4paper,10pt]{article}
\usepackage[utf8]{inputenc}

\usepackage{hyperref}
\usepackage{verbatim} 
\usepackage{graphicx}
\usepackage{caption}
\usepackage{subcaption}
\usepackage{amssymb}
\usepackage{textcomp} 
\usepackage{color}
\usepackage{enumerate}
\usepackage{mathrsfs}
\usepackage{multicol}
\usepackage{hyperref}
\usepackage{amsopn}

\usepackage{amsmath}
\usepackage{amsthm}
\usepackage{amssymb}
\usepackage{stmaryrd}

\usepackage{multicol}
\usepackage{subcaption}
\usepackage{enumitem}
\usepackage{xcolor}
\usepackage{xfrac} 
\usepackage{booktabs}

\usepackage{algorithm}
\usepackage{algorithmic}

\newcommand{\N}{\mathbb{N}}

\newcommand{\R}{\mathbb{R}}

\newcommand{\Uk}{\overline{|A^*q_k|}}
\newcommand{\Gk}{\underline{\norm{\nabla A^*q_k}_2}}

\newcommand{\calB}{\mathcal{B}}
\newcommand{\calC}{\mathcal{C}}
\newcommand{\calD}{\mathcal{D}}

\newcommand{\calM}{\mathcal{M}}
\newcommand{\calP}{\mathcal{P}}

\newcommand{\calV}{\mathcal{V}}

\newcommand{\eqdef}{\ensuremath{\stackrel{\mbox{\upshape\tiny def.}}{=}}}

\newcommand{\abs}[1]{\left\vert #1 \right\vert}
\newcommand{\norm}[1]{\Vert #1 \Vert}

\newcommand{\set}[1]{\left\lbrace #1\right\rbrace}
\newcommand{\sse}{\subseteq}
\newcommand{\sprod}[1]{\left\langle #1 \right\rangle}

\usepackage{dsfont}

\newcommand{\geqsim}{\gtrsim}
\newcommand{\leqsim}{\lesssim}

\DeclareMathOperator{\vertx}{vert}
\DeclareMathOperator{\dist}{dist}
\DeclareMathOperator{\distH}{dist_{\mathcal{H}}}

\DeclareMathOperator{\sgn}{sign}
\DeclareMathOperator{\supp}{supp}

\DeclareMathOperator{\Id}{Id}

\newcommand{\argmax}{\mathop{\mathrm{argmax}}}

\DeclareMathOperator{\dom}{dom}
\DeclareMathOperator{\interior}{int}
\DeclareMathOperator{\vol}{vol}

\newcommand{\vertices}{\mathrm{vert}}

\DeclareMathOperator{\diam}{diam}

\newtheorem{lemma}{Lemma}
\newtheorem{proposition}[lemma]{Proposition}
\newtheorem{theorem}[lemma]{Theorem}

\newtheorem{assumption}{Assumption}
\newtheorem{definition}{Definition}
\newtheorem{remark}{Remark}
\theoremstyle{definition}

\usepackage{etoolbox}
\let\bbordermatrix\bordermatrix
\patchcmd{\bbordermatrix}{8.75}{4.75}{}{}
\patchcmd{\bbordermatrix}{\left(}{\left[}{}{}
\patchcmd{\bbordermatrix}{\right)}{\right]}{}{}

\graphicspath{{./figures/}}

\title{Grid is Good: \\ Adaptive Refinement Algorithms for Off-the-Grid Total Variation Minimization}
\author{Axel Flinth, Fr\'ed\'eric de Gournay and Pierre Weiss}

\begin{document}

\maketitle


\begin{abstract}
We propose an adaptive refinement algorithm to solve total variation regularized measure optimization problems. 
The method iteratively constructs dyadic partitions of the unit cube based on i) the resolution of discretized dual problems and ii) on the detection of cells containing points that violate the dual constraints.
The detection is based on upper-bounds on the dual certificate, in the spirit of branch-and-bound methods. 
The interest of this approach is that it avoids the use of heuristic approaches to find the maximizers of dual certificates. 
We prove the convergence of this approach under mild hypotheses and a linear convergence rate under additional non-degeneracy assumptions.
These results are confirmed by simple numerical experiments. 
\footnote{The title of this article is a reference to the seminal paper of J. Tropp on sparse approximation   \cite{tropp2004greed}, which was itself a reference to the movie Wall Street of Oliver Stone.}
\end{abstract}


\section{Introduction}

We develop and analyze an algorithm to solve the following problem:
\begin{equation}
 \inf_{\mu \in \calM(\Omega)} J(\mu) \eqdef \norm{\mu}_\calM + f(A\mu) \tag{$\calP(\Omega)$} \label{eq:primal},
\end{equation}
where $\Omega \subset \R^D$ is a compact domain, $\calM(\Omega)$ is the set of Radon measures on $\Omega$, $A:\calM(\Omega)\to \R^M$ is a continuous operator, $\norm{\cdot}_\calM$ is the total variation and $f:\R^M\to\R\cup\{+\infty\}$ is a convex, proper function. A key property of this problem is that it admits \emph{sparse} solutions (see e.g. \cite{boyer2018representer}) of the form 
\begin{equation*}
    \mu^\star = \sum_{s=1}^S \alpha_s^\star \delta_{x_s^\star},
\end{equation*}
where $\alpha_s^\star\in \R$ are weights, $x_s^\star \in \Omega$ are locations and where the number of sources $S$ satisfies $S\leq M$. This feature shows that a finite dimensional problem is somewhat hidden in \eqref{eq:primal}. It makes it possible to design specific and efficient numerical procedures, for the challenging measure optimization problem \ref{eq:primal}.

\paragraph{Applications}

This problem and its variants appear in various fields. In \emph{inverse problems}, it is used heavily for sparse source localization and super-resolution \cite{de2012exact,bredies2013inverse,tang2013compressed,candes2014towards,duval2015exact}. It is also used in \emph{optimal control} with sparse controls \cite{clason2012measure,kunisch2016optimal}.
In \emph{approximation theory}, a ``generalized'' version of this problem was revisited recently in \cite{unser2017splines}. Given a surjective Fredholm operator $L:\calB(\Omega)\to \calM(\Omega)$, where $\calB(\Omega)$ is a suitably defined Banach space, consider the following problem:
\begin{equation}\label{eq:GTV}
 \inf_{u \in \calB(\Omega)} \norm{Lu}_{\calM} + f(Au).
\end{equation}
The solutions of this problem are (generalized) splines with free knots \cite{unser2017splines}. Following \cite{flinth2017exact} and letting $L^+$ denote a pseudo-inverse of $L$, this problem can be rephrased as
\begin{equation}\label{eq:GTV2}
 \inf_{\substack{\mu \in \calM(\Omega) \\ u_K\in \mathrm{ker}(L)}} \norm{\mu}_{\calM} + f(A (L^+\mu+u_K)),
\end{equation}
which is an instance of \eqref{eq:primal}.

\paragraph{Exchange algorithms}

In this work, we will introduce and study a variant of an \emph{exchange algorithm}. 
To explain its principle, let us introduce the dual problem \eqref{eq:dual} to \eqref{eq:primal}:
\begin{equation}
    \sup_{\substack{q \in \R^M \\ \norm{A^*q}_{L^\infty(\Omega)} \leq 1}} -f^*(q), \tag{$\calD(\Omega)$} \label{eq:dual}
\end{equation}
where $\norm{A^*q}_{L^\infty(\Omega)}\eqdef \sup_{x\in \Omega} |A^*q|(x)$.
This dual formulation will play a pivotal role in our analysis. It involves the optimization of a finite dimensional variable $q\in \R^M$, subject to an infinite number of linear constraints $\{|A^*q(x)|\leq 1, \forall x\in \Omega \}$. It is therefore called a \emph{semi-infinite} program \cite{hettich1993semi,rettich98,hettich1983review}.
The first algorithm proposed to tackle it is usually attributed to Remez and his exchange algorithm \cite{remez1934determination}. It dates back to the 1930's and was adapted to a specific problem of the form \ref{eq:primal}.
The general idea is to define a sequence of discretization sets $(\calV_k)_{k\in \N}$, where $\calV_k$ is a finite set of points (vertices) in $\Omega$. We can then define the discretized primal and dual problems as follows
\begin{align*}
    \inf_{\mu \in \calM(\calV_k)} \norm{\mu}_\calM + f(A\mu) \tag{$\calP(\calV_k)$} \label{eq:discprimal}
\end{align*}
\begin{align*}
    \sup_{\substack{ q \in \R^M \\  \|A^*q\|_{L^\infty(\calV_k)} \leq 1}} -f^*(q).  \tag{$\calD(\calV_k)$} \label{eq:discdual}
\end{align*}

Both problems are finite dimensional convex problem, and can be solved with off-the-shelf solvers. 
The simplest approach is to define $\calV_k$ as a Euclidean grid with edge-length $2^{-k}$ \cite{tang2013sparse,debarre2022tv}. Proving the convergence of this approach is rather straightforward. A problem however is the numerical complexity which explodes rapidly while $k$ increases. 

A lighter adaptive method consists in using the dual variable $q_k$ to construct $\calV_{k+1}$. 
It satisfies $|A^*q_k|(x)\leq 1$ for $x\in \calV_k$ by construction. 
However, there may exist locations $x\in \Omega \setminus \calV_k$ with $|A^*q_k|(x)>1$. Such points are candidates to be added to $\calV_{k+1}$. Perhaps the most popular approach in this class is the Frank-Wolfe \cite{frank1956algorithm} approach. It consists in adding only the global maximizer of $|A^*q_k|$ at each step. It is described precisely in Algorithm \ref{alg:frank_wolfe}.
\begin{algorithm}[h!] \caption{The Frank-Wolfe Algorithm\label{alg:frank_wolfe}}
\begin{algorithmic}[1]
    \STATE \textbf{Input:} \\
    $\quad \bullet$ Initial discretization set $\calV_0$, set $k=0$
    \STATE \textbf{WHILE} a stopping criterion is not satisfied
    \STATE $\qquad 1)$ Determine a solution $q_k$ of $\calD(\calV_k)$
    \STATE $\qquad 2)$ Determine $x_k^\star \eqdef \argmax_{x\in \Omega} |A^*q_k|(x)$.
    \STATE $\qquad 3)$ Set $\calV_{k+1} = \calV_k \cup \{x_k^\star\}$.
    \STATE \textbf{Output:}  \\ 
        $\quad \bullet$ The dual solution $q_k$ of \eqref{eq:discdual}. \\
\end{algorithmic}
\end{algorithm}
It was revived in signal processing thanks to Bredies et al in \cite{bredies2013inverse}. Its connection with the exchange algorithms was recalled in \cite{eftekhari2018bridge}. A linear convergence theory was developed independently by Walter and Pieper in \cite{pieper2019linear} and by the authors in \cite{FGW2019ExchangeI}.
Despite nice theoretical properties, this approach suffers from one major issue, which is the main motivation for the present paper:
\begin{center}
    \emph{How can we find the maximizer $x_k^\star \eqdef \argmax_{x\in \Omega} |A^*q_k|(x)$?}
\end{center}
This problem has no reason to be simple: it is nonconvex and depends highly on the properties of the linear forms $a_m$. As far as we could judge, it is usually tackled with heuristics, which are often swept under the carpet. For instance, multiple gradient or Newton ascents are launched in parallel starting from a set of points covering $\Omega$ sufficiently finely. At the end of the process, the point with the largest value is then kept as an approximation of $x_k^\star$. Our experience is that tuning the hyper-parameters in this quest for the global minimizer is time consuming and can represent a real headache for the optimizer. The main objective of this paper is to tackle this issue, by looking for approximate, easily detectable maximizers only. 

\paragraph{Alternative solvers}
Finally, let us mention that alternative algorithms with a different flavor are available. 
A possibility is to use the Lasserre hierarchies, which are designed to solve near generic measure optimization problems \cite{lasserre2015introduction,de2017exact}. They however scale poorly for large $M$ and will not be considered further in this work. 
Another possibility is to use continuous optimization procedures. The idea is to parameterize the measure $\mu$ by the locations and weights of its masses. Given a number $N\geq M$ of particles, a set of locations $X=(x_1,\hdots, x_N)$ in $\Omega^N$ and a weight vector $\alpha\in \R^N$, we can define the mapping
\begin{equation*}
    \mu(X,\alpha) \eqdef \sum_{n=1}^N \alpha_n \delta_{x_n}. 
\end{equation*}
Plugging it in the primal problem \ref{eq:primal}, we obtain the following non-convex, finite dimensional problem
\begin{equation}\label{eq:parameterized_problem}
    \inf_{\substack{X\in \Omega^N \\ \alpha \in \R^N}} \|\alpha\|_1 + f(A\mu(X,\alpha)).
\end{equation}
It can be solved with continuous optimization routines such as proximal gradient descents. 
The difficulty with this type of approach is the initialization: how to set the number $N$ of particles, their locations and weights? A possibility is to discretize the domain $\Omega$ finely. A nice global convergence theory has then been developed in \cite{chizat2018global,chizat2022sparse}. The price to pay is a large number of variables $N\gg M$. This is somewhat unsatisfactory since we know in advance that a solution supported on at most $M$ points exists. A possibility to avoid this pitfall consists in using an exchange algorithm to find an approximate solution to \ref{eq:primal}, and then fine-tune it with a continuous optimization routine on \eqref{eq:parameterized_problem}. It now possesses a rich convergence theory \cite{traonmilin2020basins,FGW2019ExchangeI}.
It can be used in conjunction with the algorithm proposed in this work. A variant of this approach is to launch the continuous optimization technique after every step in the Frank-Wolfe algorithm, and not only use it to fine-tune in the end. This approach was introduced and coined the sliding Frank-Wolfe algorithm in \cite{denoyelle2019sliding}.


\section{Preliminaries}

We will work on the domain $\Omega = [0,1]^D$ for simplicity. This is not a real restriction, since we can fit any compact domain of $\R^D$ to $\Omega= [0,1]^d$ using a non-degenerate affine transformation. Using this map, we may push forward measurement functions and measures, in a bijective fashion. We let $\calM(\Omega)$ denote the space of Radon measures of bounded total variation on $\Omega$ and $\calC_0(\Omega)$ the space of continuous functions on $\Omega$ vanishing on the boundary. Note that if we equip $\calM(\Omega)$ with the total variation norm $\norm{\, \cdot \,}_\calM$ and $\calC_0(\Omega)$ with the supremum norm, $\calM(\Omega)$ can be identified with the dual of $\calC_0(\Omega)$. For a subset $\omega$ of $\R^D$, we let $\vol(\omega)$ denote its volume (or Lebesgue measure).
The notation $\llbracket 1,N\rrbracket$ indicates the set of integers from $1$ to $N$. 
The relation $f\lesssim g$ indicates that $f$ is dominated by $g$ up to positive multiplicative constant.
The relation $f\asymp g$ indicates that $f$ and $g$ are equivalent, i.e. that there exists two constants $0<c_1\leq c_2$ such that $c_1 g \leq f \leq c_2 g$.


\subsection{Cells and cell partitions}
The proposed algorithms rely on the use of $2^D$-trees. 
We iteratively partition hypercubes in $2^D$ equal parts.
For instance, we will use binary trees in 1D, quadtrees in 2D, and octrees in 3D. 
Let us define some objects.

\begin{definition}[Cells, vertices, edge-length]
We call a subset $\omega \sse [0,1]$ a \emph{dyadic cell} (or simply a cell) if it is of the form $$\omega = x + 2^{-J} \cdot [0,1]^d,$$ where $J$ is a non-negative integer and  $x \in (2^{-J}\cdot\set{0, 1, \dots, 2^J-1})^d$. 

For a cell $\omega$, we let $\vertx(\omega)$ denote its vertices and $|\omega|$ denote its edge-length.
\end{definition}

\begin{definition}[Cell partition]
A cell partition $\Omega_k$ is a collection of cells such that $\Omega = \cup_{\omega \in \Omega_k} \omega$ and $\vol(\omega'\cap \omega)=0$ for all $\omega,\omega'\in \Omega_k$ with $\omega\neq \omega'$. That is to say, two cells in $\Omega_k$ can only have faces in common.
\end{definition}

\subsection{Measurement operator}

Throughout the paper, we will work under the following assumption.
\begin{assumption}[Continuous operator\label{ass:A}]
    The operator $A:\calM(\Omega)\to \R^M$ is weak-$*$-continuous. Equivalently, the \emph{measurement functionals} $a_m^*$ defined by $\sprod{a_m^*, \mu} = (A(\mu))_m$ are given for all $\mu\in \calM(\Omega)$ by 
    \begin{align*}
        \sprod{a_m^*, \mu} = \int_{\Omega} a_m d\mu,
    \end{align*}
    for functions $a_m \in \calC_0(\Omega)$. 
\end{assumption}
Given $q\in \R^M$, $A^*q=\sum_{m=1}^M q_m a_m$ is a continuous function. Assuming than $a_m\in \calC^r(\Omega)$, we let $(A^*q)'$, $(A^*q)''$, $(A^*q)^{(r)}$ denote its derivative, its Hessian and it $r$-th tensor derivative. We define the following constants
\begin{equation}
\label{def:kappa_r}
     \kappa_r \eqdef \sup_{\|q\|_2\leq 1} \sup_{x\in \Omega} \norm{(A^*q)^{(r)}(x)}_{2\to 2},
\end{equation}
where $r$ indicates the derivative's order and where $\norm{\cdot}_{2\to 2}$ is the spectral norm. 

\subsection{Set distances}

The ``distance'' between two sets $X_1$ and $X_2$ in $\R^D$ is defined by
\begin{equation*}
    \dist(X_1 , X_2) \eqdef \inf_{x_1\in X_1, x_2\in X_2} \|x_1-x_2\|_2. 
\end{equation*}
Notice that $\dist$ is not a proper distance. In particular, it does not satisfy the triangle inequality.
For a point $x\in \R^D$ and a set $X\subset \R^D$, we will use the shorthand notation
\begin{equation*}
    \dist(x,X) \eqdef \dist(\{x\},X).
\end{equation*}
The \emph{Hausdorff distance} between $X_1$ and $X_2$ is defined by
\begin{equation*}
    \distH(X_1 | X_2) \eqdef \sup_{x_2\in X_2}\inf_{x_1\in X_1} \|x_1-x_2\|_2. 
\end{equation*}
Notice that this distance is asymmetric: in general $\distH(X_1 | X_2) \neq \distH(X_2 | X_1)$.
The following inequality will play an important role in the analysis.
\begin{proposition}[Triangle inequality for set distances\label{prop:triangle_set}]
    For any triple of sets $X_1,X_2,X_3$ in $R^D$ we have 
    \begin{equation}\label{eq:set_triangle_inequality}
        \dist(X_1,X_2) \leq \distH(X_1|X_3) + \dist(X_3,X_2).
    \end{equation}
\end{proposition}

\subsection{Primal, dual and existence}

Our results will be established under the following assumptions on $f$ and $A$.  

\begin{assumption}[\label{ass:weakReg}A convexity assumption]
 The function $f:\R^M\to \R\cup \{+\infty\}$ is a convex lower semi-continuous function with $\interior(\dom(f))\neq \emptyset$.
\end{assumption}

\begin{assumption}[\label{ass:coercivity}Coercivity]
 The functional $J$ is coercive, meaning that $J(\mu) \to \infty$ when $\norm{\mu}_{\calM(\Omega)}\to \infty$. 
\end{assumption}
Notice that Assumption \ref{ass:coercivity} is granted if $f$ is lower-bounded. The following result relates the primal and the dual. 
\begin{proposition}[Existence and strong duality] \label{prop:duality}
Let $\calV \subseteq \Omega$ denote a subset of $\Omega$. 
Assume that there exists $\mu\in \calM(\Omega)$ supported on $\calV$ with $A\mu \in \interior(\dom(f))$. 
Then, under Assumptions \ref{ass:A}, \ref{ass:weakReg}  and \ref{ass:coercivity}, the following statements hold true:
\begin{itemize}
 \item The primal problem $(\calP(\calV))$ has a nonempty set of solutions, bounded in total variation norm.
 \item The dual $(\calD(\calV))$ has a nonempty set of solutions, which is also bounded.
 \item The following strong duality result holds
\begin{equation*}
 \min_{\mu \in \calM(\calV)} \norm{\mu}_{\calM(\calV)} + f(A\mu) = \max_{q \in \R^M, \|A^*q\|_{L^\infty(\calV)} \leq 1} - f^*(q).
\end{equation*}
 \item Let $(\mu^\star,q^\star)$ denote a primal-dual pair. They are related by the following primal-dual relationships:
 \begin{equation}\label{eq:primal_dual_relationship}
  A^*q^\star \in \partial_{\|\cdot\|_\calM}(\mu^\star) \mbox{ and } -q^\star\in \partial f(A\mu^\star).
 \end{equation}
\end{itemize}
\end{proposition}

The left inclusion in \eqref{eq:primal_dual_relationship} implies that the support of a solution $\mu^\star$ satisfies: $\supp(\mu^\star) \subseteq \{x\in \Omega, |A^*q^\star(x)|=1 \}$. 

\begin{remark}
    Strong duality may hold under different assumptions. 
    For instance, if $f$ is polyhedral (allowing the hard constraint $A\mu = b$), then strong duality holds \cite{Borwein1992}, but the dual solution set may be unbounded. Similarly, the coercivity of $J$ is not absolutely needed. If $J$ has a finite dimensional constancy space, the primal solution set still exists, but may be unbounded as well. In both cases, the unboundedness of either the primal or dual problem requires extra technicalities and assumptions in the proofs that we decided to discard. 
\end{remark}

\section{Main results}

This section contains our main findings. All the proofs are post-poned to the appendix. 

\subsection{The algorithm}

The algorithm we propose consists in designing a sequence of cell partitions $(\Omega_k)_{k\in \N}$ of $\Omega$.
At each step of the algorithm, $\Omega_{k+1}$ is constructed by dividing a few cells in $\Omega_k$.
Let $\calV_k$ denote the set of vertices of the partition $\Omega_k$:
\begin{equation*}
    \calV_k \eqdef \{\vertices(\omega) \, \vert \, \omega \in \Omega_k\}.
\end{equation*}
At step $k$, we solve \eqref{eq:discdual}, to obtain a solution $q_k$ of the discretized dual problem. 
It satisfies $\|A^*q_k\|_{L^\infty(\calV_k)}\leq 1$ by construction. However, it might be infeasible for the problem \eqref{eq:dual}, i.e. $\norm{A^*q_k}_{L^\infty(\Omega)}> 1$. 
This suggests that we should detect the cells $\omega\in \Omega_k$ for which $\|A^*q_k\|_{L^\infty(\omega)}>1$ and subdivide them. 
To this end, we suppose that we have access to a set of {\em candidate cells} $\Omega^\star_k \subset \Omega_k$ that are likely to satisfy $\norm{A^*q_k}_{L^\infty(\Omega)}> 1$. 
To make it easier to control the growth of $\abs{\calV_k}$ with the iteration number $k$, we propose to not refine all candidates in $\Omega^\star_k$, but only the largest of them. The complete solver is described in Algorithm \ref{alg:genericDisc}.  
One iteration of the algorithm is displayed in Figure \ref{fig:refinement}.
\begin{figure}[!h]
\centering
    \includegraphics[width=.95\textwidth]{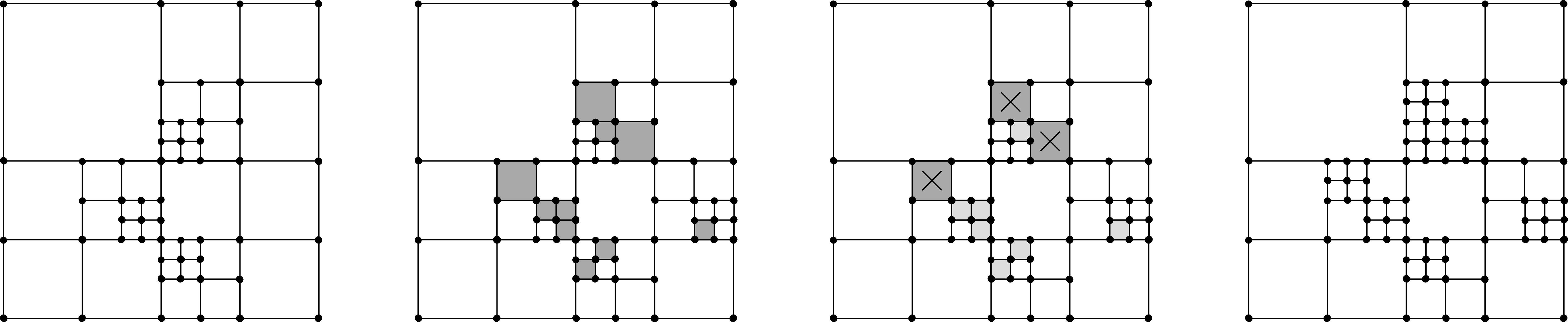}
    \caption{The refinement . From left to right: i) a cell partition, ii) the partition with the candidate cells grayed out, iii) only the candidate cells with largest diameter are selected for refinement and iv) the new resulting cell partition.\label{fig:refinement}}
\end{figure}

\begin{algorithm}[h!] \caption{Adaptive Refinement Algorithm\label{alg:genericDisc}}
\begin{algorithmic}[1]
    \STATE \textbf{Input:} \\
    $\quad \bullet$ Operator $A$ \\ 
    $\quad \bullet$ Initial partition $\Omega_0$ (e.g.  $\Omega_0 = \{\Omega\}$)\\
    $\quad \bullet$ Target precision $J$ \\
    $\quad \bullet$ Solver for the discretized primal \eqref{eq:discprimal} and dual \eqref{eq:discdual} \\
    $\quad \bullet$ Set $k=1$
    \STATE \textbf{WHILE} $\displaystyle \max_{\omega \in \Omega_{k-1}^*} |\omega| \geq 2^{-J}$ 
    \STATE $\qquad 1)$ Determine a solution $q_k$ of \eqref{eq:discdual}
    \STATE $\qquad 2)$ Determine the candidate cells $\Omega_k^\star$.
    \STATE $\qquad 3)$ Subdivide the cells in $\Omega_k^\star$ with largest diameter
    \STATE $\qquad 4)$ $k=k+1$
    \STATE \textbf{Output:}  \\ 
        $\quad \bullet$ The dual solution $q_k$ of \eqref{eq:discdual}. \\
        $\quad \bullet$ The primal solution $\mu_k$ of \eqref{eq:discprimal}.
\end{algorithmic}
\end{algorithm}

\subsection{Assumptions on the selection process}

In this paragraph, we discuss how to construct the set of candidate cells $\Omega_k^\star$. 
Following \cite{FGW2019ExchangeI}, we let $X_k$ denote the set of local maximizers of $\abs{A^*q_k}$ exceeding $1$, i.e.
\begin{align*}
 X_k \eqdef \set{x \in \Omega \, \vert \, x \text{ is a local maximizer of }\abs{A^*q_k}, \abs{A^*q_k(x)}\geq 1}.
\end{align*}
To obtain convergence guarantees we will work under the following hypothesis. 
\begin{assumption}[Generic convergence assumption\label{ass:generic_refinement}]
For each $x\in X_k$, at least one of the cells in $\Omega_k$ containing $x$ is a candidate for refinement. In other terms, $\Omega_k^\star$ satisfies
        \begin{equation*}
         X_k\subset \bigcup_{\omega\in \Omega_k^\star} \omega.
        \end{equation*}
\end{assumption}
The above condition is too weak to allow a control of the numerical complexity. For instance the choice $\Omega_k^\star=\Omega_k$ -- which corresponds to refine uniformly at each iteration -- obeys Assumption~\ref{ass:generic_refinement}, but leads to an exponential growth of $\abs{\Omega_k}$. Our complexity analysis will therefore rely upon the following extra hypothesis. 
\begin{assumption}[Second order approximation\label{ass:second_order_condition}]
    There exists a constant $\kappa>0$ independent of the iteration $k$ and the cell $\omega$ such that for every $\omega\in \Omega_k^\star$, we have
    $$ \Vert A^\star q_k\Vert_{L^\infty(\omega)} \ge 1-\kappa |\omega|^2,$$
    where $|\omega|$ is the edge-length of the cell $\omega$. 
\end{assumption}

\subsection{Construction of selection processes}

In this section, we discuss how to construct rules satisfying assumptions \ref{ass:generic_refinement} and \ref{ass:second_order_condition}.

\subsubsection{Ideal selection} 

An obvious refinement rule that obeys both assumptions \eqref{ass:generic_refinement} and \eqref{ass:second_order_condition} is to let
\begin{align*}  
\Omega_{k, \mathrm{ideal}}^* = \set{ \omega \in \Omega_k \, \vert \, \omega \cap X_k \neq \emptyset}.
\end{align*}
Indeed, Assumption \ref{ass:generic_refinement} can be reformulated as $\set{ \omega \in \Omega_k \, \vert \, \omega \cap X_k \neq \emptyset} \sse \Omega_{k}^*$, which certainly is true for the above choice. As for Assumption \ref{ass:second_order_condition}, note that for all $\omega$ with $X_k\cap \omega \neq \emptyset$, we have $\norm{A^*q_k}_{L^\infty} \geq 1 \geq 1 - \kappa \abs{\omega}^2$, since $\omega$ contains a cell point where $\abs{A^*q_k}$ succeeds one.

In order to apply this rule, we however need to know which cells contain the elements of $X_k$. As discussed in the introduction, this is in general  infeasible. 

\subsubsection{Upper-bounds selections} 

To obtain resolvable but still powerful enough selection processes, we will instead rely on the design of \emph{simple to evaluate upper-bounds} $\Uk(\omega)\in \R$ satisfying $\Uk(\omega) \geq \| A^*q_k\|_{L^\infty(\omega)}$.
Equipped with such an upper-bound, we can define the candidate cells as:
$$\Omega_k^\star=\left\{\omega\in \Omega_k \, \vert \, \Uk(\omega) \geq 1 \right\}.$$
By construction, this selection process guarantees Assumption \ref{ass:generic_refinement}. Indeed, a cell $\omega$ with $\omega \cap X_k$ obviously obeys $\|A^*q_k\|_{L^\infty(\omega)} \geq 1$, and therefore also $\Uk(\omega)\geq \|A^*q_k\|_{L^\infty(\omega)} \geq 1$.

Notice that this principle is similar to a branch-and-bound approach \cite{lawler1966branch}. We base our decisions on upper-bounds, which secure that some regions of space can be safely neglected. An important difference lies in the fact that the objective function $|A^*q_k|$ varies at each iteration, meaning that one region which might have been discarded at one iteration can be refined some iterations later. Let us describe two such upper bounds.

\begin{definition}[A first order selection process]
Assume that $a_m\in C^1(\Omega)$ for all $m$. Let us define
\begin{equation} \label{eq:first order}
\Uk(\omega)=\inf_{v \in \vertx(\omega)} |A^*q_k(v)|+ \kappa_1(q_k,\omega) \diam(\omega),
\end{equation}
with
\begin{equation}\label{eq:def_kappa1_loc}
    \kappa_1(q_k,\omega) = \sum_{m=1}^M |q_k[m]| \|a_m'\|_{L^\infty(\omega)}.
\end{equation} 
The first order selection process is defined as 
\begin{align*}
\Omega_{k,1}^\star=\left\{\omega\in \Omega_k \, \vert \, \Uk(\omega) \geq 1 \right\} 
\end{align*}
\end{definition}
\begin{proposition}\label{prop:first_order_selection}
    The first order selection process $\Omega_{k,1}^\star$ satisfies Assumption \ref{ass:generic_refinement}. However it may not respect Assumption \ref{ass:second_order_condition}.
\end{proposition}

In order to satisfy both Assumption \ref{ass:generic_refinement} and Assumption \ref{ass:second_order_condition}, we will use second order selection rules. 
\begin{definition}[A second order selection process]\label{def:second_order_selection}
Assume that $a_m\in C^2(\Omega)$ for all $m$. 
For any cell $\omega$, and all $q_k\in \R^M$, define 
\begin{equation}
\label{eq:second_order_bnd}
\Uk(\omega) \eqdef \inf_{v\in \vertices(\omega)} \sup_{x\in \omega}  |A^*q_k(v) + \langle (A^*q_k)'(v), x-v\rangle| +  \frac{\kappa_2(q_k,\omega)}{2} \Vert x-v\Vert^2,
\end{equation}
with 
\begin{equation}
    \label{eq:defin:kappa_2}
     \kappa_{2}(q_k,\omega) = \sum_{m=1}^M |q_k[m]|\sup_{x\in \omega} \|a_m''(x)\|_{2\to 2}.  
\end{equation}
The second order selection process is defined as 
\begin{align}
\Omega_{k,2}^\star=\left\{\omega\in \Omega_k \, \vert \, \Uk(\omega) \geq 1 \right\} \label{eq:Second order}
\end{align}
\end{definition}

\begin{proposition}\label{prop:second_order_selection}
    If the sequence $(q_k)_{k\in \N}$ is uniformly bounded, then the second order selection process $\Omega_{k,2}^\star$ satisfies both Assumption \ref{ass:generic_refinement} and \ref{ass:second_order_condition}.
\end{proposition}

Importantly, notice that the problem
\[ \sup_{x\in \omega}  |A^*q_k(v) + \langle (A^*q_k)'(v), x-v\rangle|+\frac{\kappa_2(q_k,\omega)}{2}\Vert x-v\Vert^2,\]
consists in maximizing a convex function over a polyhedron.  A solution is therefore attained on $\vertx(\omega)$, and can be evaluated in constant time per cell.

\begin{remark}
\label{rem:an_infinite_number_of_estimators_of_derivative}
    The values $\kappa_1(q_k,\omega)$ and $\kappa_2(q_k,\omega)$ can be replaced by any upper-bound on the Lipschitz constant of $A^*q_k$ and the Lipshitz constant of $(A^*q_k)'$ respectively. In particular, it is possible to use the global bounds $\kappa_1(q_k,\omega) = \kappa_1 \|q_k\|_2$ and $\kappa_2(q_k,\omega) = \kappa_2 \|q_k\|_2$, where $\kappa_1$ and $\kappa_2$ are defined in equation \eqref{def:kappa_r}.
\end{remark}

\subsubsection{Combining upper-bounds and lower-bounds on the gradient norm}

The larger $\Omega_k^\star$, the higher the chances of selecting unwanted cells for refinement. 
To reduce the cardinality of the candidate cells, it makes sense to only refine the cells where the function $\abs{A^*q_k}$ might surpass $1$ \emph{and} where the gradient's norm $\|(A^*q_k)'(x)\|_2$ might cancel.
Indeed, Assumption~\ref{ass:generic_refinement} only requires the {\em local maximizers} of $|A^\star q_k|$ that surpass $1$ to be selected. In cells containing maximizers, the gradient of $|A^\star q_k|$ vanishes. Consequently we can design a selection process based on lower bounds of the gradient.
\begin{definition}[Second order selection process with first order gradient\label{def:second_order_selection_with_grad}]
Assume that $a_m\in C^2(\Omega)$ for all $m$. Define $\Uk(\omega)$ and $\kappa_2(q_k,\omega)$ as in \eqref{eq:second_order_bnd} and \eqref{eq:defin:kappa_2} respectively.
For any cell $\omega$, and all $q_k\in \R^M$, define
\begin{equation}
\Gk(\omega) = \sup_{v \in \vertx(\omega)} \norm{(A^*q_k)'(v)}_2 - \kappa_2(q_k,\omega)\diam(\omega)\label{eq:grad1} 
\end{equation}
The second order selection process with first order gradient is defined as $$\Omega_{k,2, 1}^\star=\left\{\omega\in \Omega_k \, \vert \, \Uk(\omega) \geq 1 \text{ and } \Gk(\omega) \le 0\right \}.$$
\end{definition}
\begin{proposition}\label{prop:second_order_selection_with_grad}
    If the sequence $(q_k)_{k\in \N}$ is uniformly bounded, then the selection process 
    $\Omega_{k,2,1}^\star$ satisfies both Assumption~\ref{ass:generic_refinement} and \ref{ass:second_order_condition}. Incorporating gradient lower bounds reduces the cardinality of $\Omega_k^\star$ in the sense that the cardinality of $\Omega_k^\star$ in Definition~\ref{def:second_order_selection_with_grad} is non greater than the one of $\Omega_k^\star$ in Definition~\ref{def:second_order_selection}.
\end{proposition}



\begin{remark}
Following Remark~\ref{rem:an_infinite_number_of_estimators_of_derivative}, the definition of $\kappa_2(q_k,\omega)$ given in ~\eqref{eq:defin:kappa_2} can be replaced by any upper bound of the Lipschitz constant of the gradient of $|A^* q_k|$. Moreover, we use a first order Taylor expansion for $\Gk$. It is possible to replace or even combine this lower bound with bounds stemming from higher order Taylor expansions.
\end{remark}

\subsection{Generic convergence guarantees}

To obtain a generic convergence result, we first need to prove that the algorithm is well defined. To this end, we introduce the following assumption.

\begin{assumption}[Well-posedness of the algorithm \label{ass:wellposedness}]
The initial set of vertices $\calV_0$ is admissible in the sense that there exists $\mu\in \calM(\calV_0)$ with $A\mu \in \interior(\dom(f))$.
\end{assumption}

\begin{theorem} \label{th:generalConv} 
    Under Assumptions \ref{ass:A}, \ref{ass:weakReg}, \ref{ass:coercivity}, \ref{ass:generic_refinement} and \ref{ass:wellposedness}, the sequences $(\mu_k)_{k \in \N}$ and $(q_k)_{k\in \N}$ defined in Algorithm \eqref{alg:genericDisc} are well-defined. 
    They contain subsequences that converge weakly to solutions $\mu^\star$ and $q^\star$ of \eqref{eq:primal} and \eqref{eq:dual}, as well as in optimal function value. If either the primal or dual solution is unique, the whole corresponding sequence converges.
\end{theorem}
\begin{proof}
    The proof of this theorem is postponed to Section \ref{sec:proof_generic_convergence}.
\end{proof}

\begin{remark}
	In \cite{FGW2019ExchangeI}, additional assumptions were made (either smoothness of $f$  or surjectivity of $A$ restricted to $\calM(\calV_0)$) to obtain generic convergence. The reason we can remove that assumption is based on a refined analysis. 
\end{remark}
\begin{remark}
The proof of this result relies on the fact that the sequence $(\calV_k)_{k\in \N}$ is nested. 
In exchange algorithms, it is possible to not only add, but also discard points from $\calV_k$ to construct $\calV_{k+1}$.
The obvious interest is to reduce the numerical complexity.
We do not know if it possible to adapt the algorithm and the proof to allow for points suppression as well.
\end{remark}

\subsection{Linear convergence rates}

Having established the generic convergence result, we move on to providing an  eventual linear convergence rate under additional regularity conditions. We first need a couple of additional regularity conditions on $f$, $A$ and the primal-dual solution pair, which are similar to those in \cite{FGW2019ExchangeI} and \cite{pieper2019linear}.

\begin{assumption}[Linear convergence conditions\label{ass:regularity}]
\
\begin{itemize}
    \item The functionals $a_m$ are twice differentiable: $a_m\in C^2_0(\Omega)$ for all $1\leq m\leq M$.
    \item The function $f$ is convex, differentiable with an $L$-Lipschitz gradient.
\end{itemize}
\end{assumption}

Following \cite{duval2015exact}, we also require the following condition.
\begin{assumption}[Non-degenerate source condition\label{ass:nondegenerate}] 
 We say that the \emph{non-degenerate source condition} \cite{duval2015exact} holds if we have the following:
 \begin{itemize}
  \item The solution $\mu^\star$ of \eqref{eq:primal} is unique and supported on $S\in N$ points
   \begin{align*}
    \mu^\star = \sum_{s=1}^S \alpha_s^\star \delta_{x_s^\star}
 \end{align*}
 for some $\alpha_s^\star \in \R$ and $x_s^\star \in \Omega$. In what follows, we let $X^\star \eqdef \{x_1^\star, \hdots, x_S^\star\}$.
    \item The \emph{dual certificate} $\abs{A^*q^\star}$ is only equal to $1$ in the points $x_1^{\star}, \dots, x_S^{\star}$ and is strictly concave around those points. 
    This ensures the existence of a parameter $\gamma > 0$ and a radius $R >0$ with
        \begin{align} \label{eq:reg}
        B(x_{s_1}^\star,R)\cap B(x_{s_2}^\star,R) = \emptyset, &  \quad \forall s_1\neq s_2 \\
     \sgn(A^*q^\star(x)) (A^*q^\star)'' \preccurlyeq - \gamma \Id  \ &\text{ for $x$ with } \dist(x, X^\star) \leq R, \nonumber \\ 
        \abs{A^*q^\star(x)} \leq 1- \frac{\gamma R^2}{2} \ &\text{ for $x$ with } \dist(x, X^\star) \geq R. \label{eq:reg2}
    \end{align}

    \end{itemize}
\end{assumption}

This last assumption is generic, given that the solution is unique. It is a condition that has appeared in the literature as a mean to prove recovery of sparse measures using problems of the form \eqref{eq:primal} -- see e.g. \cite{poon2018support,candes2014towards}.
We can now formulate our main result.  

\begin{theorem} \label{th:linConv}
Under Assumptions \ref{ass:coercivity}, \ref{ass:generic_refinement}, \ref{ass:second_order_condition}, \ref{ass:regularity} and \ref{ass:nondegenerate}, Algorithm \ref{alg:genericDisc} eventually converges linearly. That is, there exists constants $k_0\in \N$,  $c>0$ (depending on $A$, $f$ and $\mu^\star$) such that the algorithm terminates in no more than $k = k_0 + c SJ$ iterations. 
For sufficiently large $J$, we further have
\begin{align*}
|\calV_k| &= O(J), \qquad &\textrm{controlled complexity}\\ 
\distH(X^\star | \vertx(\Omega_k^\star)) &\lesssim 2^{-J}, \qquad &\textrm{controlled localization} \\
\|q_k-q^\star\|_2  &\lesssim 2^{-J}, \qquad &\textrm{certificate on the dual} \\ 
J(\mu_k) - J(\mu^\star) &\lesssim 2^{-2J}, \qquad &\textrm{certificate on the primal}.
\end{align*}
\end{theorem}
\begin{proof}
The proof of this theorem is quite technical, and is therefore postponed to Section \ref{sec:proof_linear_convergence}. It relies on a few technical inequalities from the companion paper \cite{FGW2019ExchangeI}, but differs significantly to account for the discretization procedure. 

    Informally, it is built using the following arguments. First, appealing to the generic convergence result and the fact that only finitely many cells have an edge length larger than any fixed value $\delta>0$, we argue that after warming period of at most $k_0$ iterations, $q_k$ is close to $q^\star$, and no cells with an edge-length larger than a critical value can be active. Once that happens, the algorithm will only be able to refine cells close to the maximizers $X^\star$. This results in a multiscale refinement of local regions around the sought-for locations, see e.g. Figure \ref{fig:behavior1D_grad}.
\end{proof}
    
\begin{remark}
    We did not keep track of the constants in the above inequalities to simplify the reading. 
    While some of them are explicit, others, like the time $k_0$ to reach a linear convergence rate are not.    
\end{remark}

\begin{remark} We can replace $\mu_k$, the solution of $(\calP(\calV_k))$, by the solution $\tilde \mu_k$ of $\calP(\vertx(\Omega_k^\star))$ defined as
    \begin{equation}
    \tilde \mu_k \eqdef \inf_{\mu \in \calM(\vertx(\Omega_k^\star))} \|\mu\|_{\calM} + f(A\mu).
\end{equation}
and still have $J(\tilde \mu_k) - J(\mu^\star) \lesssim 2^{-2J}$.
The interest of this alternative problem is that the cardinality of $\vertx(\Omega_k^\star)$ is significantly smaller than that of $\calV_k$, helping to reduce the numerical complexity.
\end{remark}

\begin{remark}
	Our algorithm relies on dyadic subdivision of cells. We therefore cannot expect the algorithm to converge faster than linearly. In that regard, Theorem \ref{th:linConv} is optimal.
\end{remark}

\section{Numerical experiments} \label{sec:numerical}

In this section, we aim at illustrating our main findings through some simple numerical experiments.
We consider problems of sparse source recovery problem with filtered measurements. 
That is, given a ground truth $\bar \mu$, we set $y = A\bar\mu$ and let $f(q) = \frac{1}{2}\norm{q-y}_2^2$. 
This yields
\begin{equation*}
    f^*(q')\eqdef \sup_{q\in \R^M} \langle q,q'\rangle - \frac{1}{2}\|q-y\|_2^2 = \frac{1}{2}\|q'\|_2^2 + \langle q',y\rangle.
\end{equation*}
We consider Gaussian measurements functions of the form 
\begin{equation*}
  a_m(x) = \frac{1}{2\pi \sigma} \exp\left( \frac{- \|x - z_m\|_2^2 }{2\sigma^2}\right)
\end{equation*}
for some value $\sigma>0$. 
To properly define our selection procedures in \eqref{eq:defin:kappa_2}, we need an upper-bound on the second order derivatives. 
\begin{proposition}\label{prop:k2_gaussian_2D}
Define
\[\kappa_{2,m}(\omega) =\frac{a_m(\dist(z_m,\omega))}{\sigma^4} \max \left(\sigma^2 , (\dist(z_m,\omega) + \sqrt{D}|\omega|)^2\right)\] then
$	\kappa_{2,m}(\omega) \ge \sup_{x\in \omega} \Vert a_m''(x)\Vert_{2\to 2} 
.$
We can choose $\kappa_2(q_k,\omega) =\sum_{m=1}^M |q_k|[m] \kappa_{2,m}(\omega)$, in Proposition~\ref{eq:defin:kappa_2} to define the second order candidates $\Omega_k^\star$.
\end{proposition}

\subsection{Implementation details}
We implement our algorithm in Python using the numpy package. To solve the discretized dual \eqref{eq:discdual}, we rely on the SCS solver of the CVXPY package \cite{diamond2016cvxpy,agrawal2018rewriting}. The selection procedures are defined and implemented as described in the main text.
To assess the convergence rates, we compute the exact solutions of the primal problem by running a fixed step gradient descent in the parameter space, see equation~\eqref{eq:parameterized_problem}, initialized in the 'ground truth measures' we specify. This is sound, since the true solution lies close to them (see e.g. \cite{poon2018support}).

\subsection{1D-experiments}

\paragraph{The problem} We set $y = A\bar \mu$ with $\bar \mu \eqdef 8 \delta_{1/3} - 9\delta_{2/3}$. We choose the Dirac mass locations to lie at $1/3$ and $2/3$, since these points are the hardest to reach with dyadic partitions. The sampling locations and $\sigma$-parameter are set to $z_m \eqdef m/M$ and $\sigma = 2/M$, with $M=20$. 


\paragraph{Second-order upper bound} 
The behavior of the second order selection process algorithm of  Definition~\ref{def:second_order_selection} is displayed in Table \ref{tab:behavior1D_nograd} and Figure \ref{fig:behavior1D_nograd}.
As can be seen in the figure, the algorithm starts by a burn-in period of $4$ iterations. 
This transient behavior explains why the linear convergence rate only occurs after a finite number of iterations. 
Then, only cells in a neighborhood of $\{1/3,2/3\}$ are refined. 
The Table \ref{tab:behavior1D} clearly indicates that the distance $\distH(\calV_k,X^\star)$ decays exponentially fast, illustrating Theorem \ref{th:linConv} and the linear convergence rate.
Observe that less than 300 vertices are enough to obtain a precision $10^{-6}$, while a uniform refinement would require $10^6$ vertices. This illustrates the huge computational/memory advantage of this adaptive method. 

\begin{figure}[!t]
    \centering
    \begin{subfigure}[b]{0.3\textwidth}
      \includegraphics[width=\textwidth]{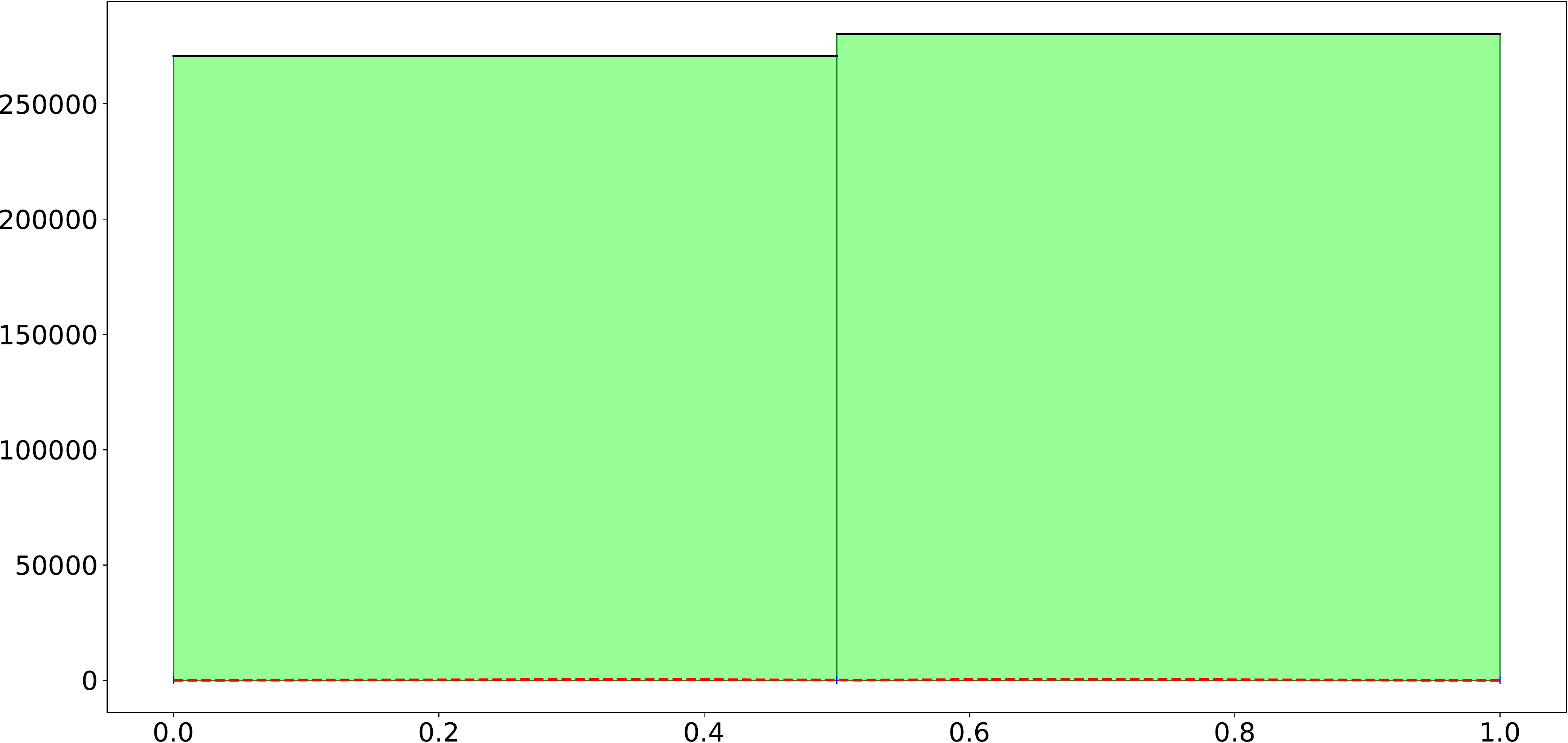}
      \caption{$|\calV_1|=3$}
    \end{subfigure}
    \begin{subfigure}[b]{0.3\textwidth}
      \includegraphics[width=\textwidth]{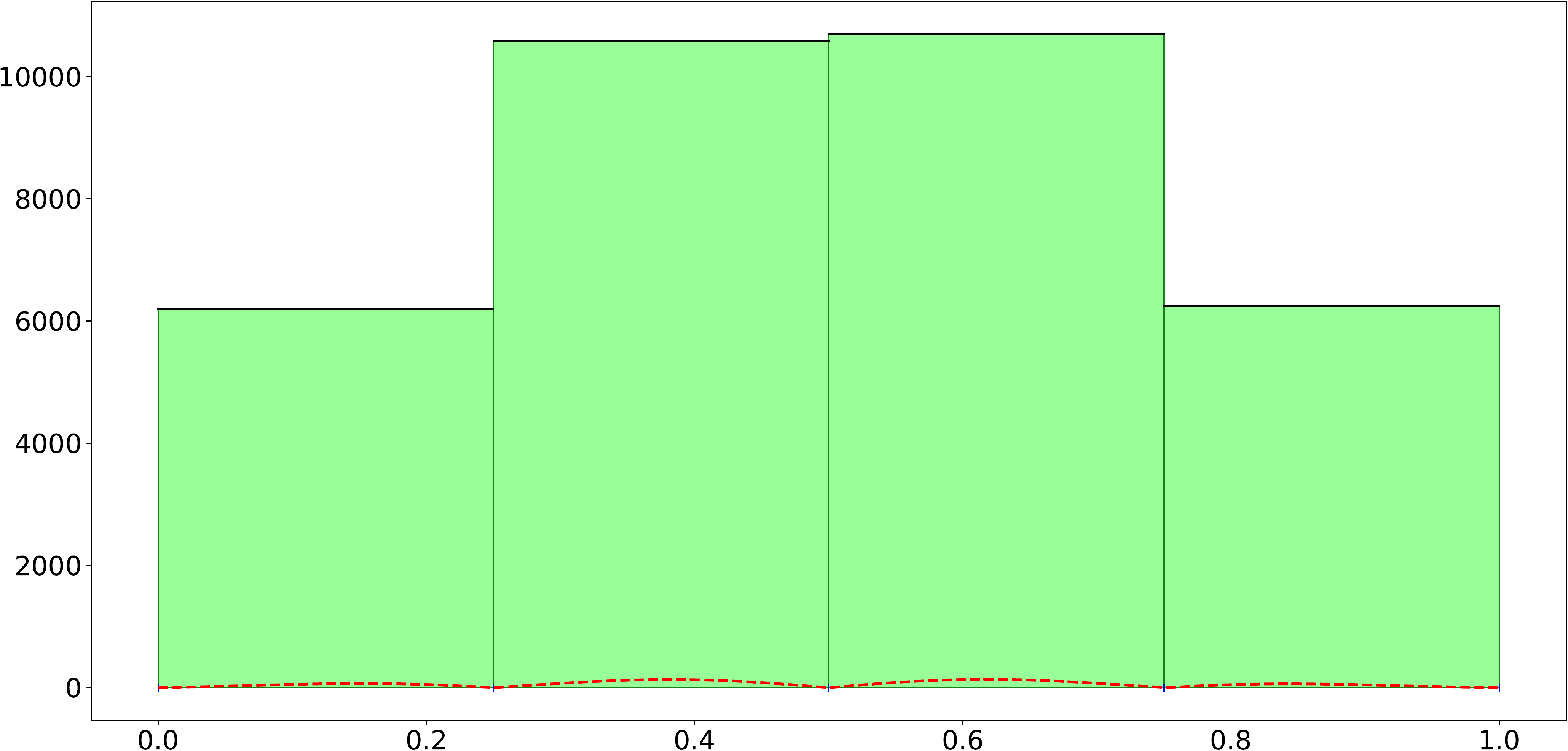}
      \caption{$|\calV_2|=5$}
    \end{subfigure}
    \begin{subfigure}[b]{0.3\textwidth}
      \includegraphics[width=\textwidth]{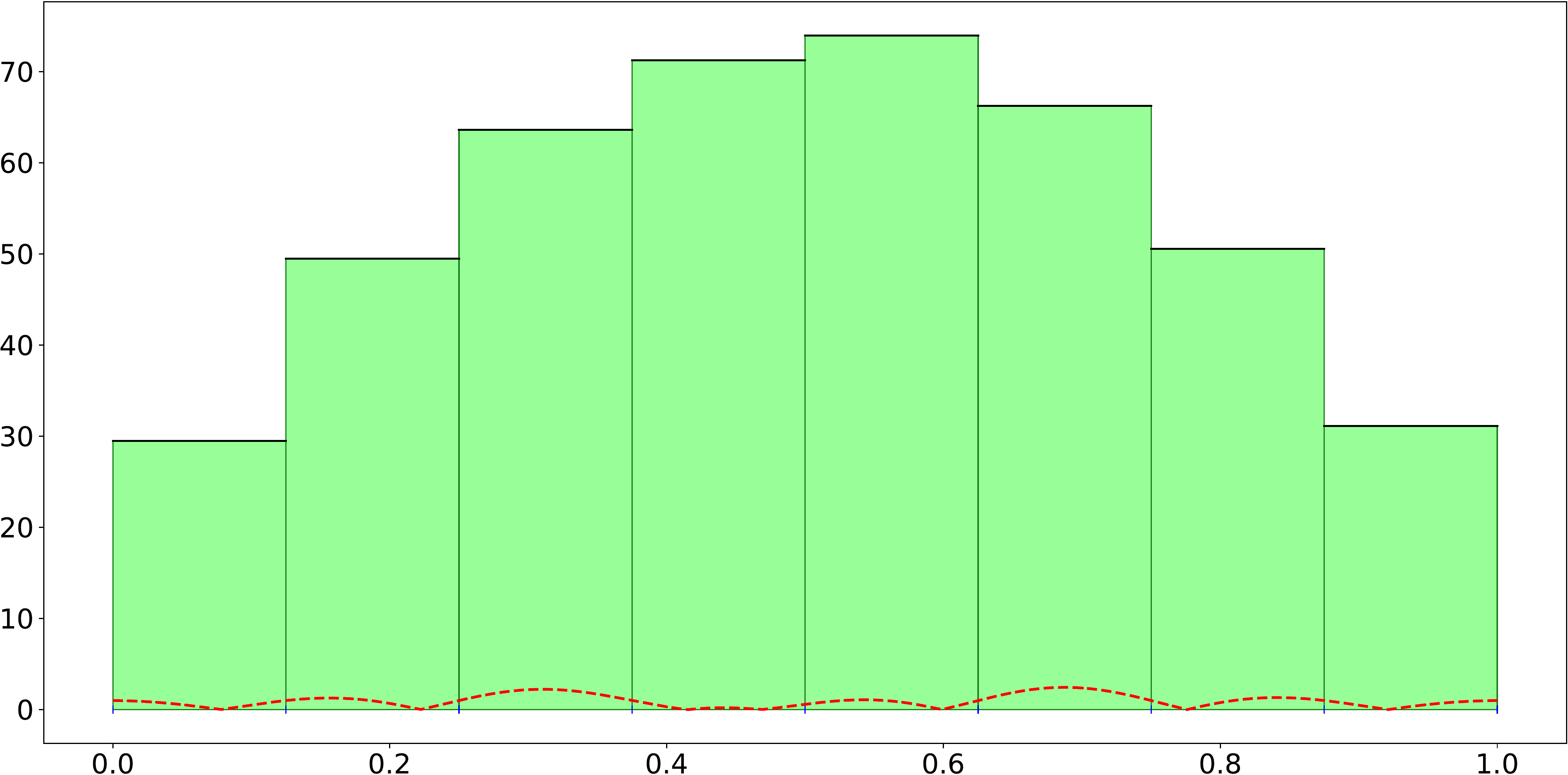}
      \caption{$|\calV_3|=9$}
    \end{subfigure}
    \begin{subfigure}[b]{0.3\textwidth}
      \includegraphics[width=\textwidth]{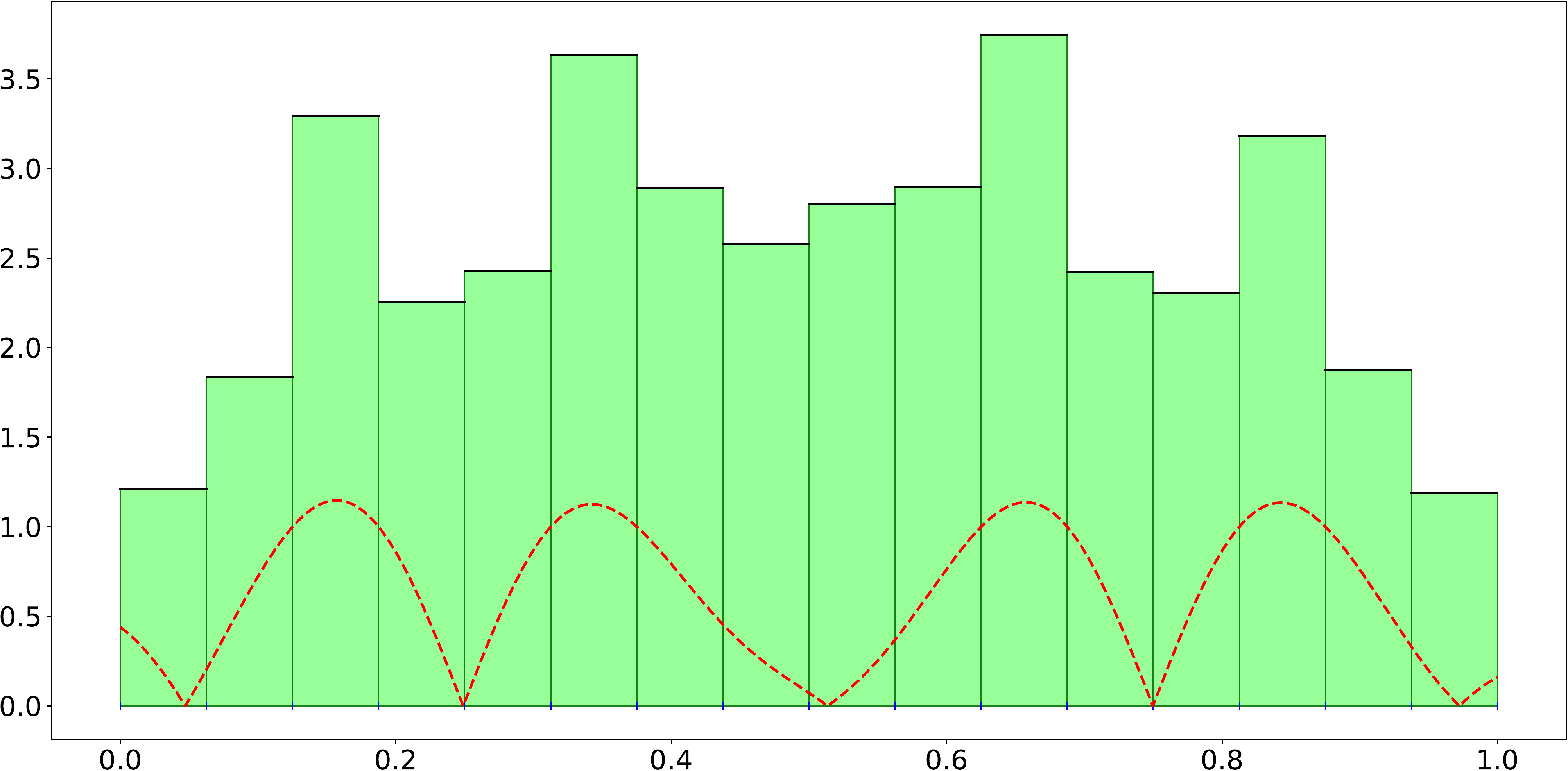}
      \caption{$|\calV_4|=17$}
    \end{subfigure}
    \begin{subfigure}[b]{0.3\textwidth}
      \includegraphics[width=\textwidth]{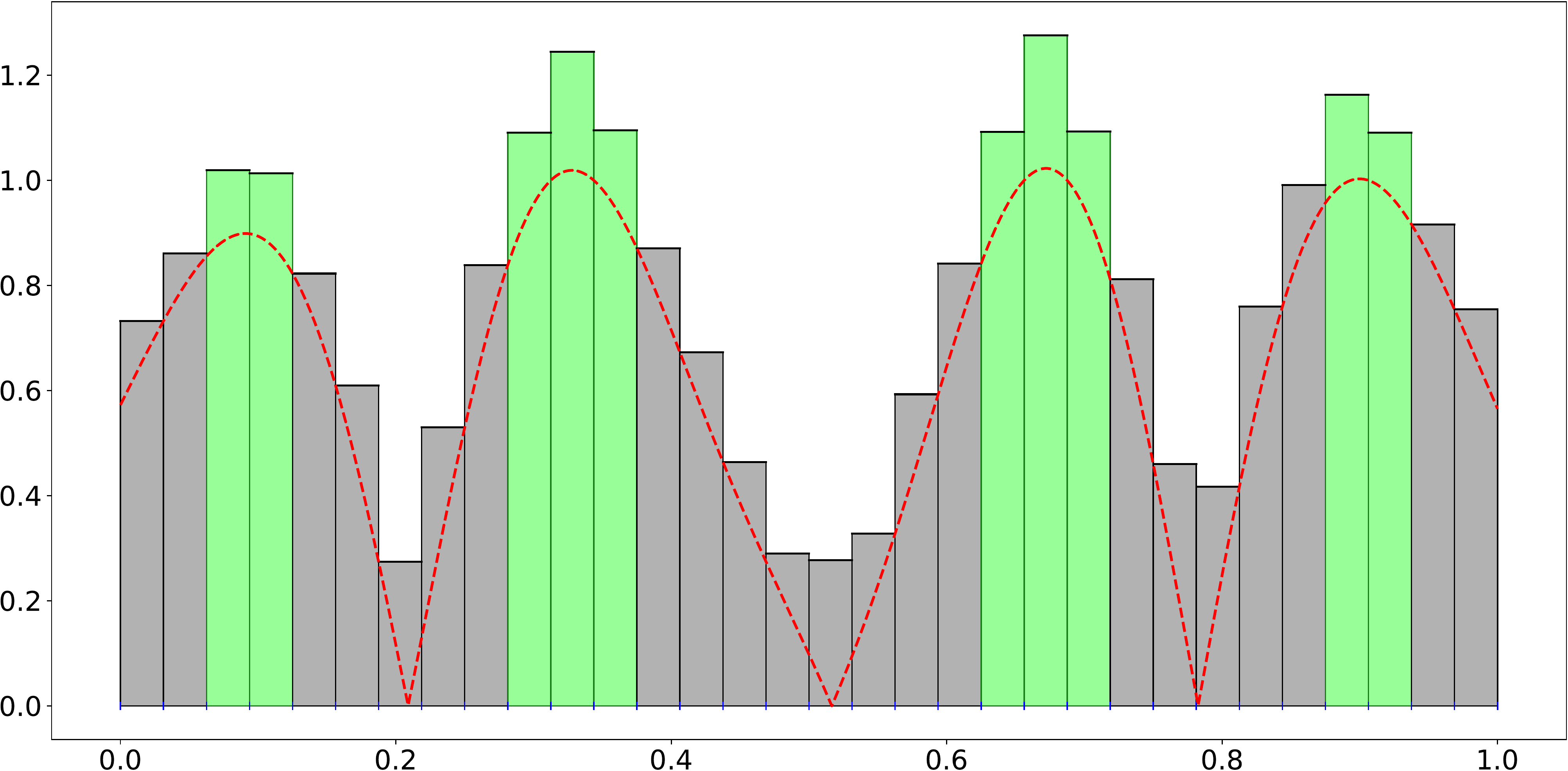}
      \caption{$|\calV_5|=33$}
    \end{subfigure}
    \begin{subfigure}[b]{0.3\textwidth}
      \includegraphics[width=\textwidth]{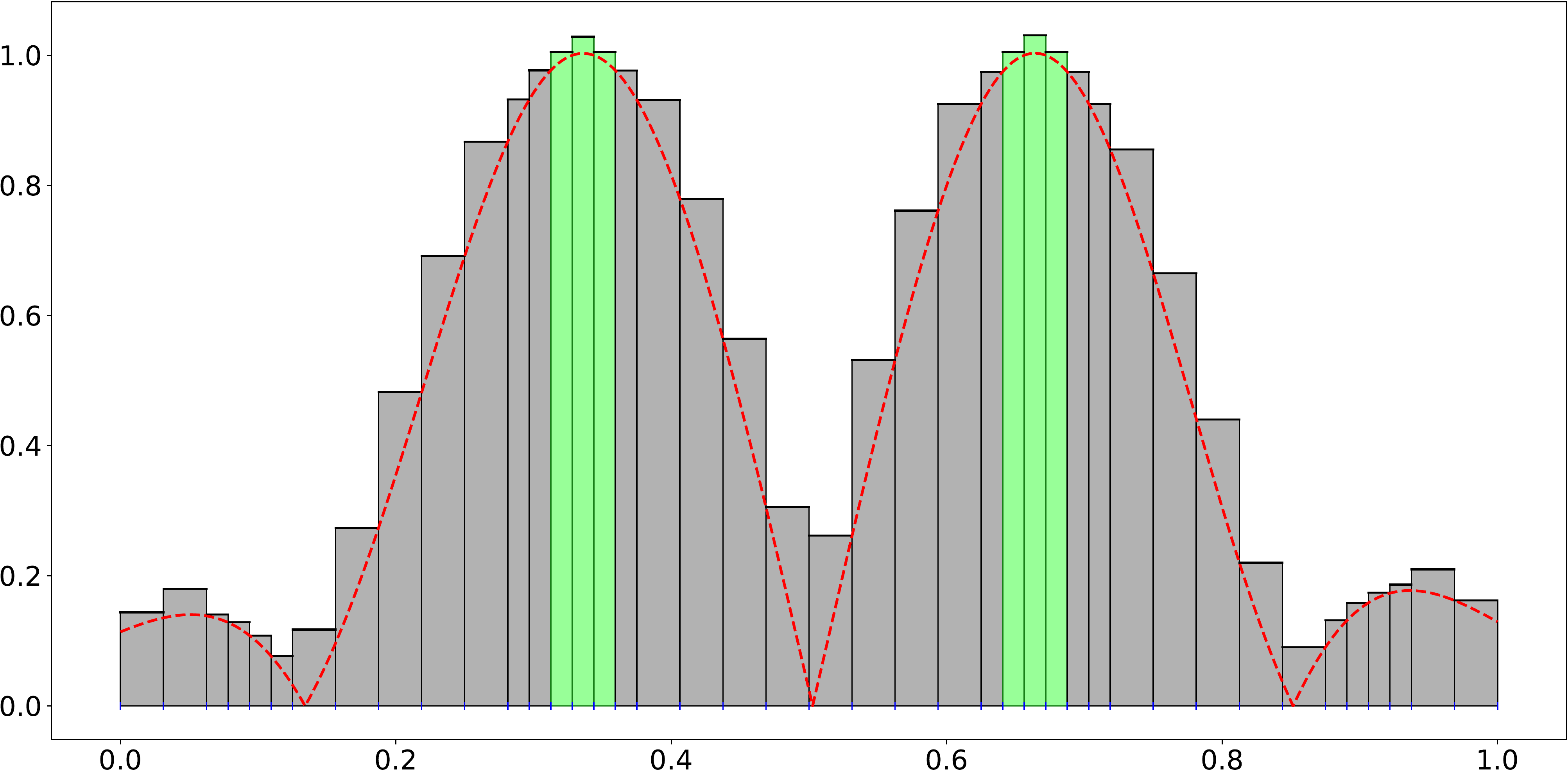}
      \caption{$|\calV_6|=43$\label{fig:behavior1D_nograd6}}
    \end{subfigure}
    \begin{subfigure}[b]{0.3\textwidth}
      \includegraphics[width=\textwidth]{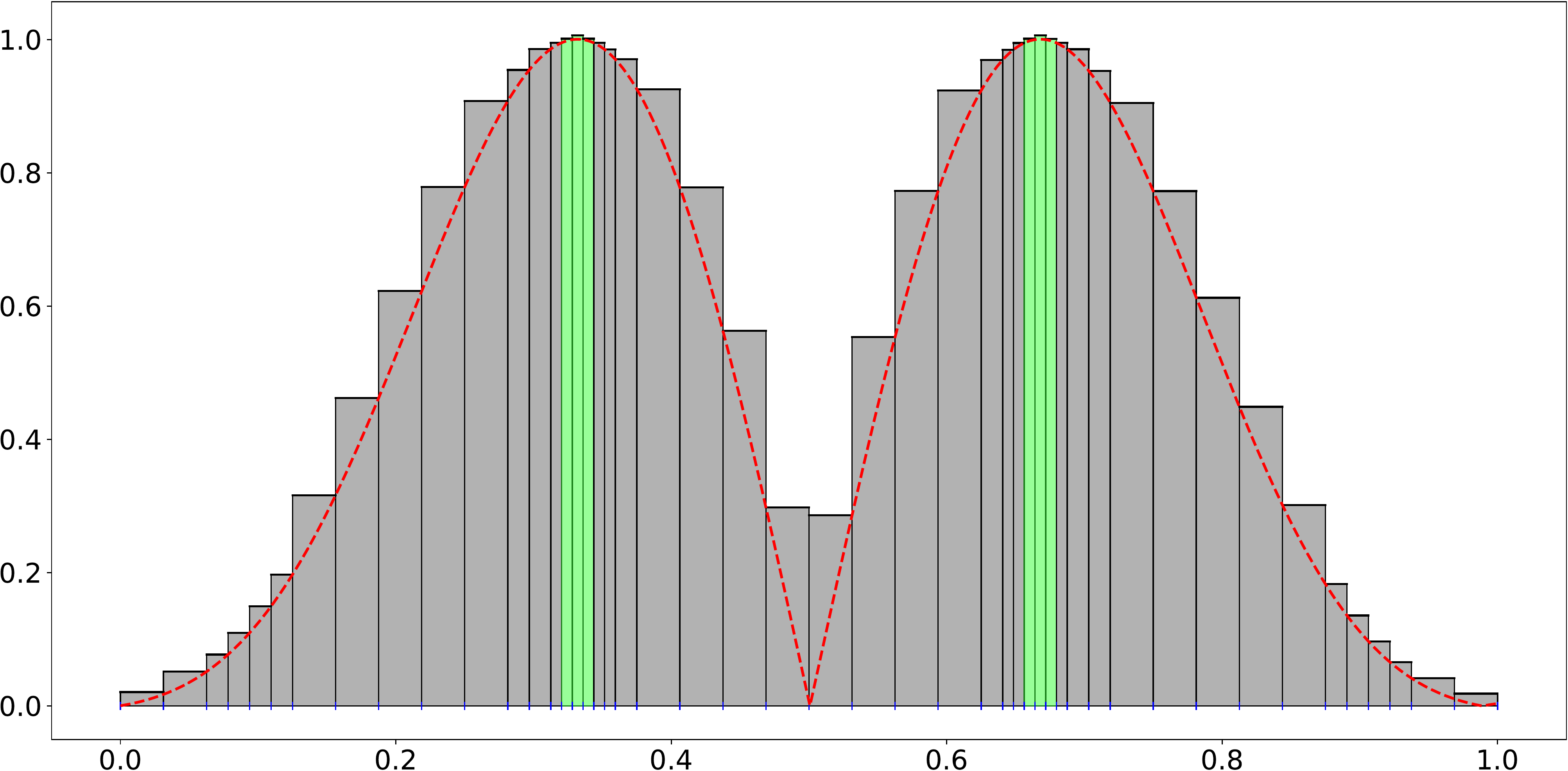}
      \caption{$|\calV_7|=49$}
    \end{subfigure}
    \begin{subfigure}[b]{0.3\textwidth}
      \includegraphics[width=\textwidth]{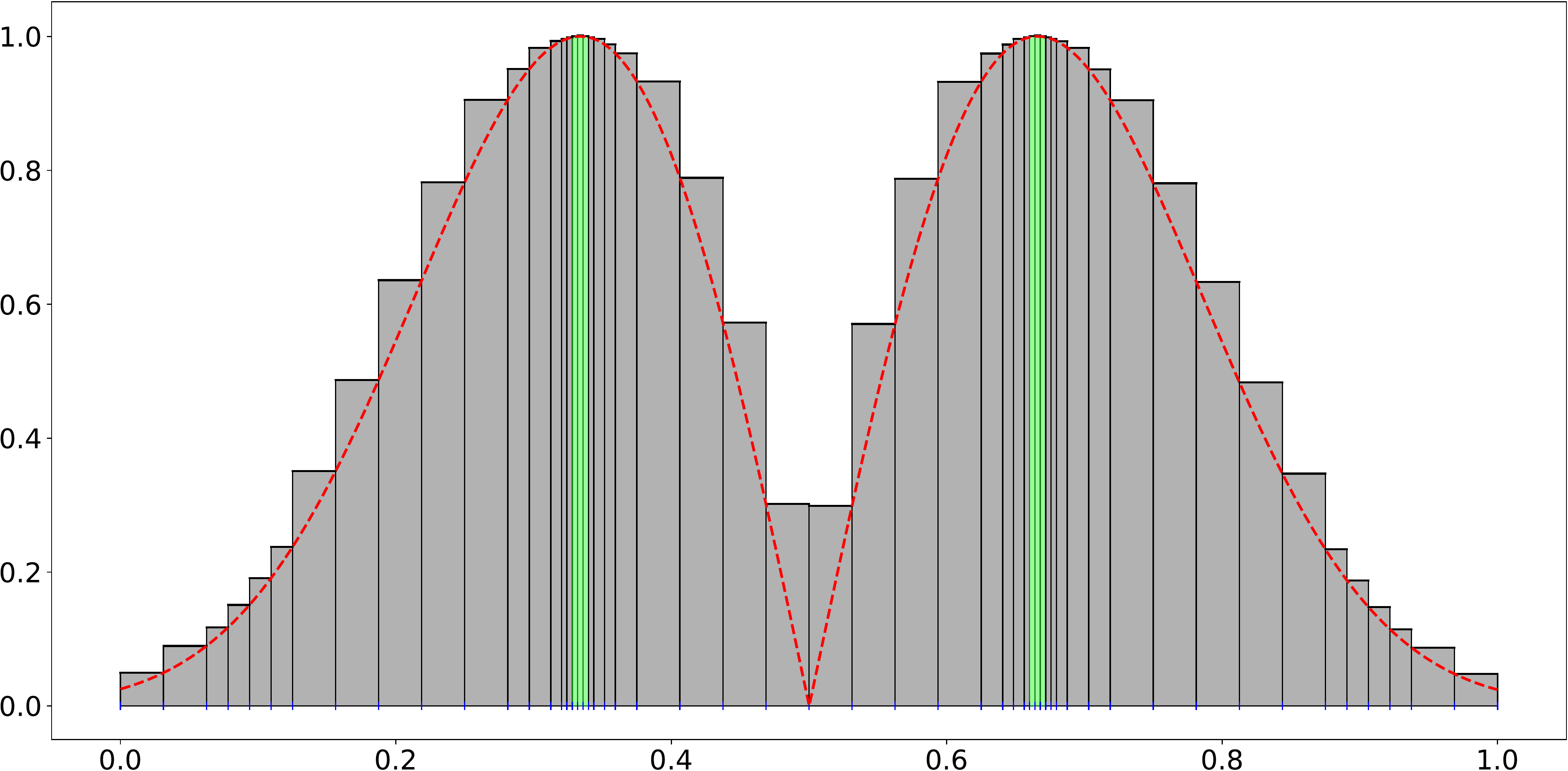}
      \caption{$|\calV_8|=55$}
    \end{subfigure}
    \begin{subfigure}[b]{0.3\textwidth}
      \includegraphics[width=\textwidth]{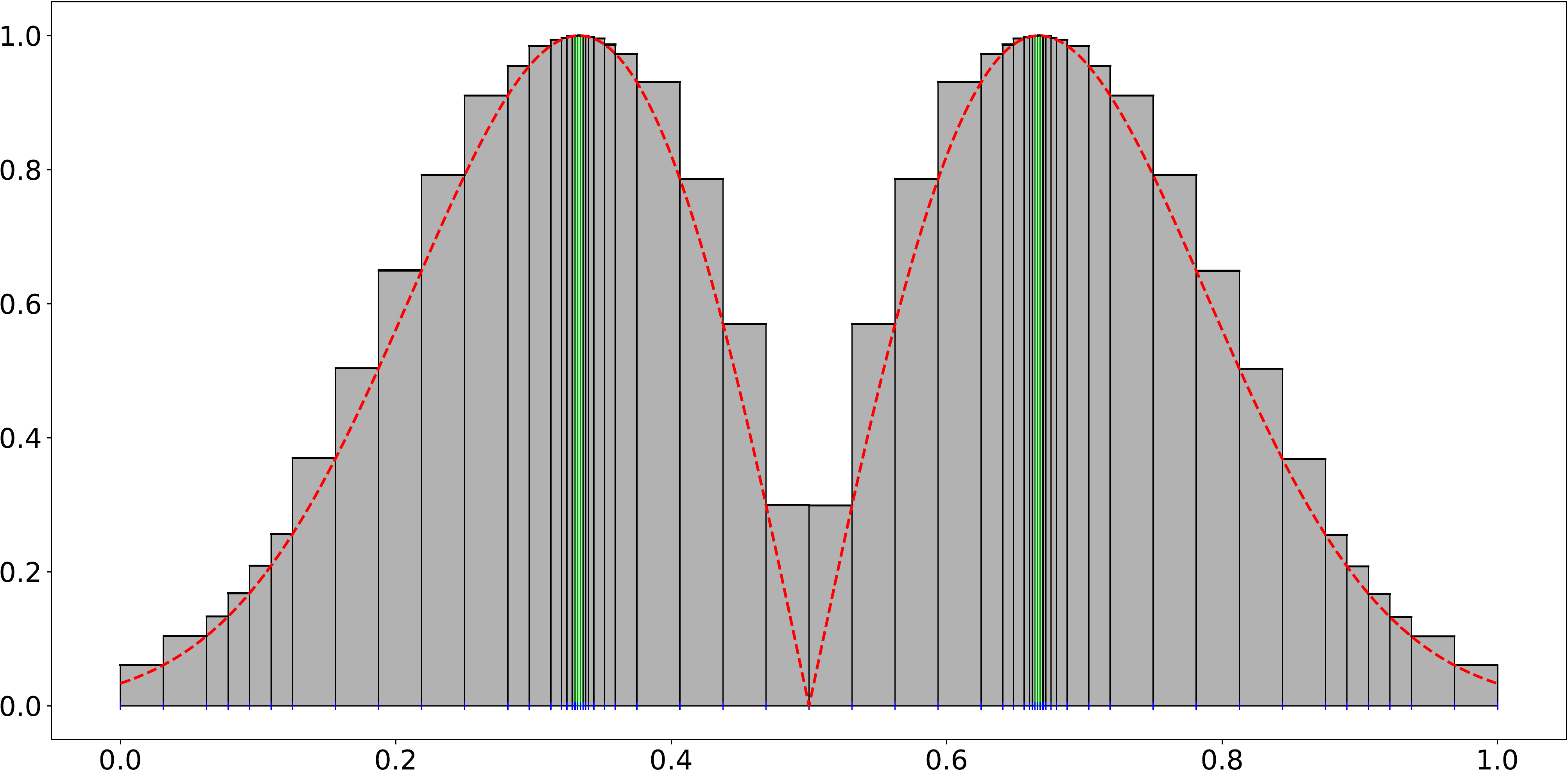}
      \caption{$|\calV_9|=61$}
    \end{subfigure}
    \begin{subfigure}[b]{0.3\textwidth}
      \includegraphics[width=\textwidth]{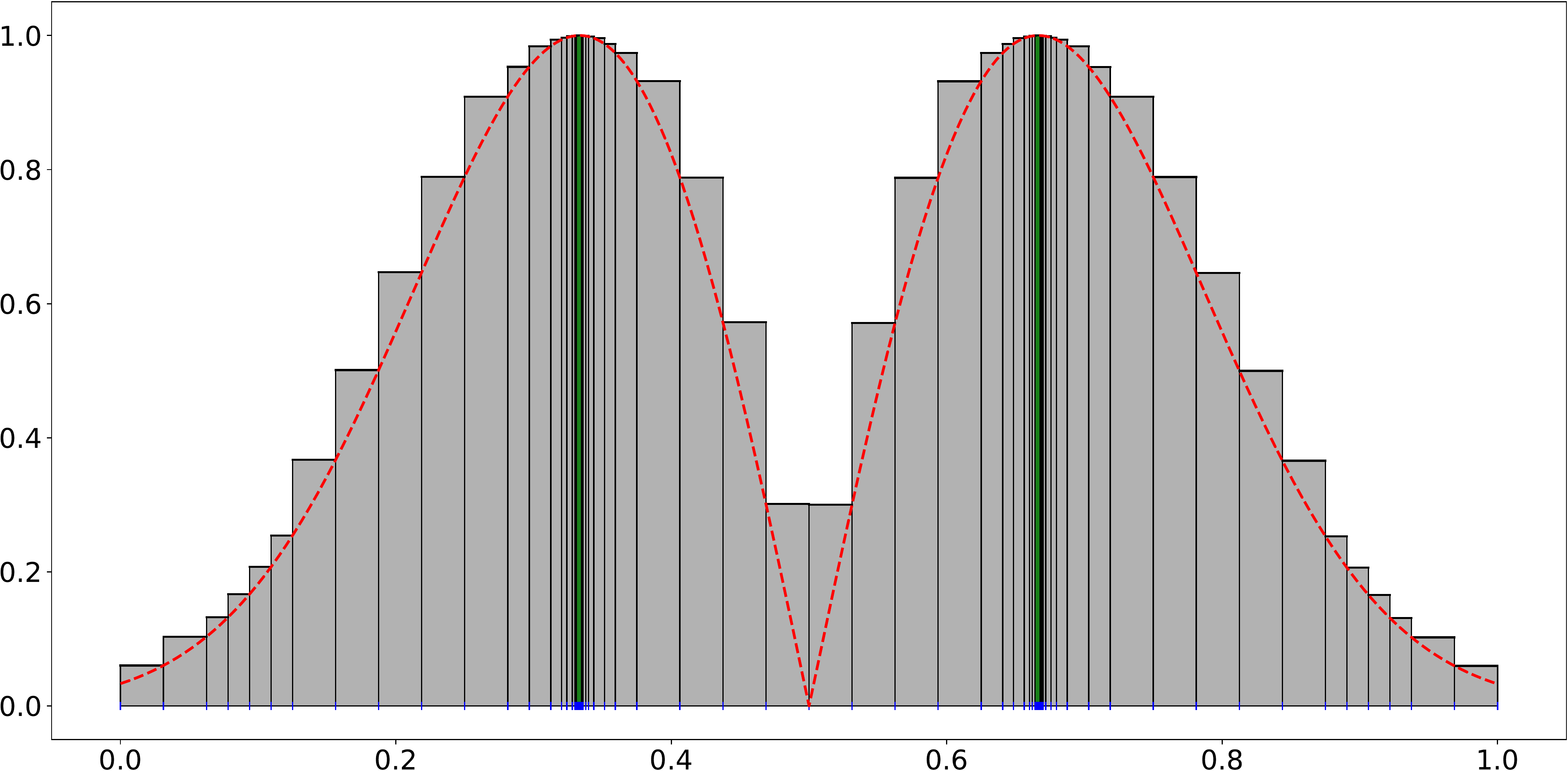}
      \caption{$|\calV_{10}|=67$}
    \end{subfigure}
    \begin{subfigure}[b]{0.3\textwidth}
      \includegraphics[width=\textwidth]{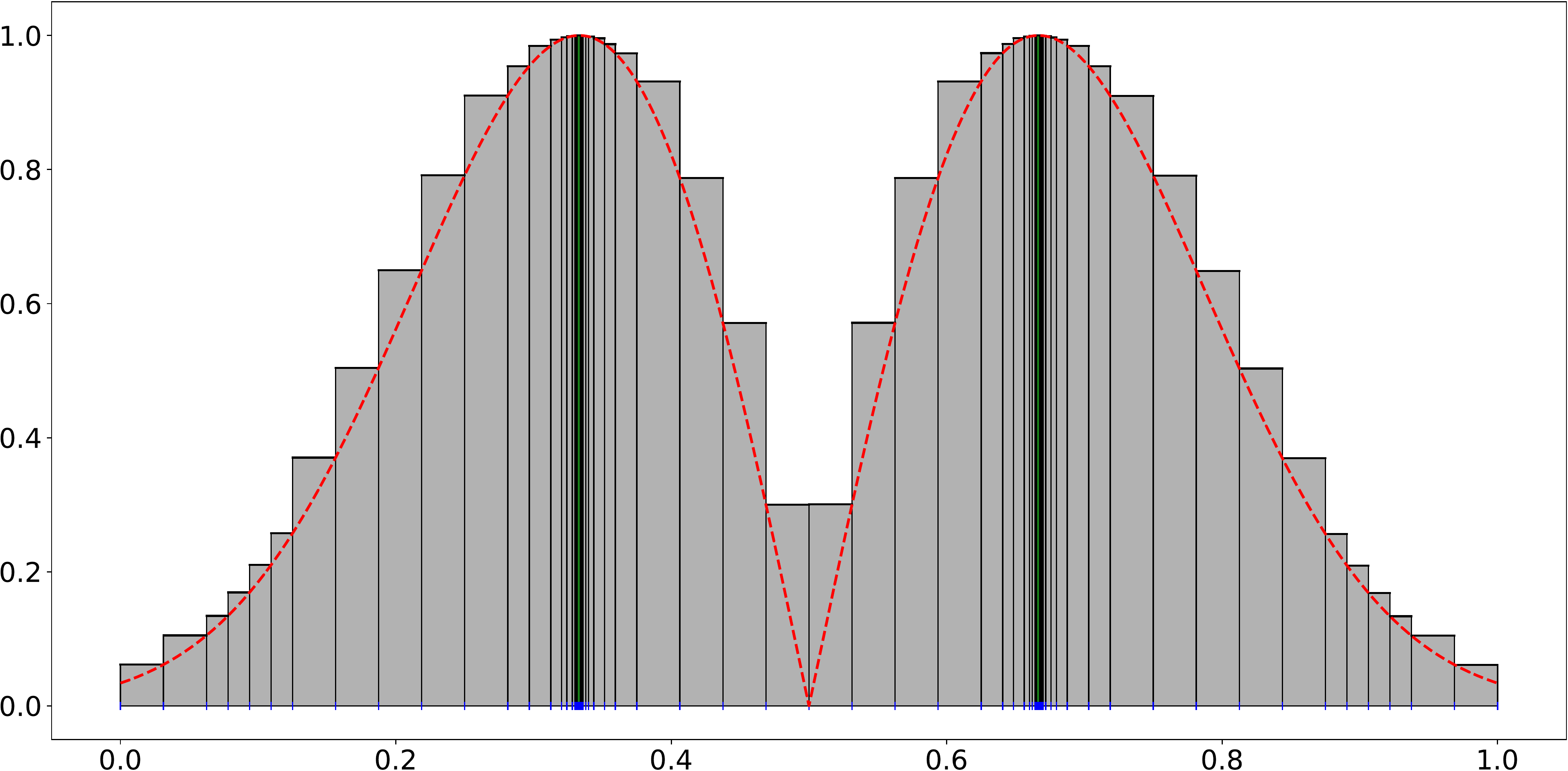}
      \caption{$|\calV_{11}|=73$}
    \end{subfigure}
    \begin{subfigure}[b]{0.3\textwidth}
      \includegraphics[width=\textwidth]{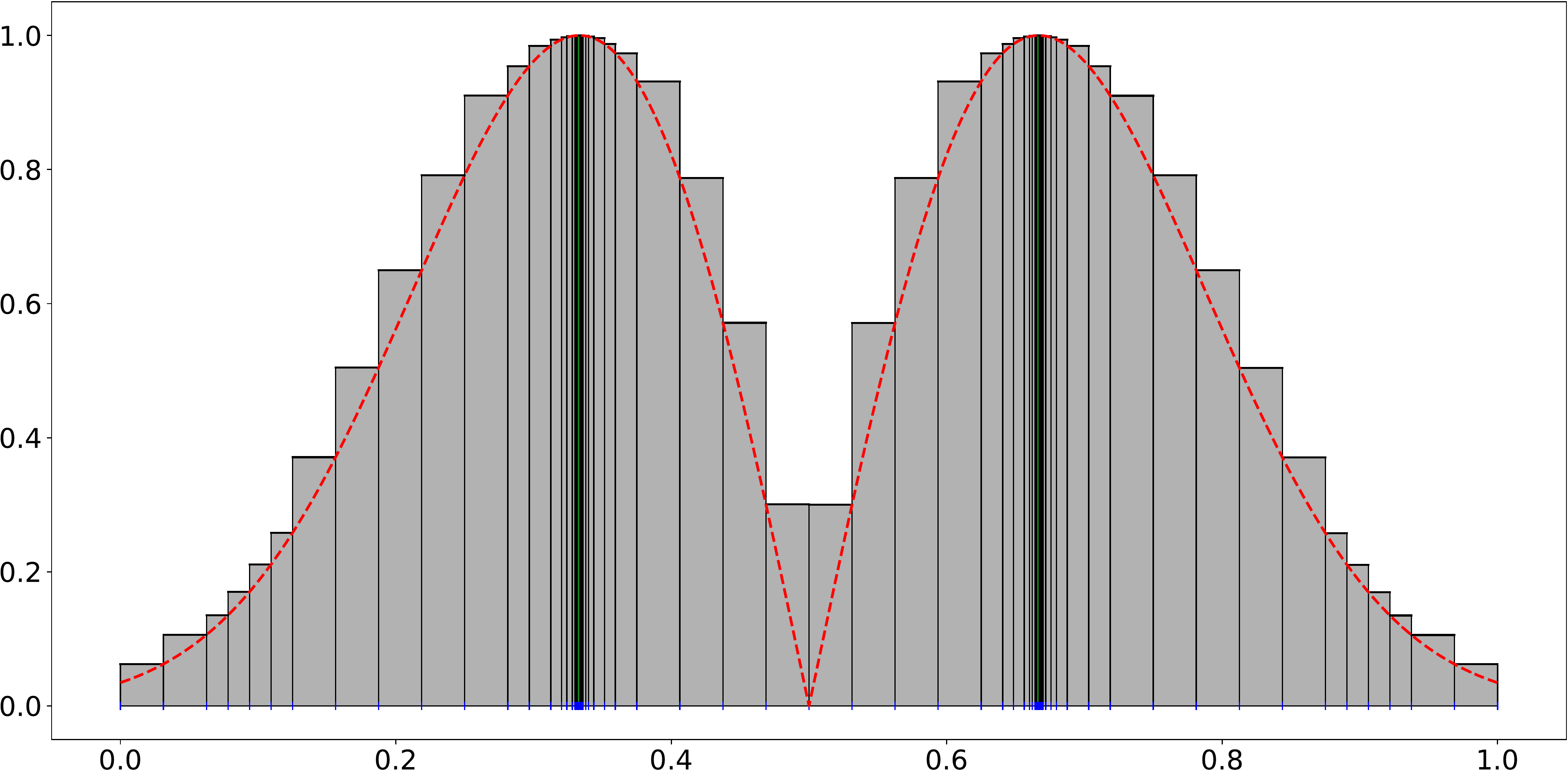}
      \caption{$|\calV_{12}|=79$}
    \end{subfigure}
    \caption{The behavior of the adaptive refinement algorithm with a second order selection process, on a 1D sparse recovery problem. 
    The set $\Omega_k^\star$ is displayed in green, the function $|A^*q_k|$ is displayed in dashed red, the upper-bound $\Uk$ is the piecewise-constant function. Observe that it always dominates $|A^*q_k|$.
    The algorithm starts with a burn-in period of 4 iterations. There, it refines all cells uniformly since the upper-bound is highly inaccurate. After a while, only the cells around the locations $X^\star$ get refined in a multiscale fashion.\label{fig:behavior1D_nograd}}
  \end{figure}

\begin{figure}[!t]
    \centering
    \begin{subfigure}[b]{0.3\textwidth}
      \includegraphics[width=\textwidth]{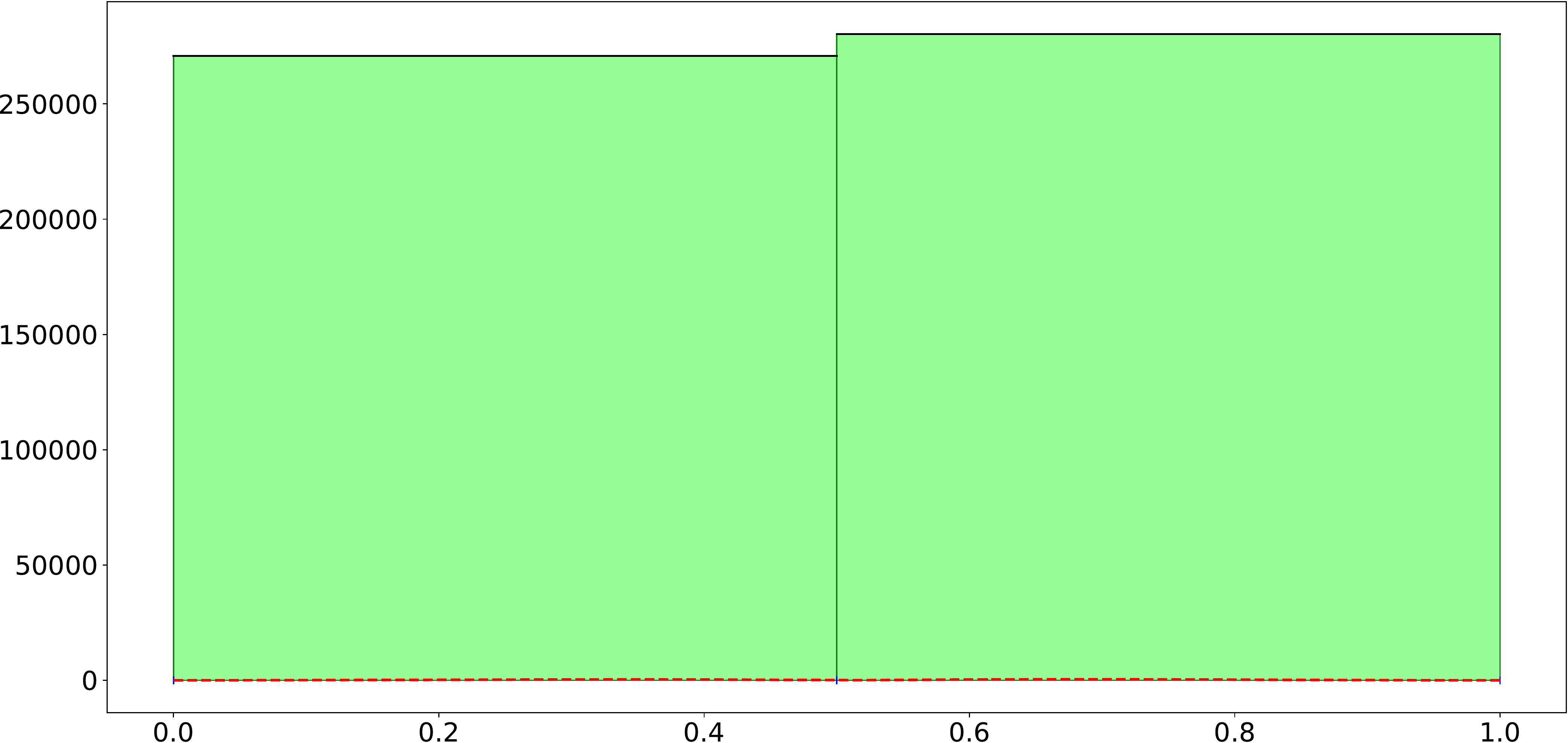}
      \caption{$|\calV_1|=3$}
    \end{subfigure}
    \begin{subfigure}[b]{0.3\textwidth}
      \includegraphics[width=\textwidth]{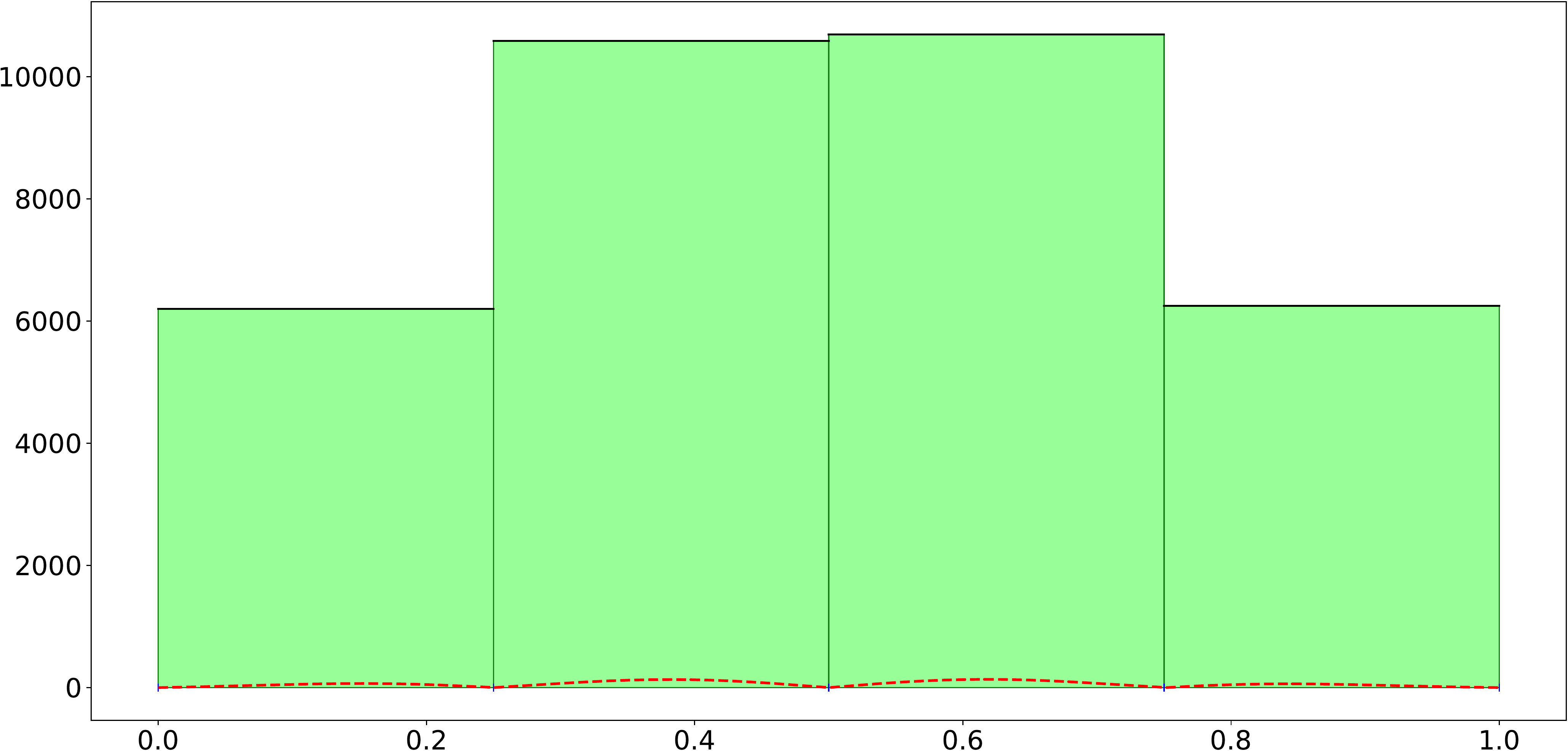}
      \caption{$|\calV_2|=5$}
    \end{subfigure}
    \begin{subfigure}[b]{0.3\textwidth}
      \includegraphics[width=\textwidth]{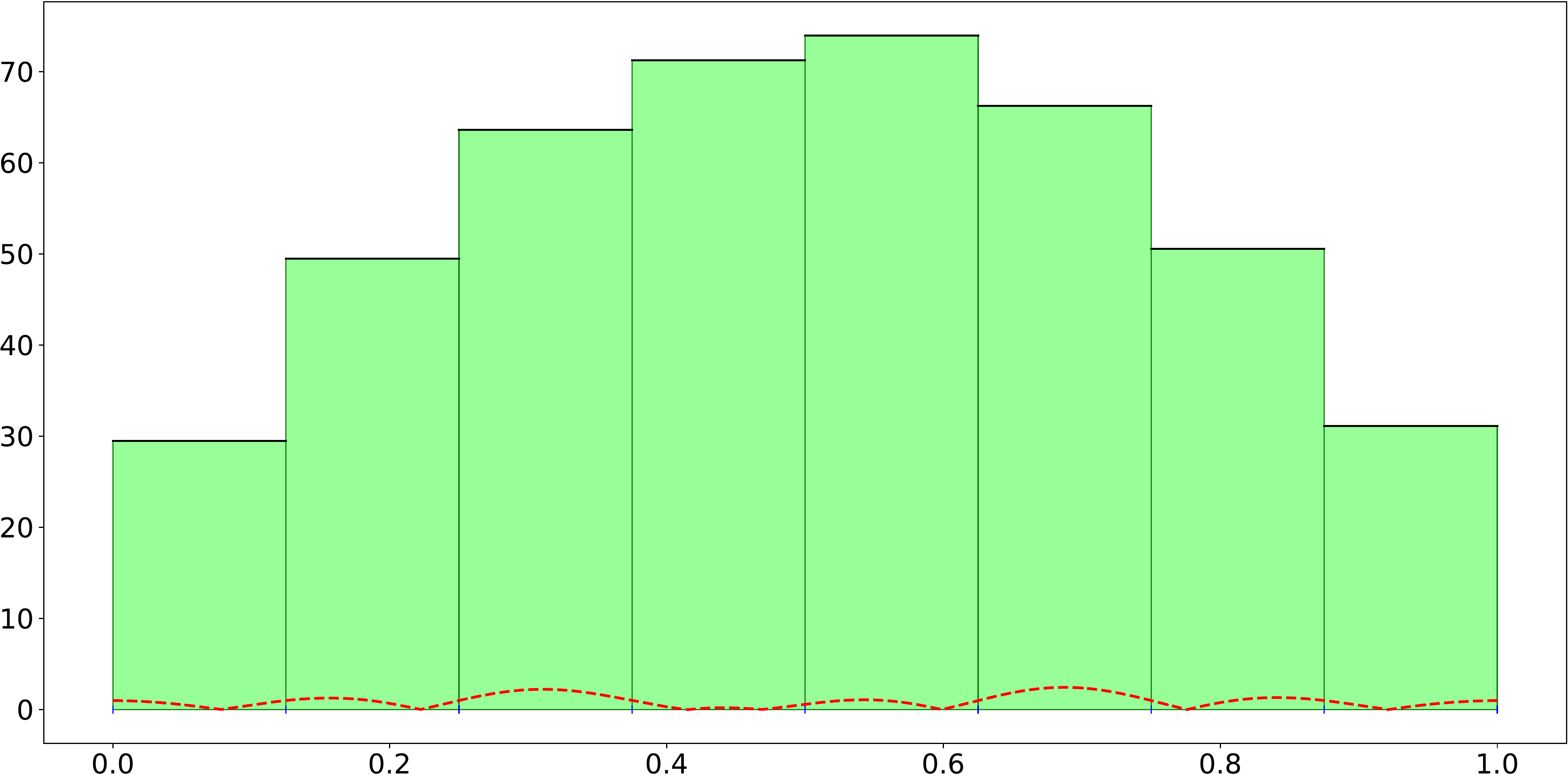}
      \caption{$|\calV_3|=9$}
    \end{subfigure}
    \begin{subfigure}[b]{0.3\textwidth}
      \includegraphics[width=\textwidth]{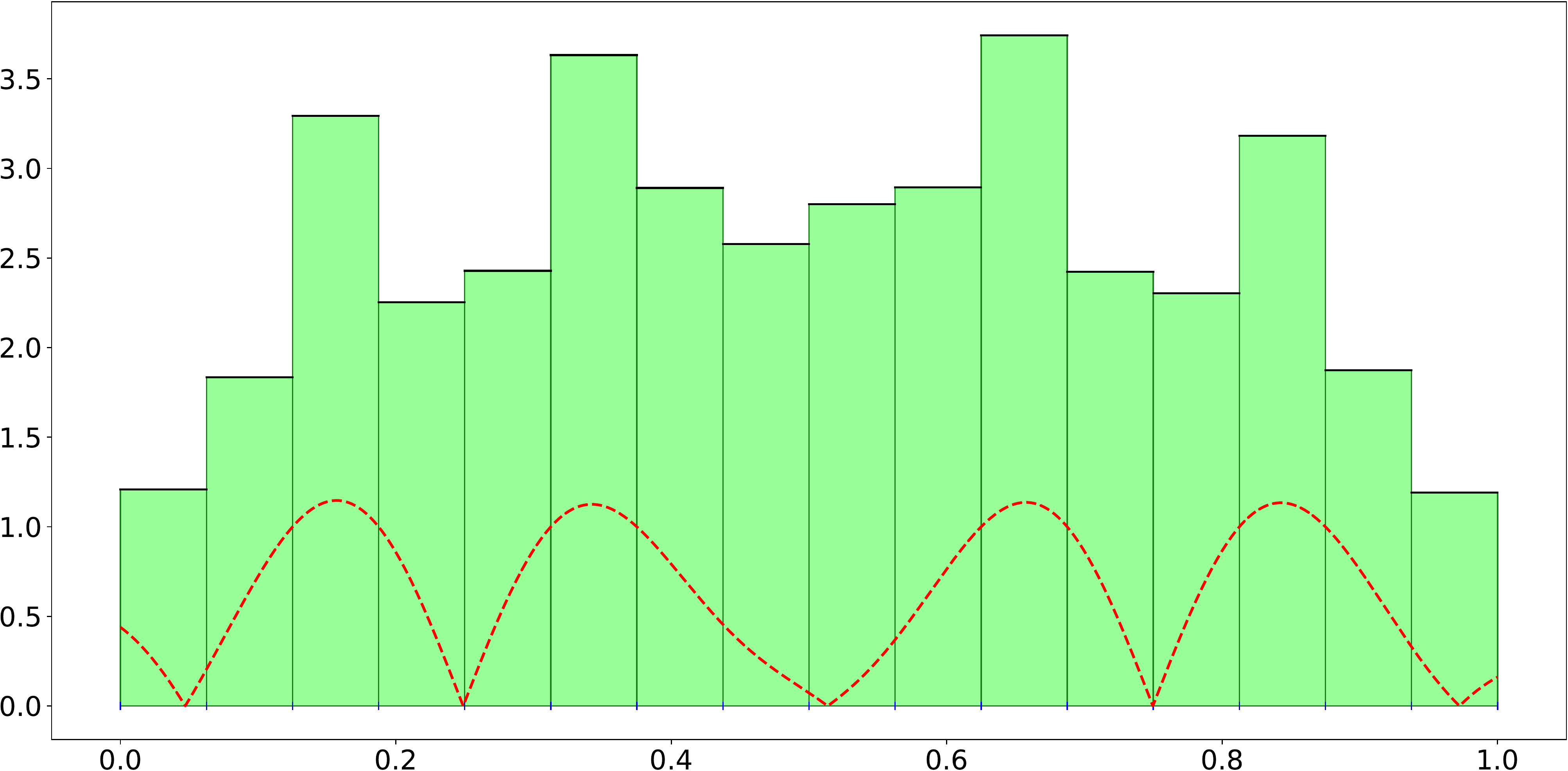}
      \caption{$|\calV_4|=17$}
    \end{subfigure}
    \begin{subfigure}[b]{0.3\textwidth}
      \includegraphics[width=\textwidth]{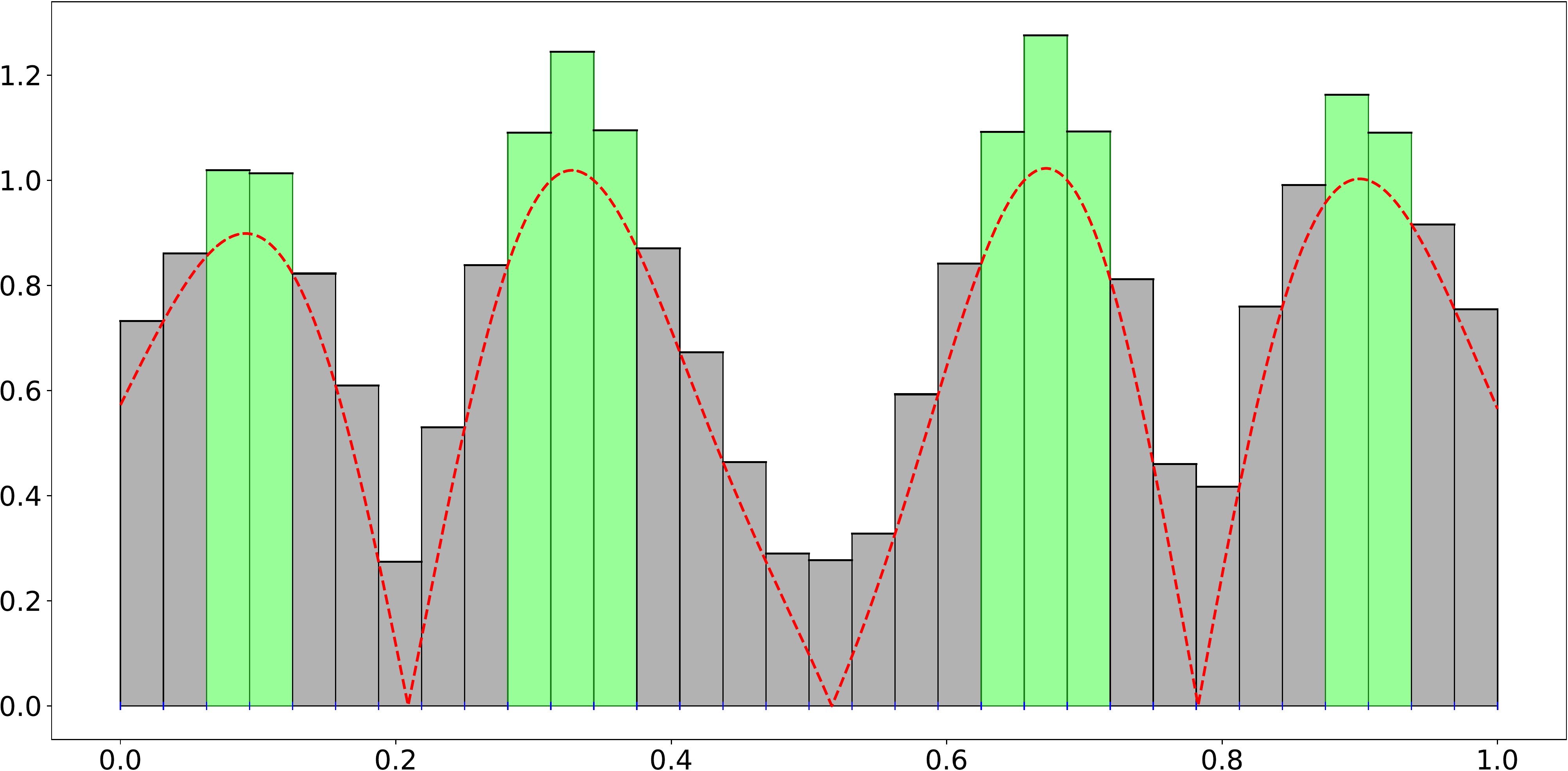}
      \caption{$|\calV_5|=33$}
    \end{subfigure}
    \begin{subfigure}[b]{0.3\textwidth}
      \includegraphics[width=\textwidth]{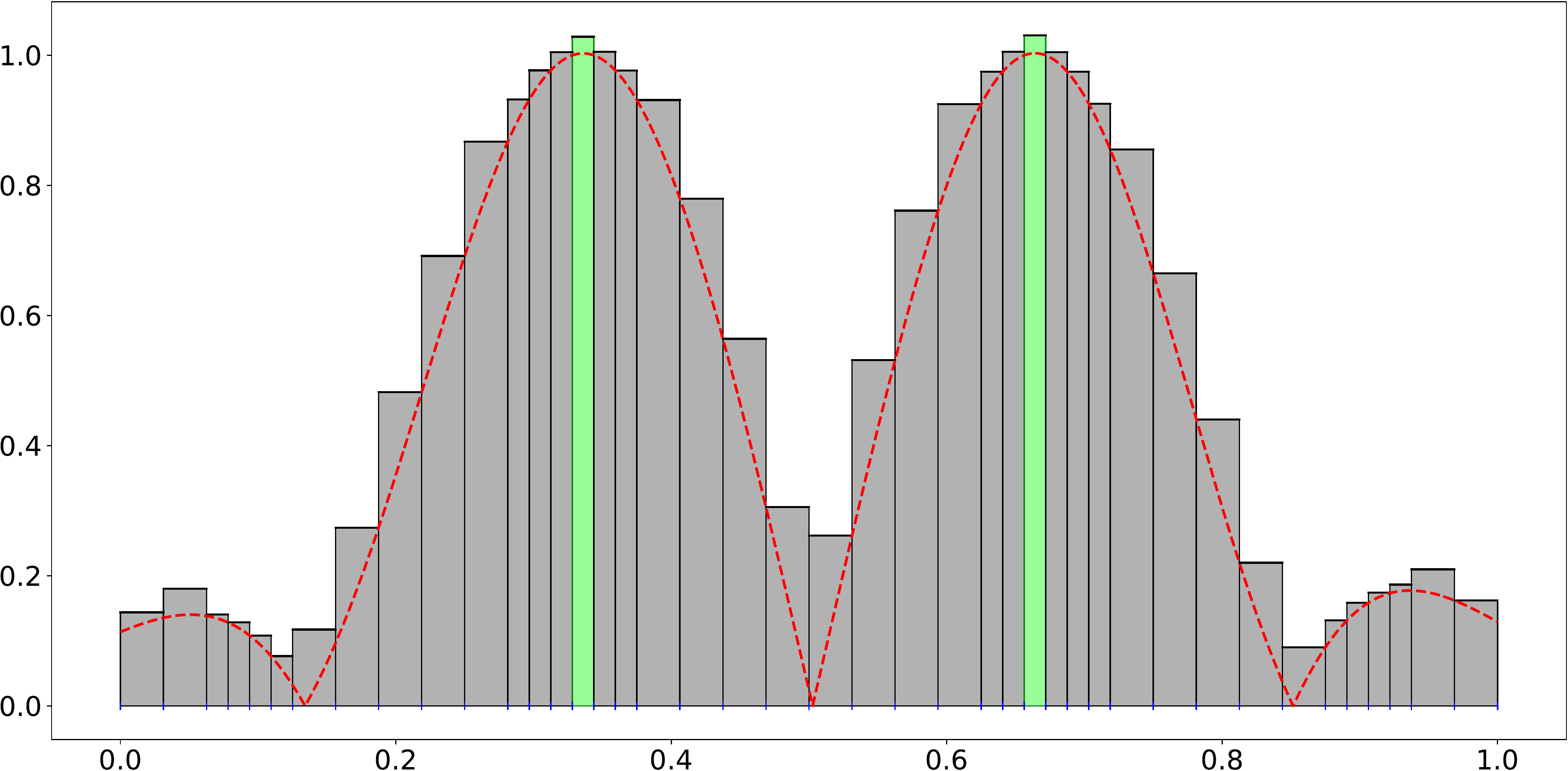}
      \caption{$|\calV_6|=43$\label{fig:behavior1D_grad6}}
    \end{subfigure}
    \begin{subfigure}[b]{0.3\textwidth}
      \includegraphics[width=\textwidth]{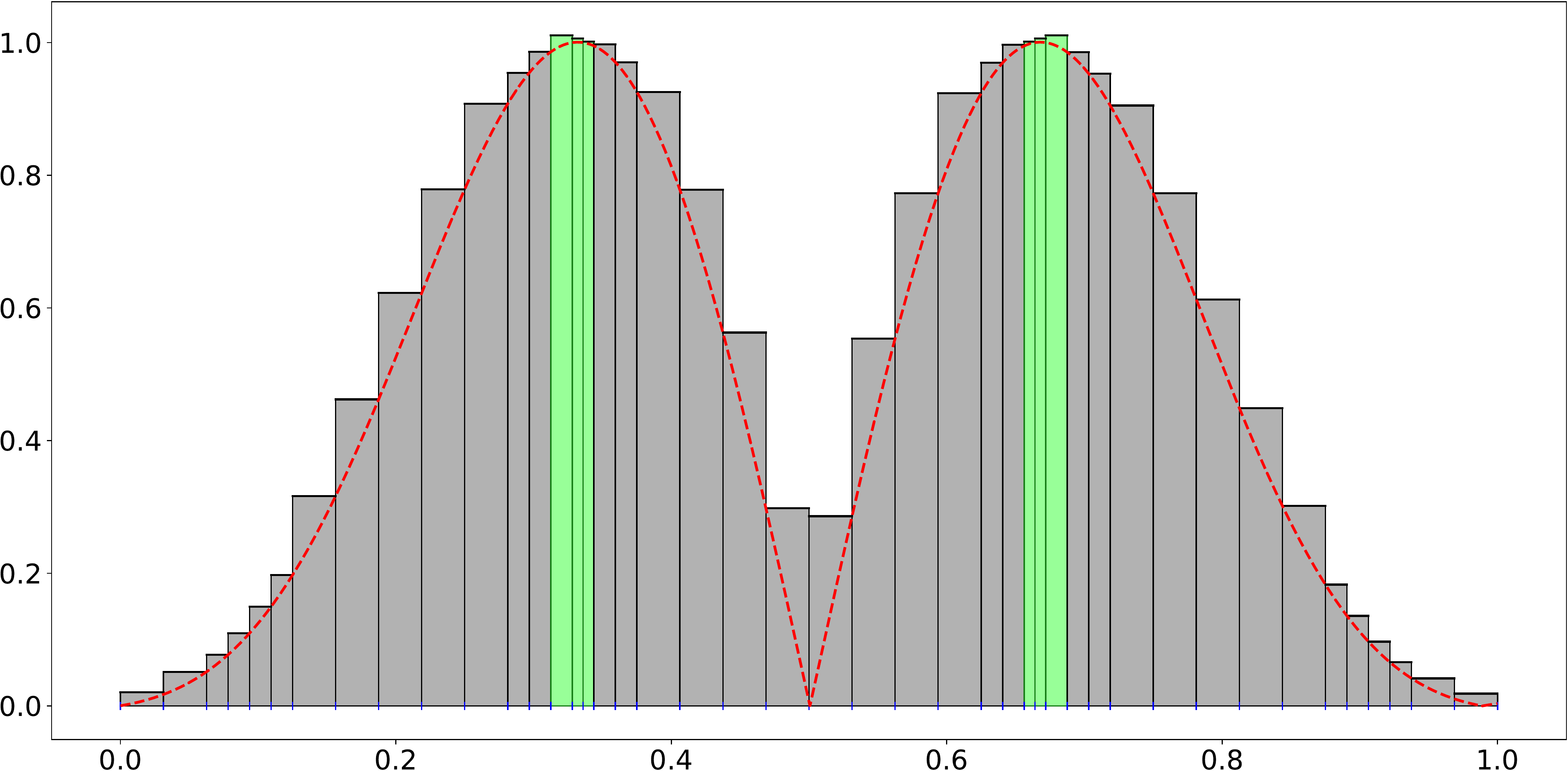}
      \caption{$|\calV_7|=45$\label{fig:behavior1D_grad7}}
    \end{subfigure}
    \begin{subfigure}[b]{0.3\textwidth}
      \includegraphics[width=\textwidth]{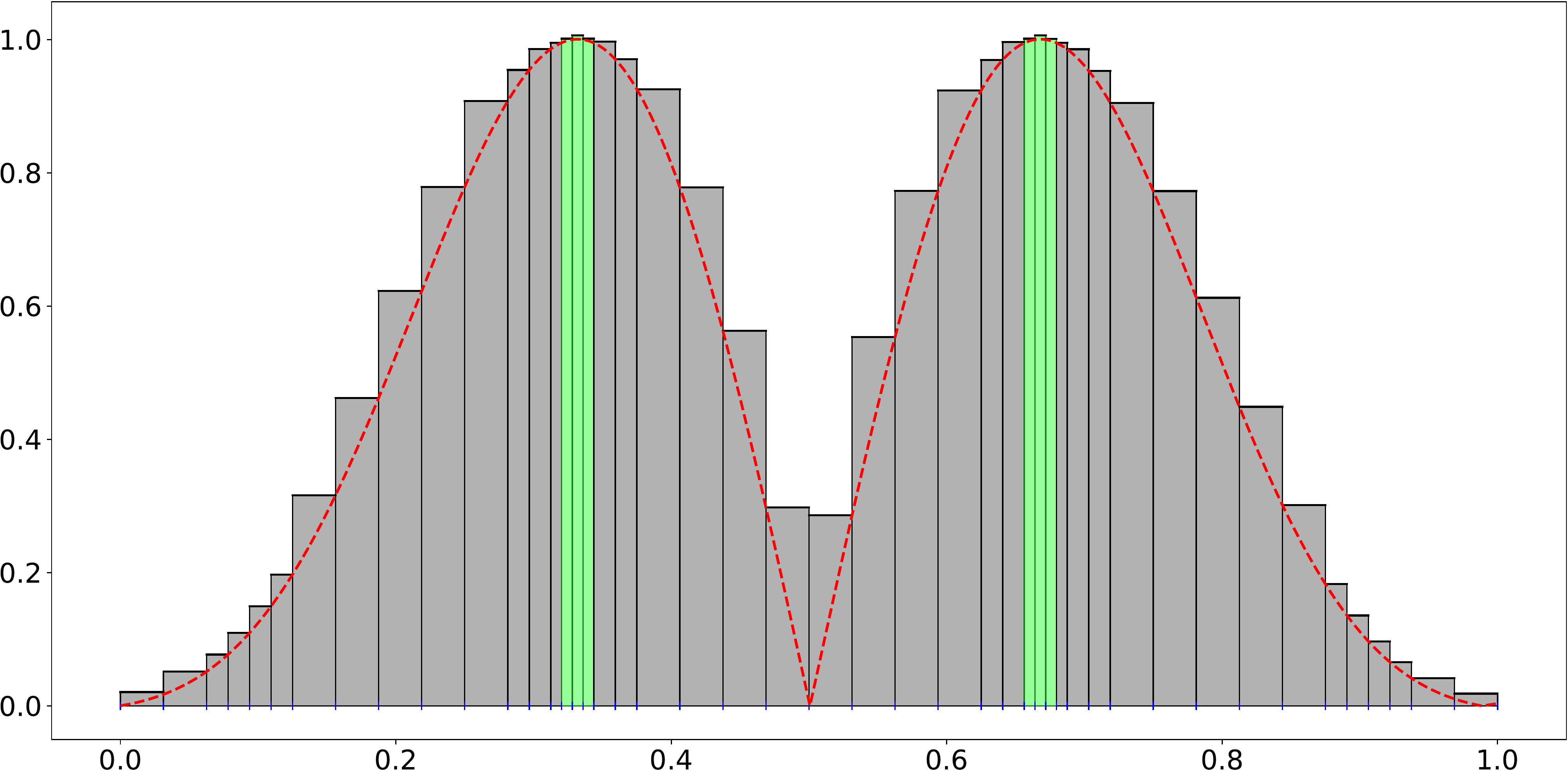}
      \caption{$|\calV_8|=47$}
    \end{subfigure}
    \begin{subfigure}[b]{0.3\textwidth}
      \includegraphics[width=\textwidth]{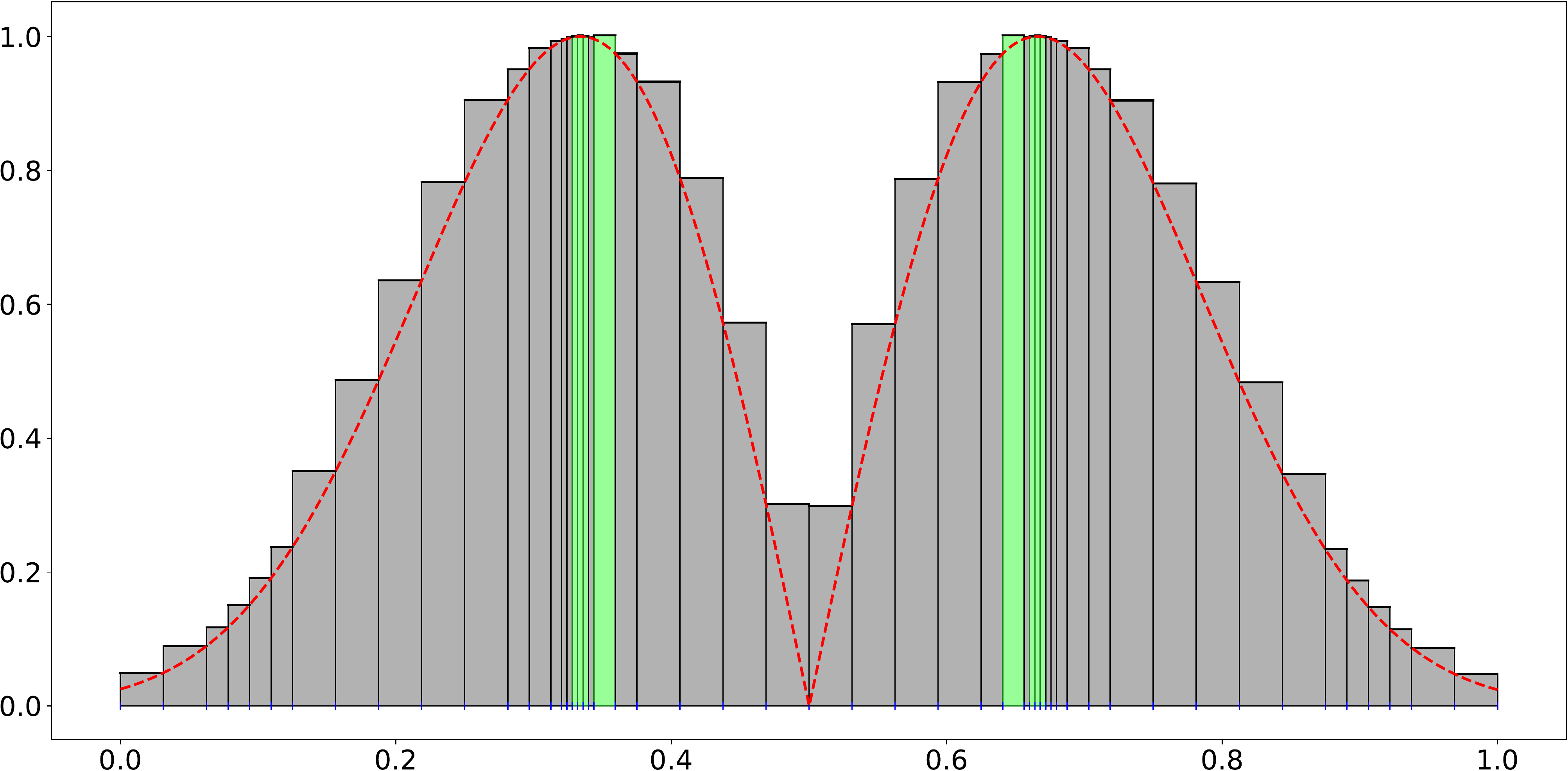}
      \caption{$|\calV_9|=53$}
    \end{subfigure}
    \begin{subfigure}[b]{0.3\textwidth}
      \includegraphics[width=\textwidth]{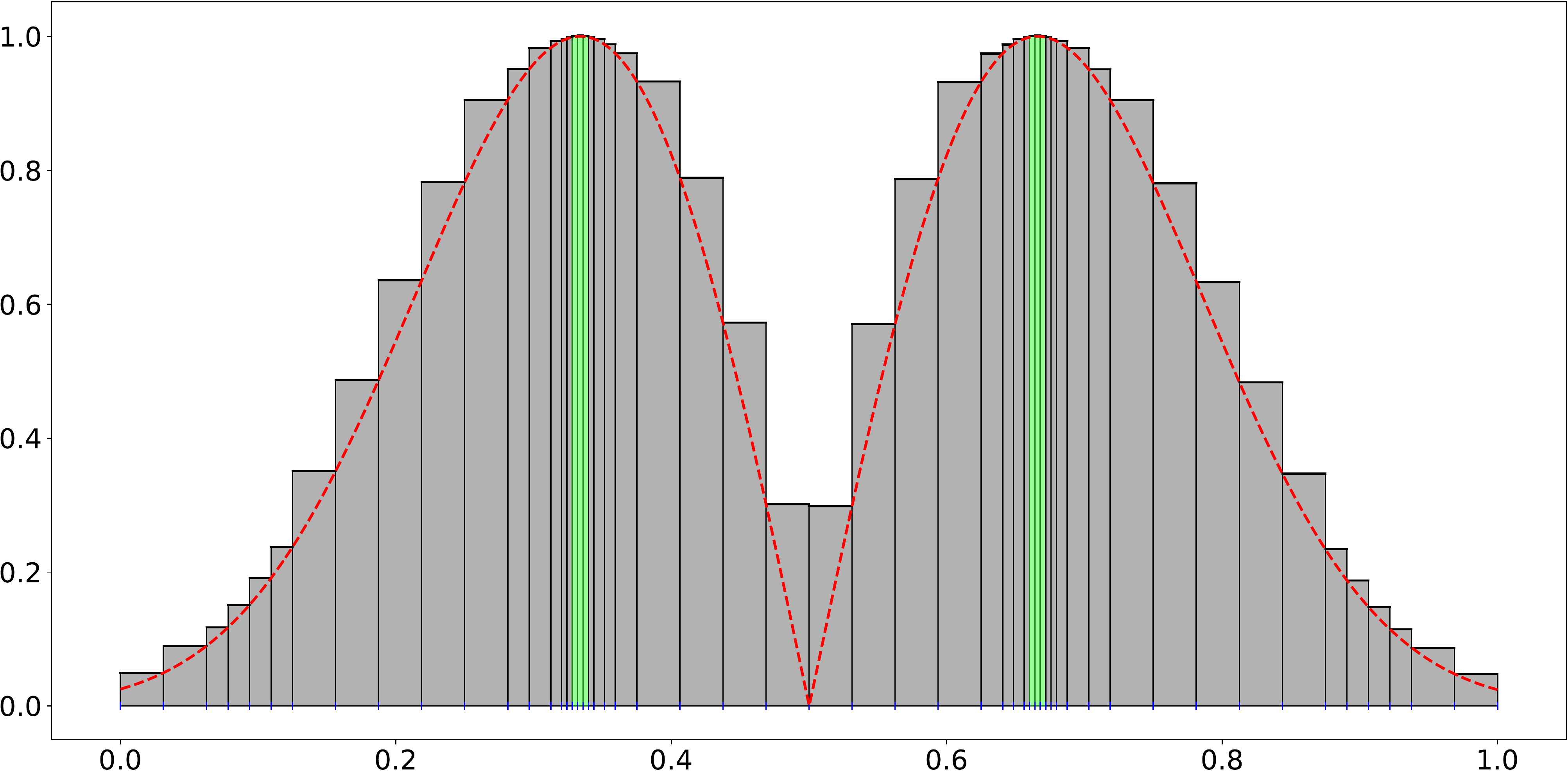}
      \caption{$|\calV_{10}|=55$}
    \end{subfigure}
    \begin{subfigure}[b]{0.3\textwidth}
      \includegraphics[width=\textwidth]{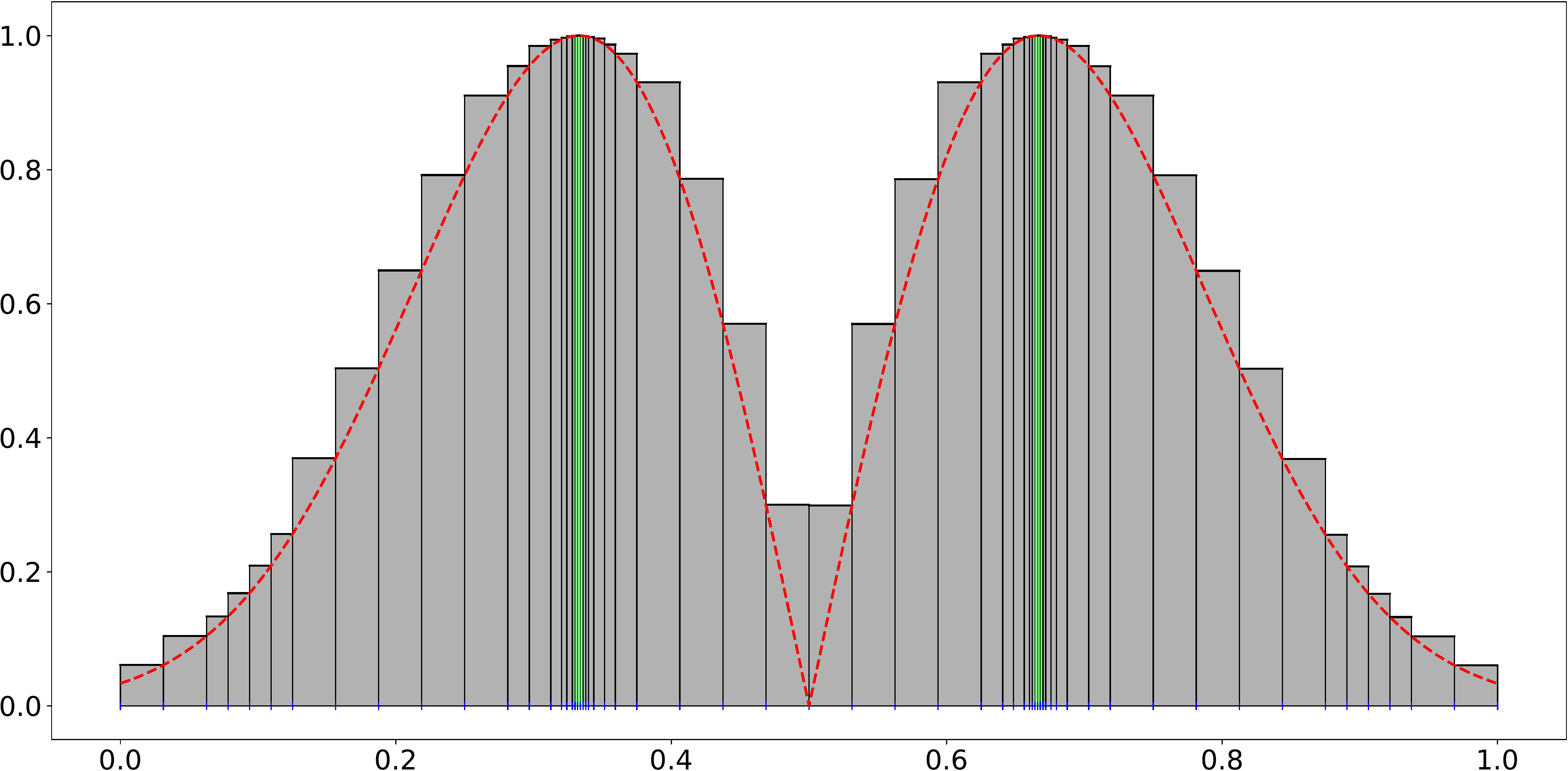}
      \caption{$|\calV_{11}|=61$}
    \end{subfigure}
    \begin{subfigure}[b]{0.3\textwidth}
      \includegraphics[width=\textwidth]{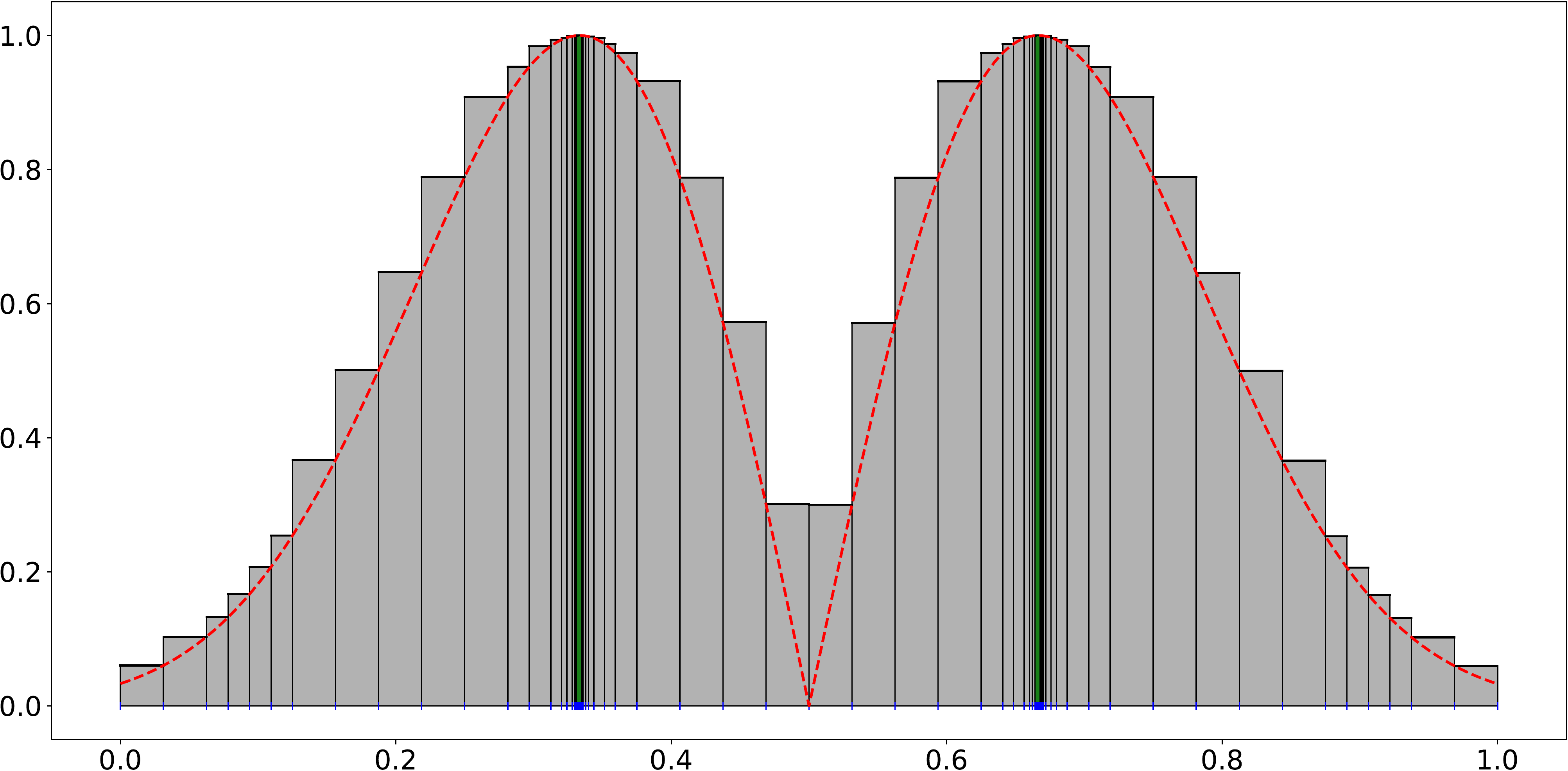}
      \caption{$|\calV_{12}|=67$}
    \end{subfigure}
    \caption{The adaptive refinement algorithm with a second order selection process and first order gradient. The setting is the same as the one of Figure~\ref{fig:behavior1D_nograd}. The cardinality of the set of candidates $\Omega_k^\star$ is smaller, as can be seen when comparing Figure~\ref{fig:behavior1D_grad6} with Figure~\ref{fig:behavior1D_nograd6}.}
    \label{fig:behavior1D_grad}
  \end{figure}

\begin{table}[!ht]
\begin{subtable}[t]{0.45\textwidth}
\centering
\begin{tabular}[t]{lccccc}
\toprule
Iteration & $|\calV_k|$ & primal & $\distH(\calV_k, X^\star)$\\
\midrule
0 & 2 & 3.80563e+03 & 3.3e-01 \\
1 & 3 & 3.79912e+03 & 1.7e-01 \\
2 & 5 & 9.39226e+02 & 8.3e-02 \\
3 & 9 & 3.01878e+01 & 4.2e-02 \\
4 & 17 & 1.84675e+01 & 2.1e-02 \\
5 & 33 & 1.72061e+01 & 1.0e-02 \\
6 & 43 & 1.70209e+01 & 5.1e-03 \\
7 & 49 & 1.69895e+01 & 2.7e-03 \\
8 & 55 & 1.69826e+01 & 1.2e-03 \\
9 & 61 & 1.69810e+01 & 7.2e-04 \\
10 & 67 & 1.69873e+01 & 2.6e-04 \\
11 & 73 & 1.69828e+01 & 2.3e-04 \\
12 & 79 & 1.69806e+01 & 1.9e-05 \\
13 & 89 & 1.69811e+01 & 1.9e-05 \\
14 & 105 & 1.69806e+01 & 1.9e-05 \\
15 & 132 & 1.69805e+01 & 1.2e-05 \\
16 & 162 & 1.69805e+01 & 4.3e-06 \\
17 & 208 & 1.69805e+01 & 3.5e-06 \\
18 & 272 & 1.69805e+01 & 4.6e-07 \\
\bottomrule
\end{tabular}
\caption{Refinement rule with second-order bounds.\label{tab:behavior1D_nograd}}
\end{subtable}
\hspace{\fill}
\begin{subtable}[t]{0.45\textwidth}
\centering
\begin{tabular}[t]{lccccc}
\toprule
Iteration & $|\calV_k|$ & primal & $\distH(\calV_k, X^\star)$\\
\midrule
0 & 2 & 3.80563e+03 & 3.3e-01 \\
1 & 3 & 3.79912e+03 & 1.7e-01 \\
2 & 5 & 9.39226e+02 & 8.3e-02 \\
3 & 9 & 3.01878e+01 & 4.2e-02 \\
4 & 17 & 1.84675e+01 & 2.1e-02 \\
5 & 33 & 1.72061e+01 & 1.0e-02 \\
6 & 43 & 1.70209e+01 & 5.1e-03 \\
7 & 45 & 1.69895e+01 & 2.7e-03 \\
8 & 47 & 1.69895e+01 & 2.7e-03 \\
9 & 53 & 1.69826e+01 & 1.2e-03 \\
10 & 55 & 1.69826e+01 & 1.2e-03 \\
11 & 61 & 1.69810e+01 & 7.2e-04 \\
12 & 67 & 1.69873e+01 & 2.6e-04 \\
13 & 73 & 1.69828e+01 & 2.3e-04 \\
14 & 79 & 1.69806e+01 & 1.9e-05 \\
15 & 83 & 1.69831e+01 & 1.9e-05 \\
16 & 87 & 1.69816e+01 & 1.9e-05 \\
17 & 92 & 1.69806e+01 & 1.2e-05 \\
18 & 96 & 1.69805e+01 & 4.3e-06 \\
19 & 98 & 1.69805e+01 & 4.3e-06 \\
20 & 100 & 1.69805e+01 & 4.3e-06 \\
21 & 104 & 1.69805e+01 & 3.5e-06 \\
... & ... & ... & ... \\
31 & 125 & 1.69805e+01 & 3.5e-06 \\
32 & 128 & 1.69805e+01 & 4.6e-07 \\
\bottomrule
\end{tabular}
\caption{Refinement rule with second-order upper bounds and gradient lower bound.\label{tab:behavior1D_grad}}
\end{subtable}
\caption{Algorithm's behavior for the 1D super-resolution problem. Here, we set $\sigma = 2/M$, $M=20$.\label{tab:behavior1D}}
\end{table}%

\paragraph{Second-order upper bound with first order gradient}
We turn our attention to the second order selection process with first order gradient, see Definition~\ref{def:second_order_selection_with_grad}. The results are displayed in Table~\ref{tab:behavior1D_grad} and Figure \ref{fig:behavior1D_grad}. In comparison to the previous test, a lower bound on the gradient is used to reduce the cardinality of $\Omega_k^\star$. Two measures of complexity can be used to compare the approaches: i) the cardinality $|\calV_k|$ needed to reach a given accuracy $\distH(\calV_k,X^\star)$, or ii) the number of iterations to reach the same accuracy. 
Reducing the cardinality of $\Omega_k^\star$ can be detrimental to the second notion of complexity. 
For example, compare Figure~\ref{fig:behavior1D_grad6} and Figure~\ref{fig:behavior1D_nograd6}. The cells that are not flagged for refinement in Figure~\ref{fig:behavior1D_grad6} are flagged in Figure~\ref{fig:behavior1D_grad7} and refined at Iteration $7$. Iteration 7 can be seen as a failed {\em zwischenzug} iteration that loses a tempo. However, for the first notion of complexity, the conclusion is different. We see that for this particular example, adding a gradient lower bounds allows reaching the target precision in Table~\ref{tab:behavior1D_grad} with less than half the number of vertices for the vanilla second order bound.
A full complexity analysis would require a fine analysis of the quadratic programming solver, which is out of the scope of this paper. 

\subsection{2D experiments} \label{sec:2dexp}

\paragraph{The problem}
In this section, we assume that the sampling points $z_m$ lie on a Euclidean grid. 
More precisely, we suppose that $\sqrt{M}\in \N$ and that each index $m\in \llbracket 1, M\rrbracket$ can be decomposed as $m=(m_1,m_2) \in \llbracket 0,\sqrt{M}-1 \rrbracket^2$ and $z_m = \frac{1}{\sqrt{M}}(m_1,m_2)$ \text{ with $M=15$. $\bar \mu$ is chosen as}
\begin{equation*}
  \bar \mu = -9\delta_{(1/3,1/3)} + 8\delta_{(1/3,2/3)} + 5\delta_{(2/3,2/3)}.
\end{equation*}

\paragraph{Results}

We begin by showcasing the behaviour of the algorithm when the second-order upper bound \label{eq:second_order_bound} is used. Table \ref{tab:super_resolution_2D_nograd} and Figure \ref{fig:super_resolution_2D} summarize the algorithm's behavior.
The conclusions are similar to the previous section and consistent with Theorem \ref{th:linConv}: after a burn-in period, the grid is refined in a multi-scale fashion, only around the support $X^\star$ of the solution $\mu^\star$. To control the complexity of our algorithm we refine only the cells with largest diameter. The effect of this strategy is striking in $2D$, where the algorithm spends some iterations to refine larger cells only. See Figure~\ref{fig:super_resolution_2D}, iterations $6,8,9, 10$. At these iterations, it is not the cells containing the maximizers $X_k$ which are refined, but only the largest ones which were not refined in the previous iterations. Yet, the table indicates a clear advantage of this adaptive method: about 3000 vertices are sufficient to reach a precision $10^{-4}$, while the same guarantee would be obtained only with $10^8$ vertices for a uniform refinement. The results when a lower bound of the gradient is added are displayed in Table \ref{tab:super_resolution_2D_grad} and Figure \ref{fig:super_resolution_2D_grad}. For this example, there is no increase in the number of iterations, and only a slight decrease of the number of vertices is observed. Again, a more detailed analysis of the effects of gradient-including rules is beyond the scope of this paper.

\begin{figure}[!ht]
    \centering
    \begin{subfigure}[b]{0.22\textwidth}
      \includegraphics[width=\textwidth]{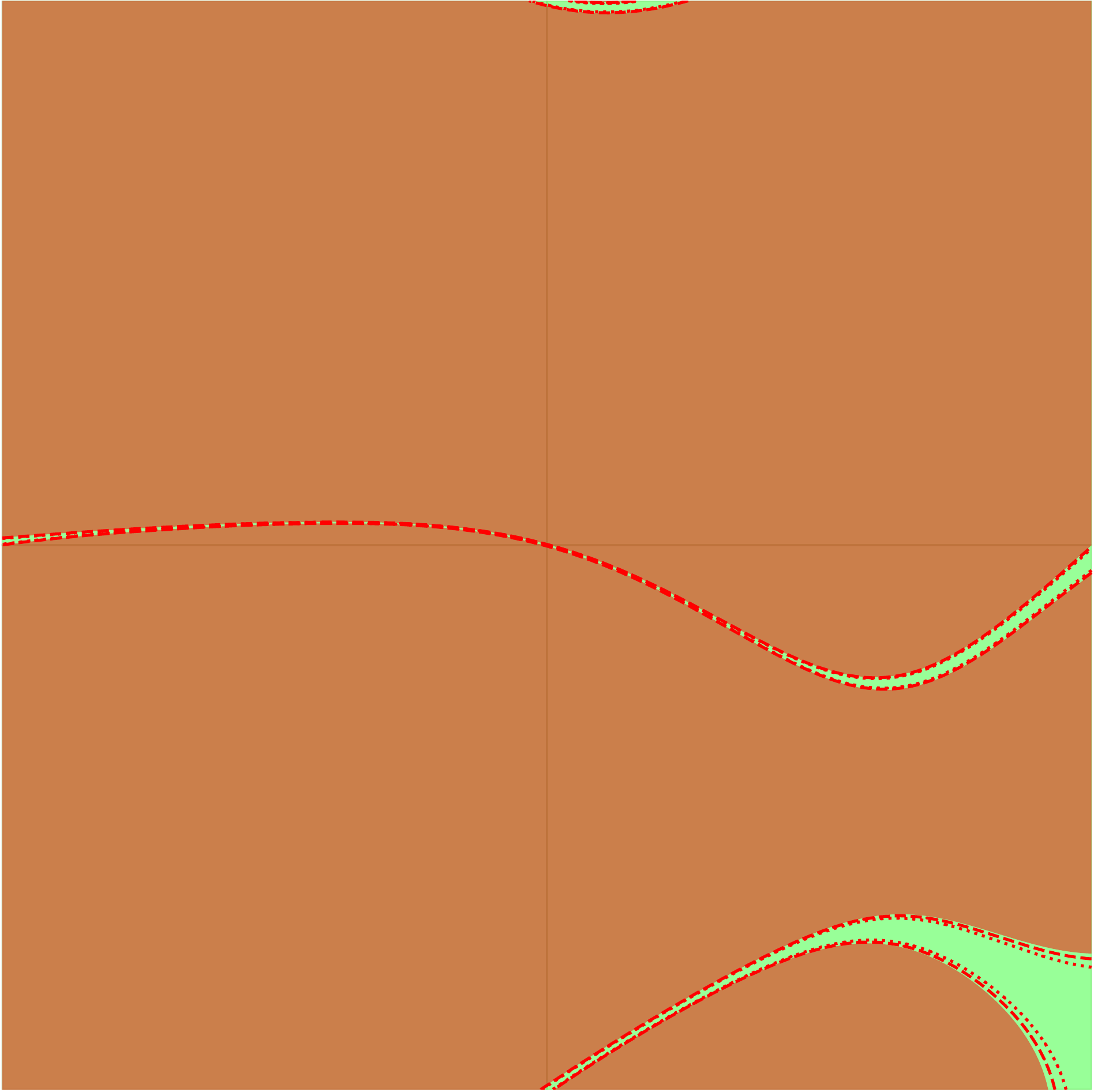}
      \caption{$|\calV_1|=9$}
    \end{subfigure}
    \begin{subfigure}[b]{0.22\textwidth}
      \includegraphics[width=\textwidth]{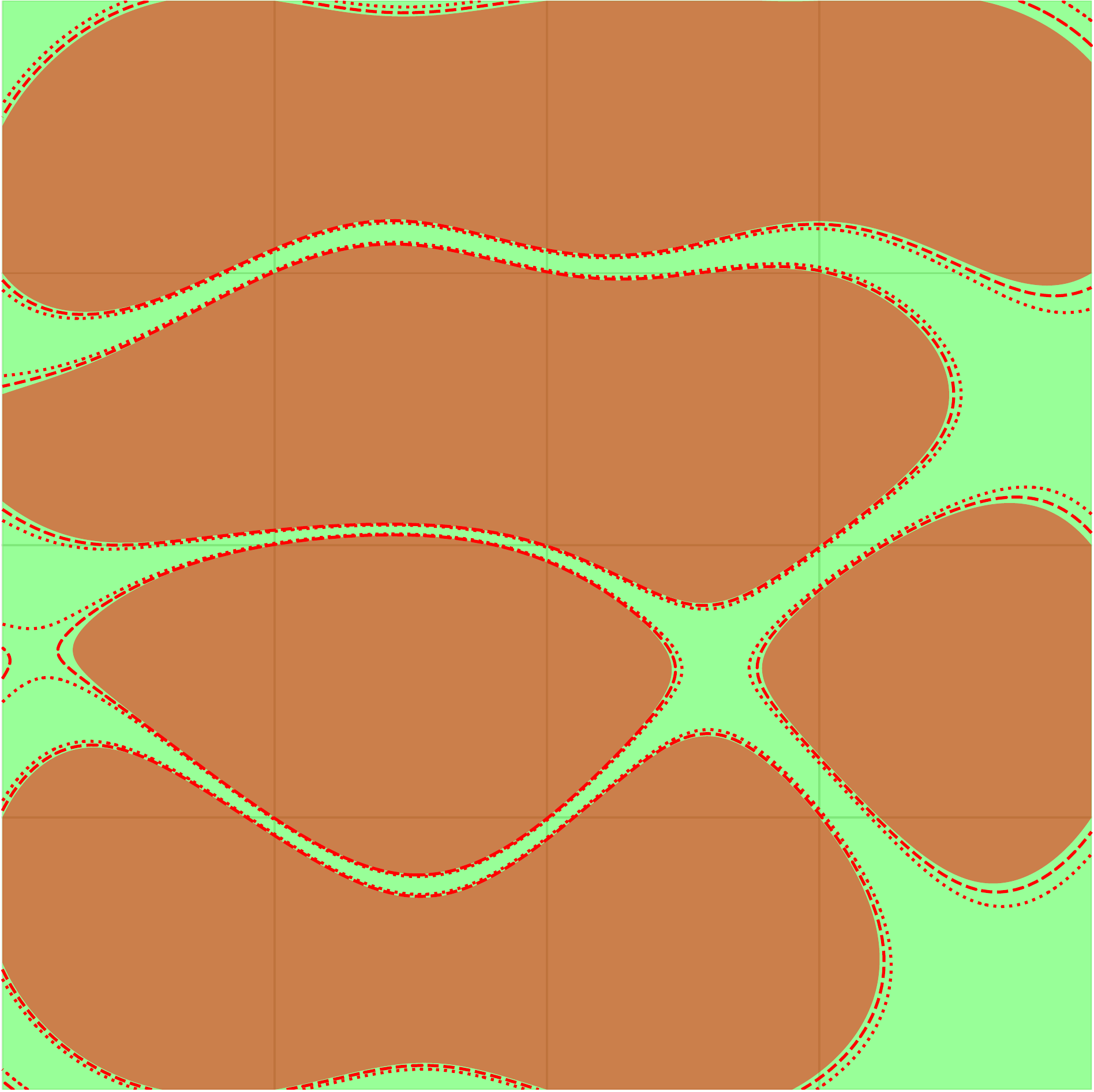}
      \caption{$|\calV_2|=25$}
    \end{subfigure}
    \begin{subfigure}[b]{0.22\textwidth}
      \includegraphics[width=\textwidth]{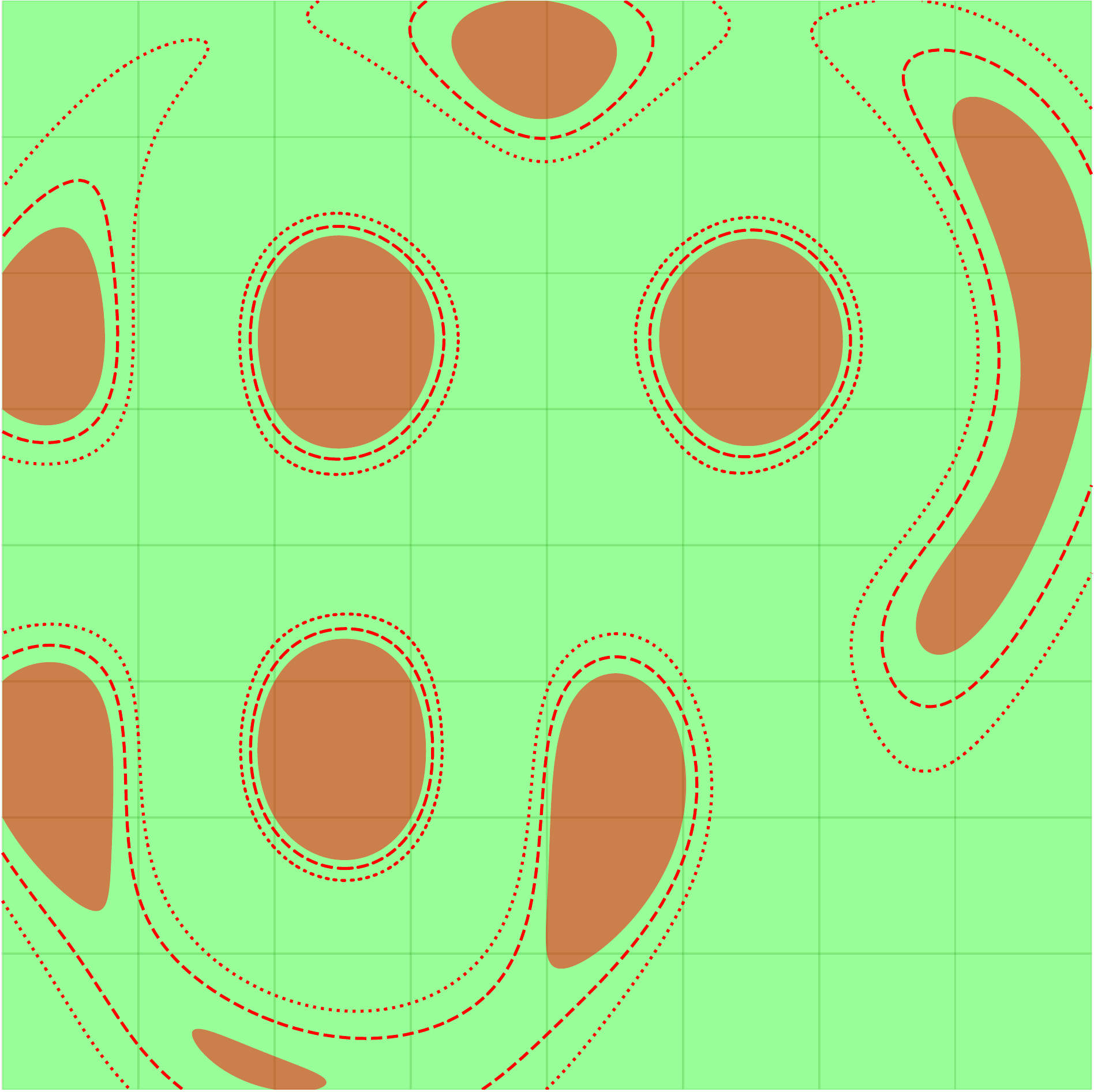}
      \caption{$|\calV_3|=81$}
    \end{subfigure}
    \begin{subfigure}[b]{0.22\textwidth}
      \includegraphics[width=\textwidth]{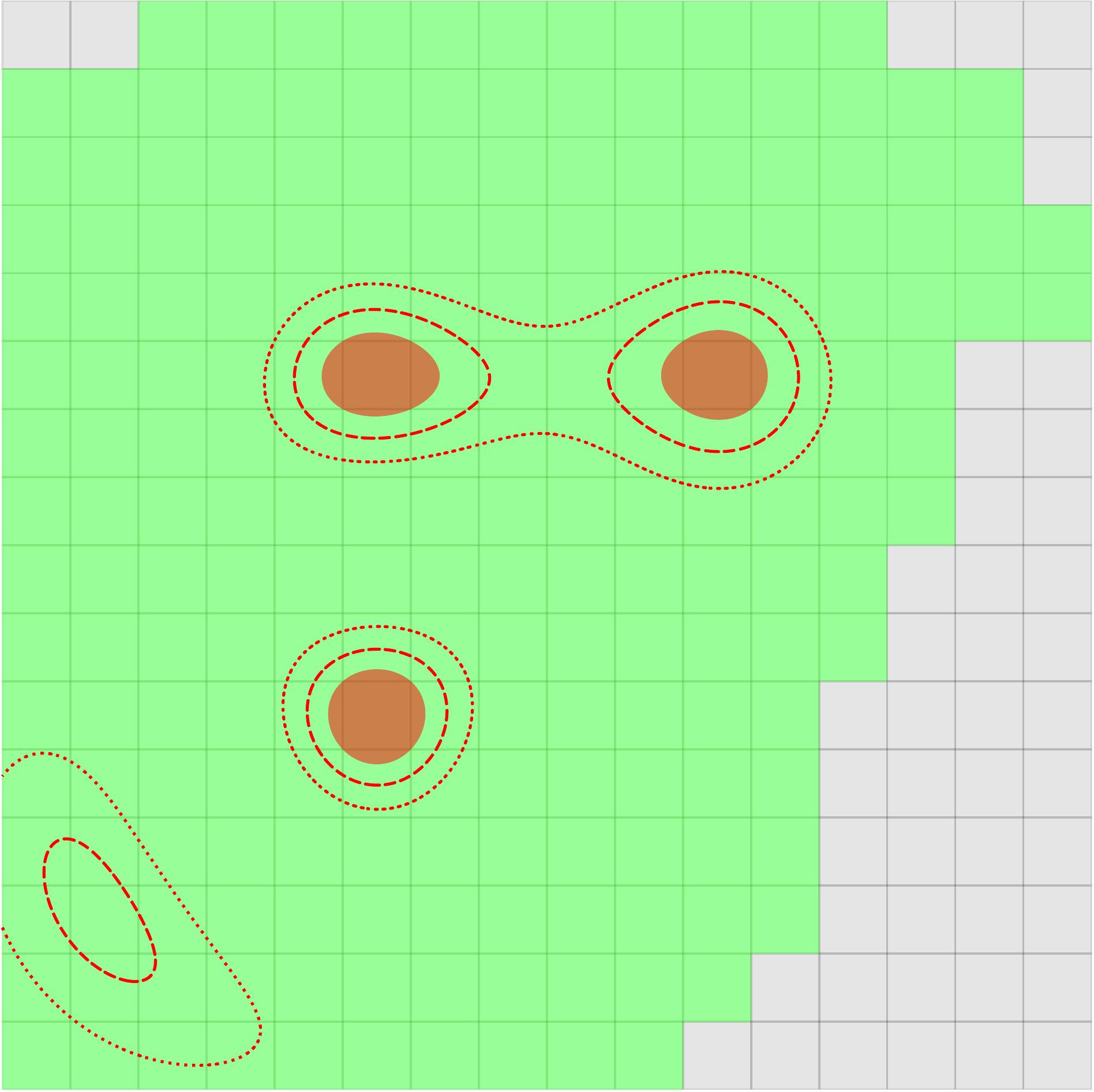}
      \caption{$|\calV_4|=289$}
    \end{subfigure}
    \begin{subfigure}[b]{0.22\textwidth}
      \includegraphics[width=\textwidth]{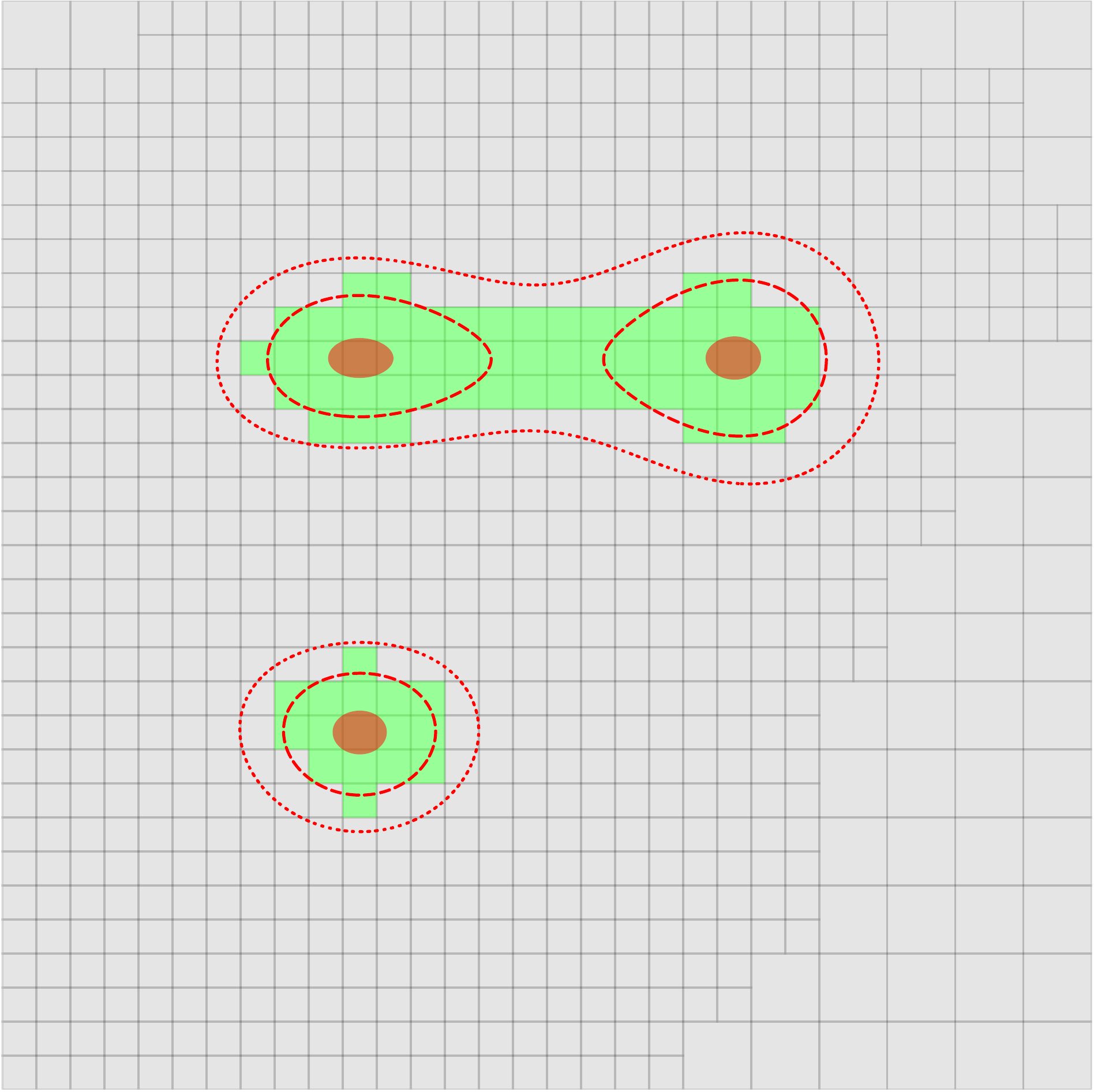}
      \caption{$|\calV_5|=951$}
    \end{subfigure}
    \begin{subfigure}[b]{0.22\textwidth}
      \includegraphics[width=\textwidth]{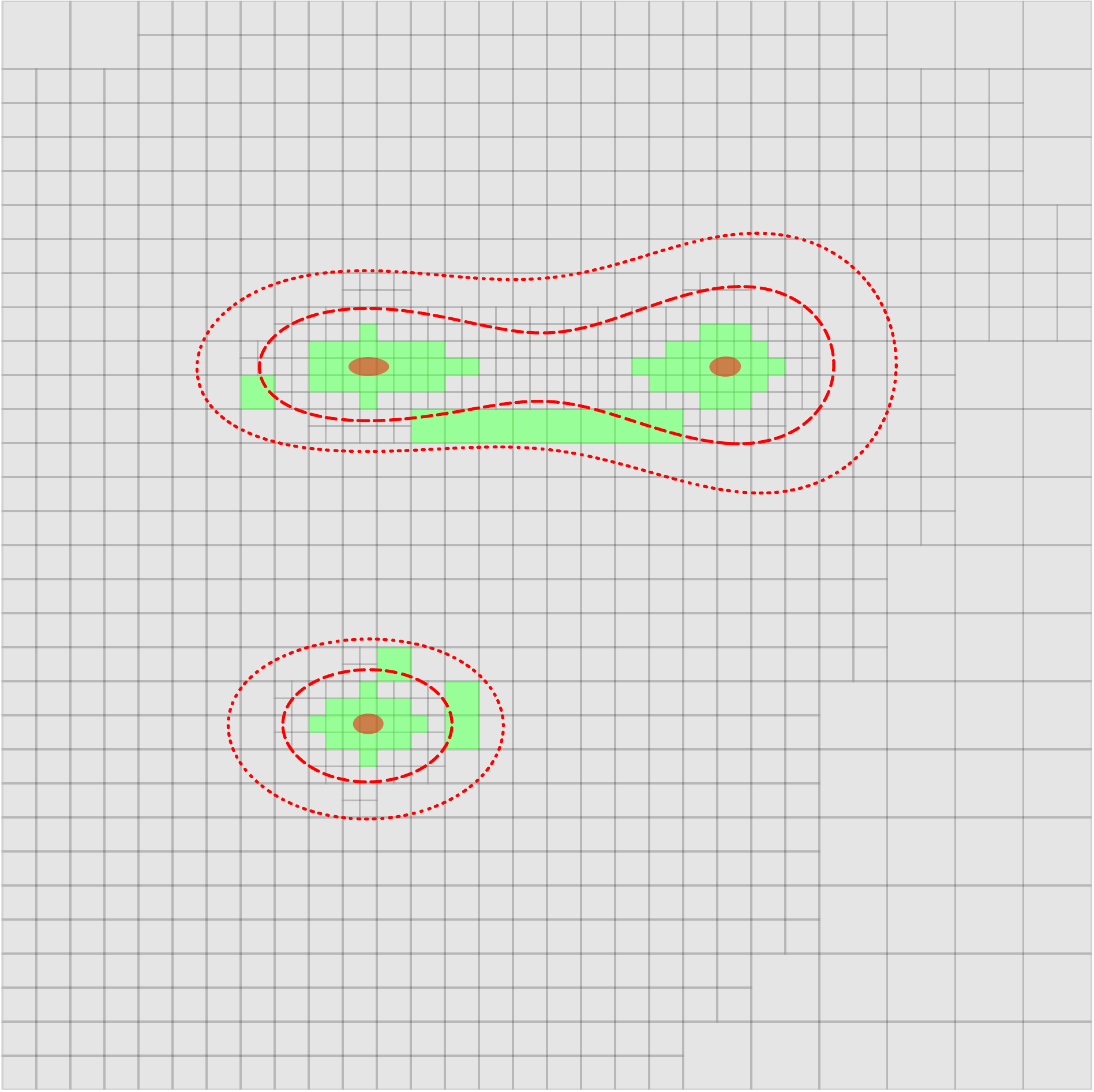}
      \caption{$|\calV_6|=1210$}
      \label{fig:2D1:6}
    \end{subfigure}
    \begin{subfigure}[b]{0.22\textwidth}
      \includegraphics[width=\textwidth]{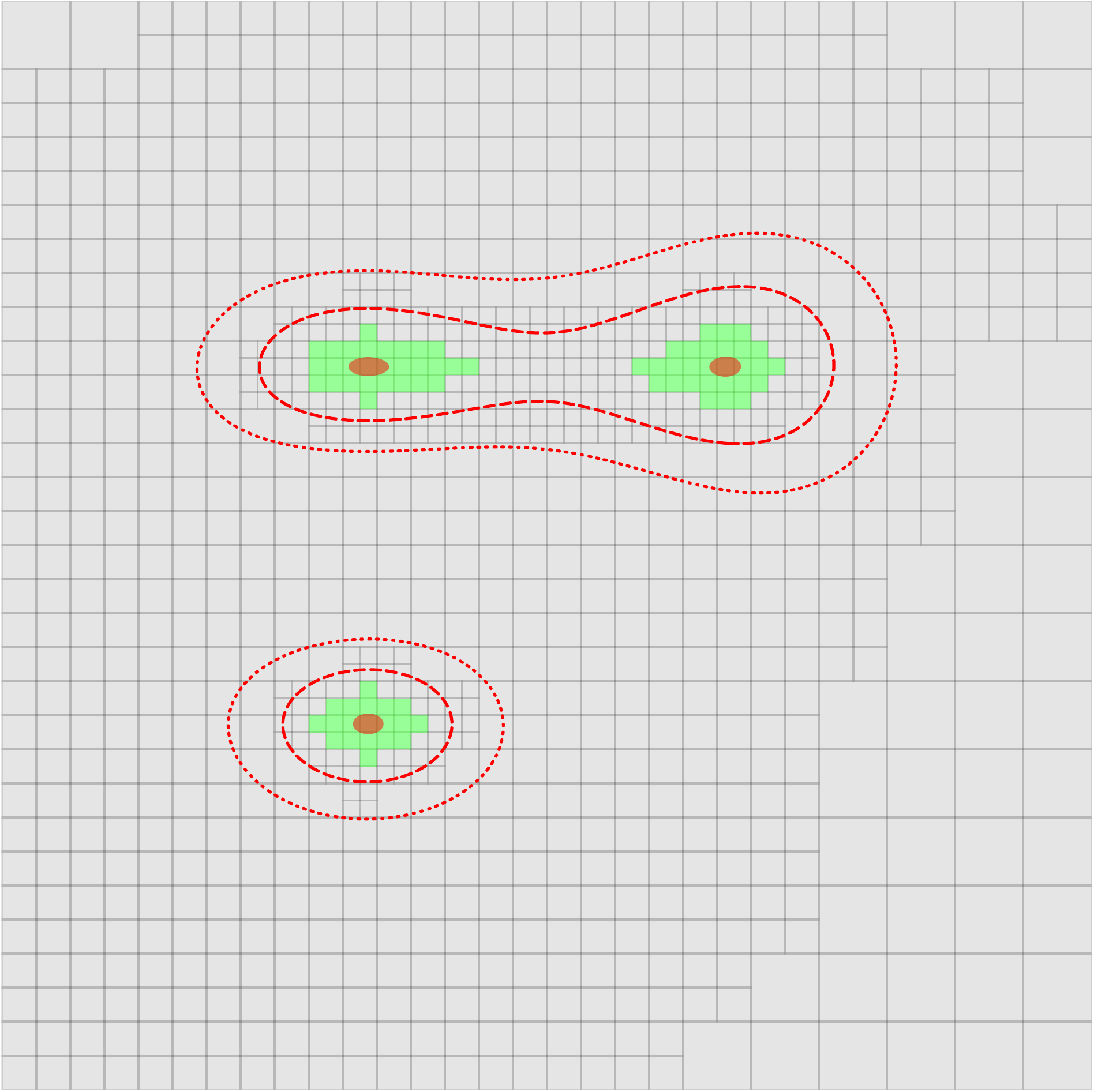}
      \caption{$|\calV_7|=1246$}
      \label{fig:2D1:7}
    \end{subfigure}
    \begin{subfigure}[b]{0.22\textwidth}
      \includegraphics[width=\textwidth]{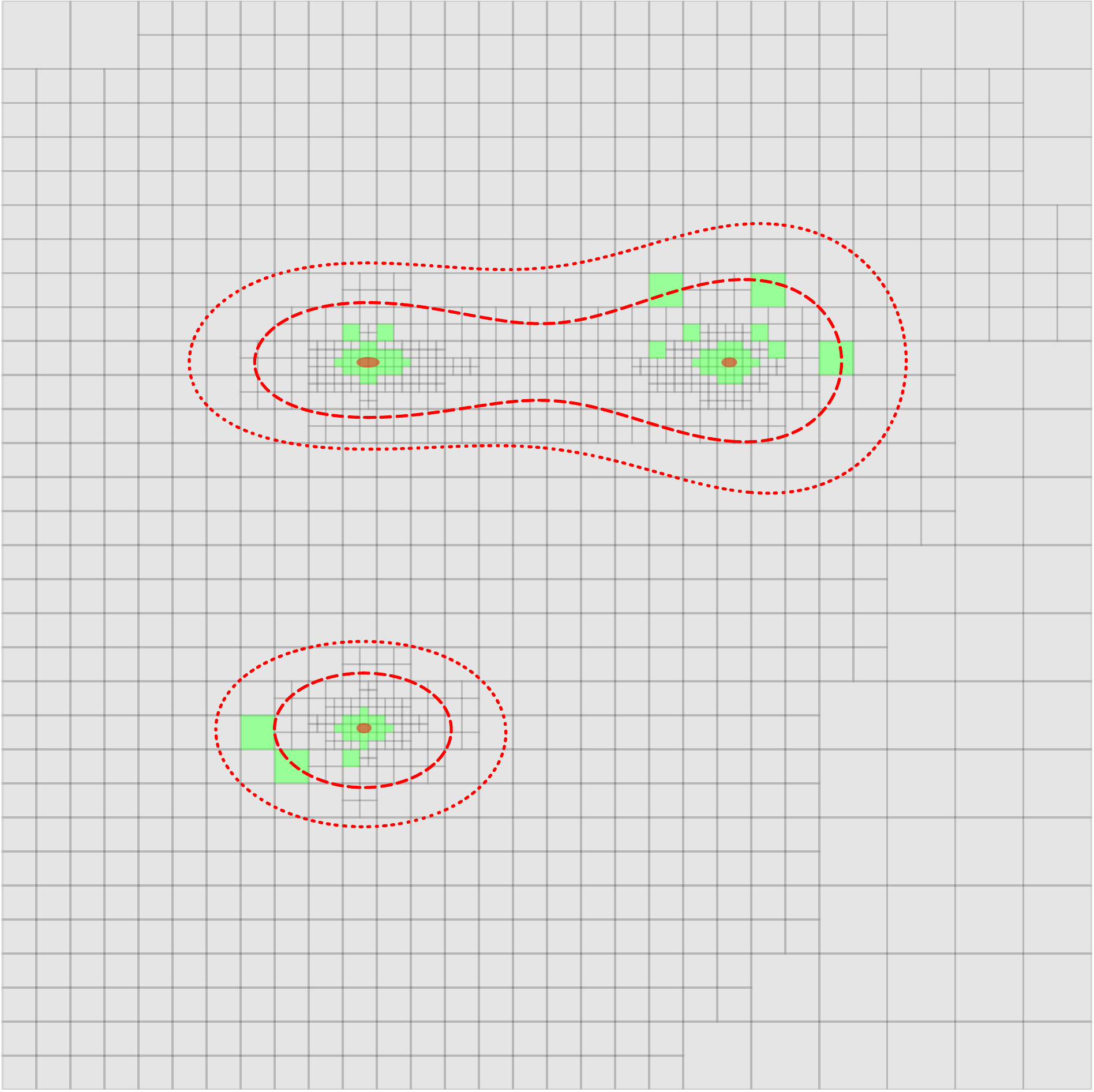}
      \caption{$|\calV_8|=1512$}
    \end{subfigure}
    \begin{subfigure}[b]{0.22\textwidth}
      \includegraphics[width=\textwidth]{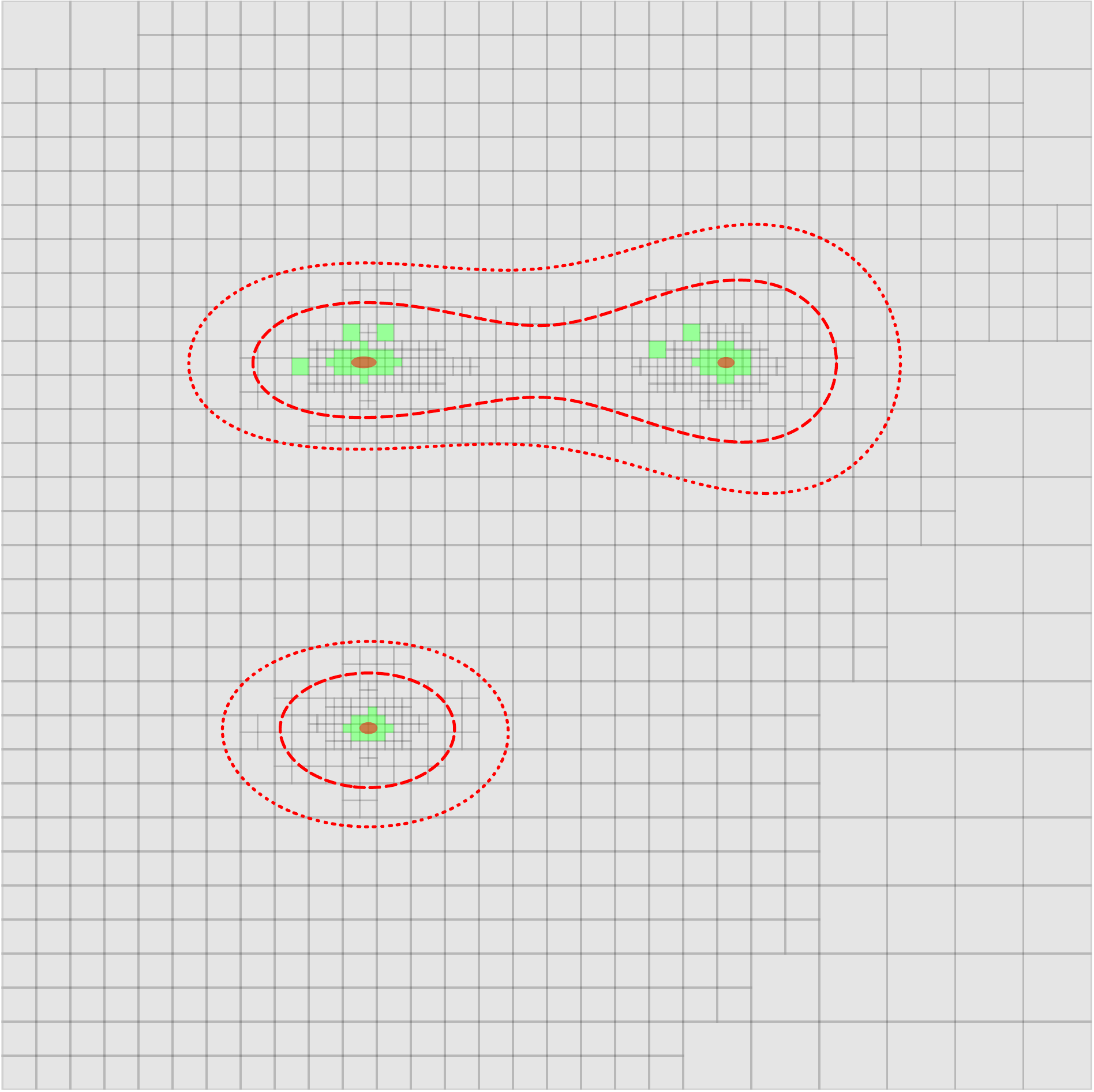}
      \caption{$|\calV_9|=1529$}
    \end{subfigure}
    \begin{subfigure}[b]{0.22\textwidth}
      \includegraphics[width=\textwidth]{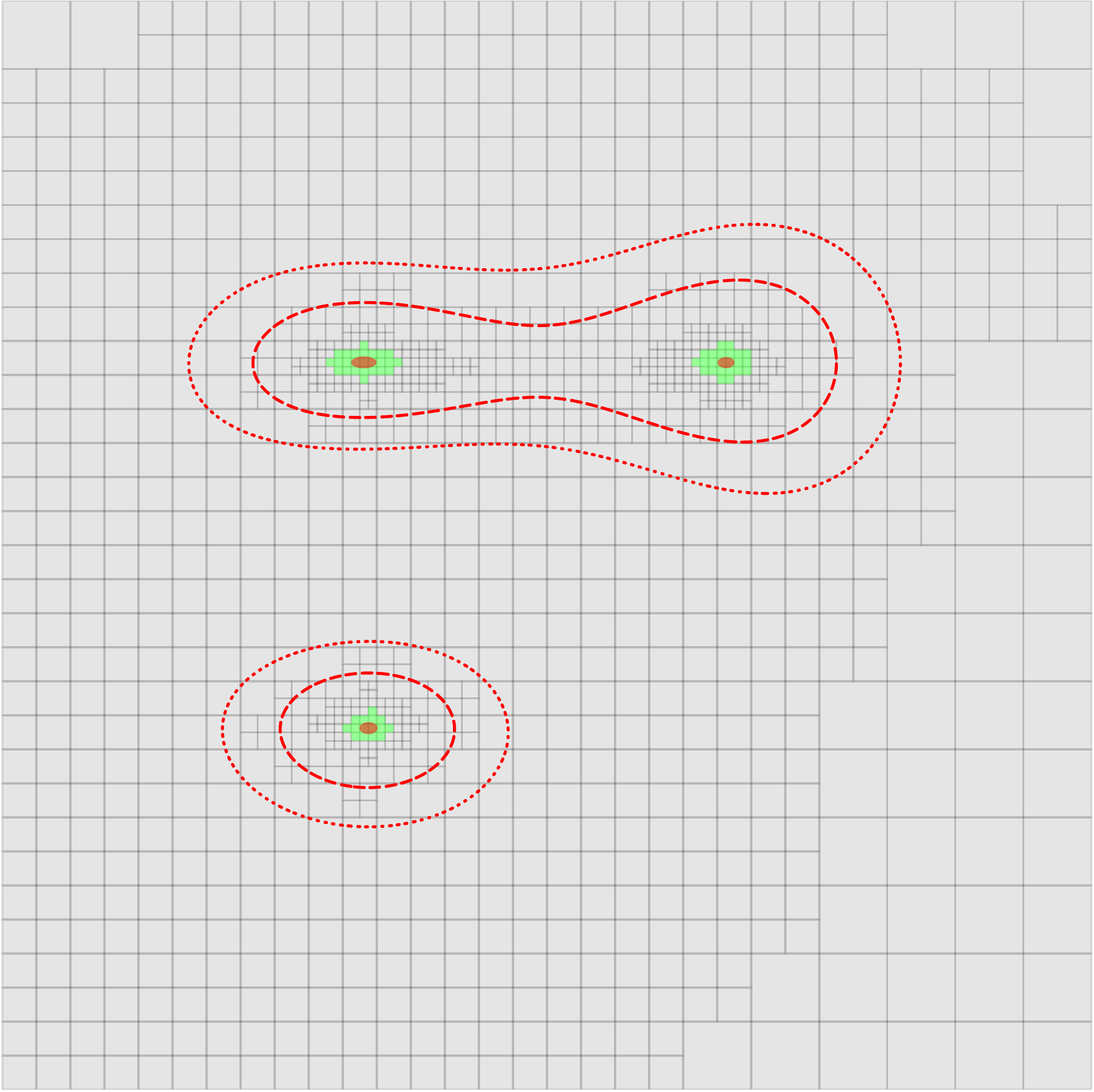}
      \caption{$|\calV_{10}|=1545$}
    \end{subfigure}
    \begin{subfigure}[b]{0.22\textwidth}
      \includegraphics[width=\textwidth]{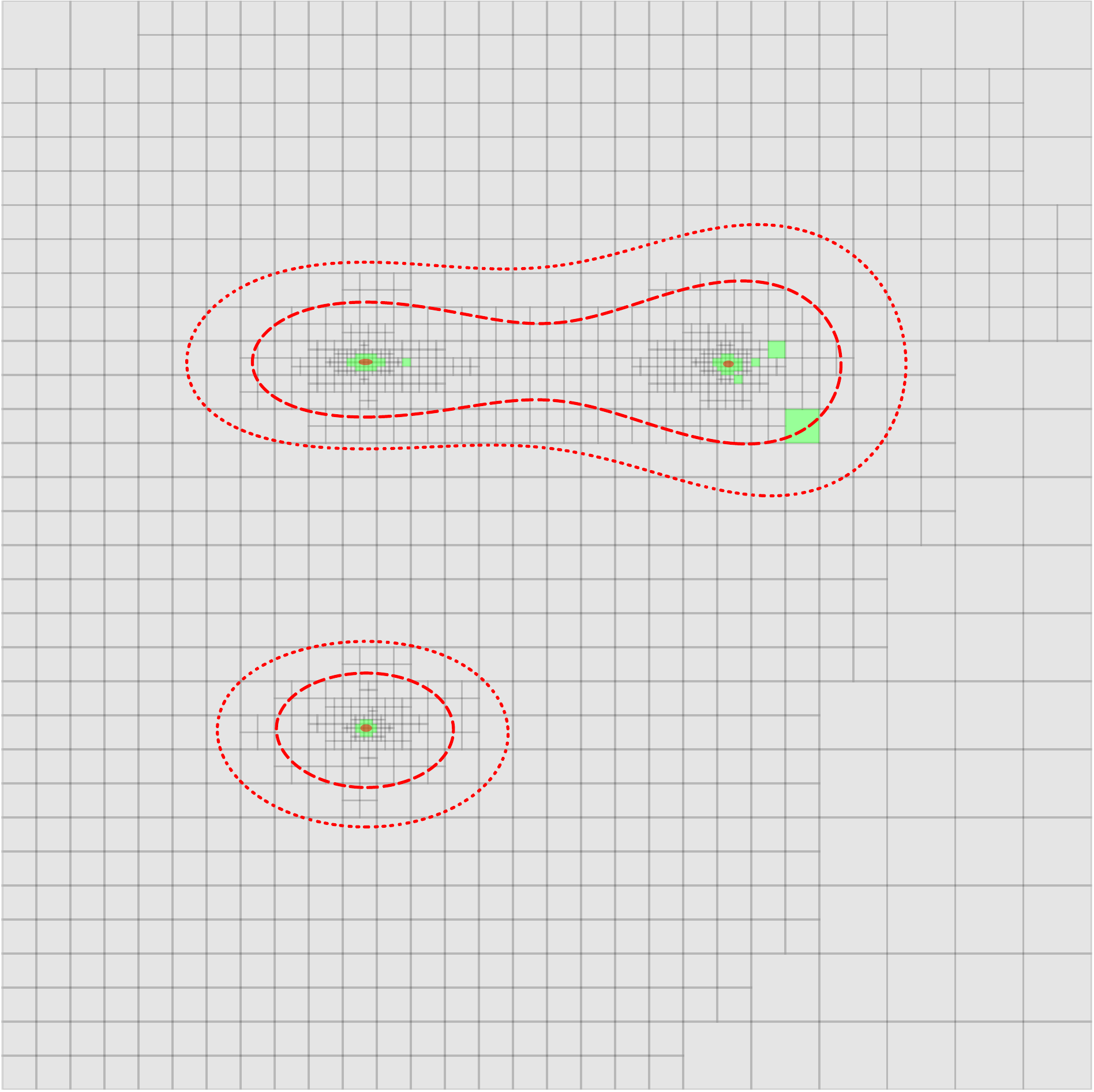}
      \caption{$|\calV_{11}|=1770$}
    \end{subfigure}      
    \begin{subfigure}[b]{0.22\textwidth}
      \includegraphics[width=\textwidth]{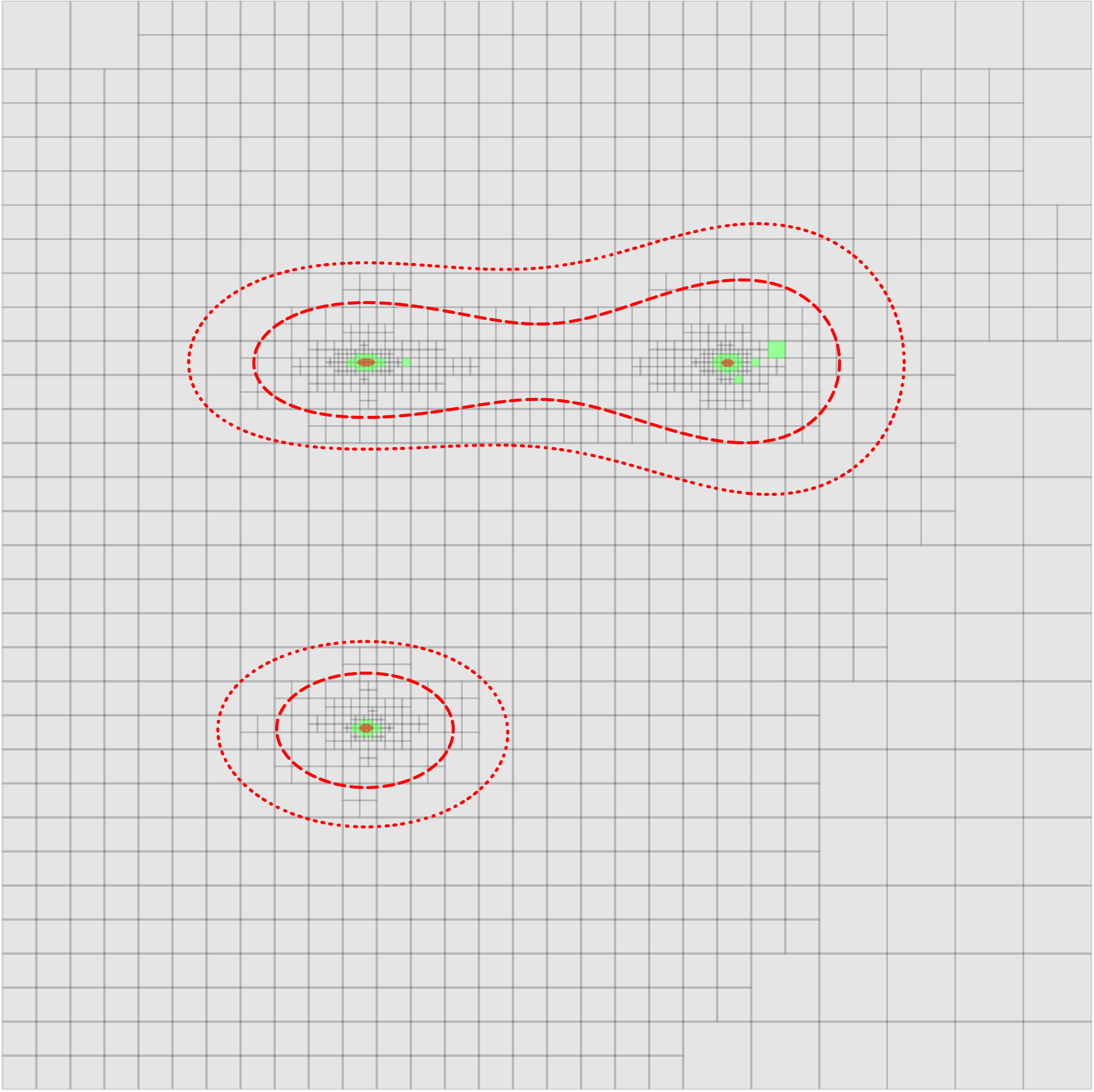}
      \caption{$|\calV_{12}|=1773$}
    \end{subfigure}
    \begin{subfigure}[b]{0.22\textwidth}
      \includegraphics[width=\textwidth]{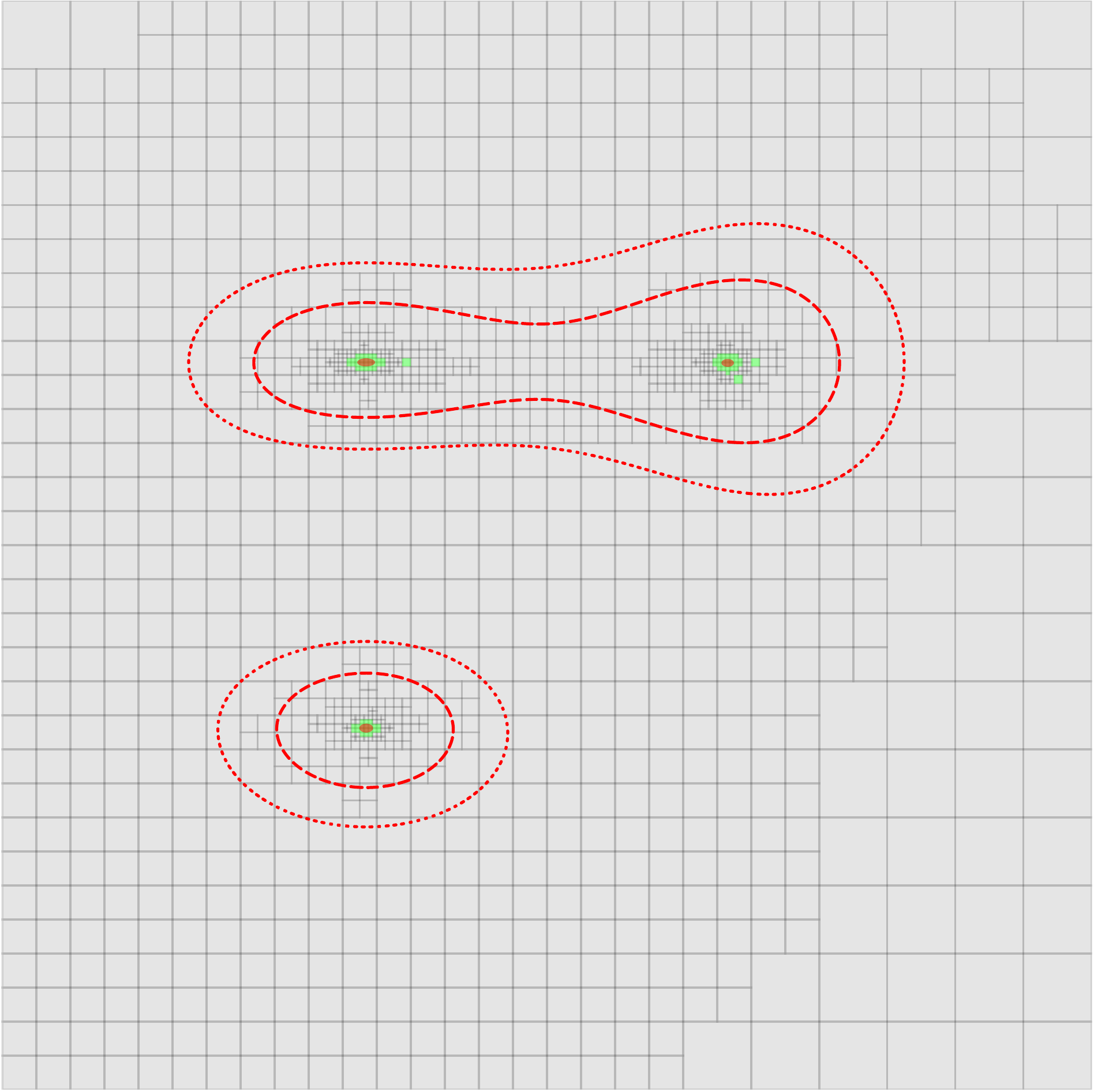}
      \caption{$|\calV_{13}|=1776$}
    \end{subfigure}
    \begin{subfigure}[b]{0.22\textwidth}
      \includegraphics[width=\textwidth]{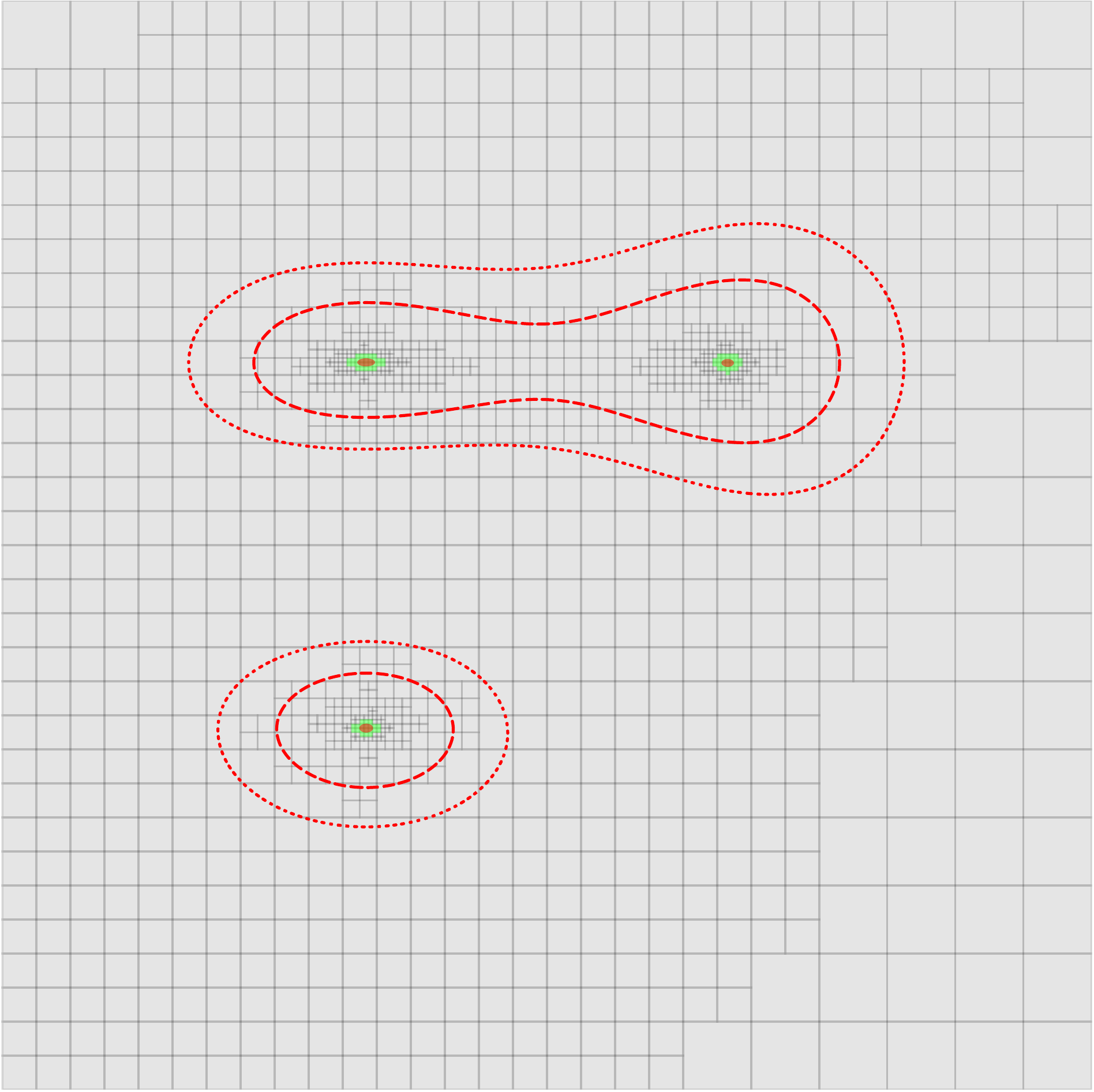}
      \caption{$|\calV_{14}|=1787$}
    \end{subfigure}
    \begin{subfigure}[b]{0.22\textwidth}
      \includegraphics[width=\textwidth]{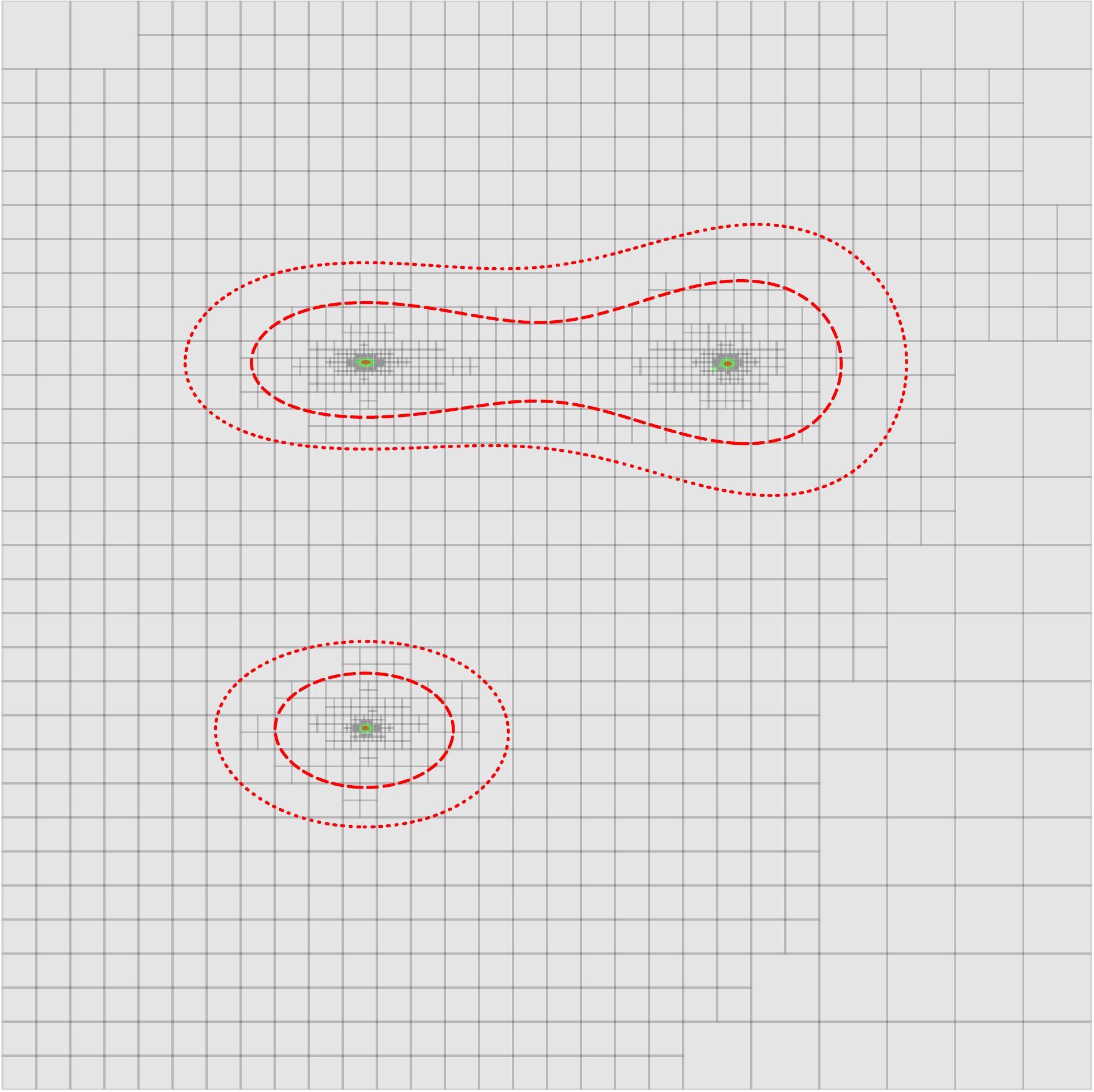}
      \caption{$|\calV_{15}|=2042$}
    \end{subfigure}
    \begin{subfigure}[b]{0.22\textwidth}
      \includegraphics[width=\textwidth]{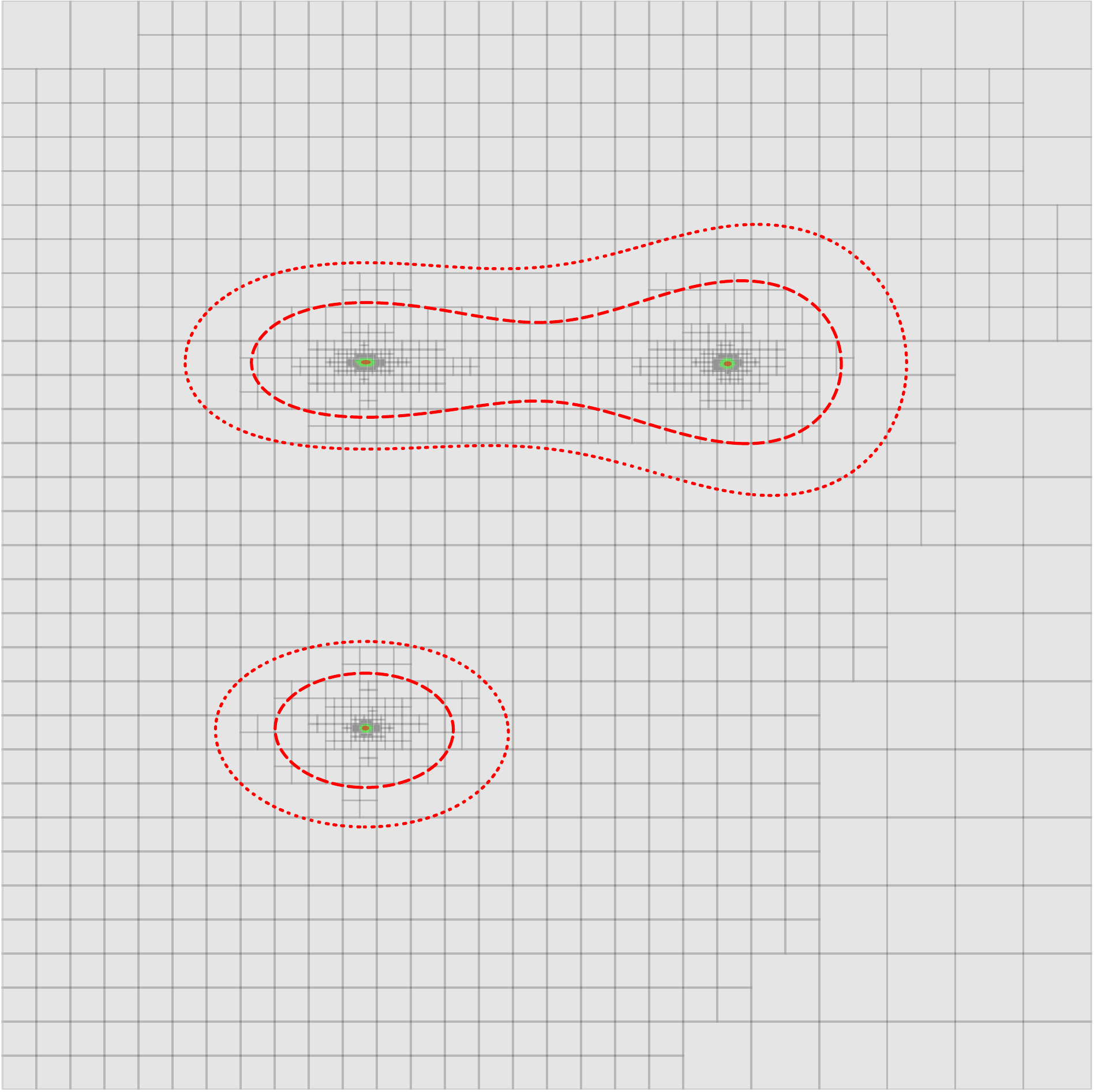}
      \caption{$|\calV_{16}|=2045$}
    \end{subfigure}
    \caption{Our algorithm's behavior on a 2D example for second-order selection rules.
  The set $\Omega_k^\star$ is displayed in green, the superlevelset $1$ of $|A^*q_k|$ is filled with red,  the level-set $0.9$ is represented with a dashed red line and the levelset $0.75$ with a dotted red line. 
    The algorithm starts with a burn-in period of 3 iterations. There, it refines all cells uniformly since the upper-bound is highly inaccurate. Then, only the cells around the locations of $X^\star$ get refined in a multiscale fashion. Remember that only the cells in $\Omega_k^*$ with largest diameter are refined. This explains the behavior of the algorithm between, e.g. Figures~\ref{fig:2D1:6} and \ref{fig:2D1:7}.    
    \label{fig:super_resolution_2D}}
  \end{figure}

\begin{figure}[!t]
    \centering
    \begin{subfigure}[b]{0.22\textwidth}
      \includegraphics[width=\textwidth]{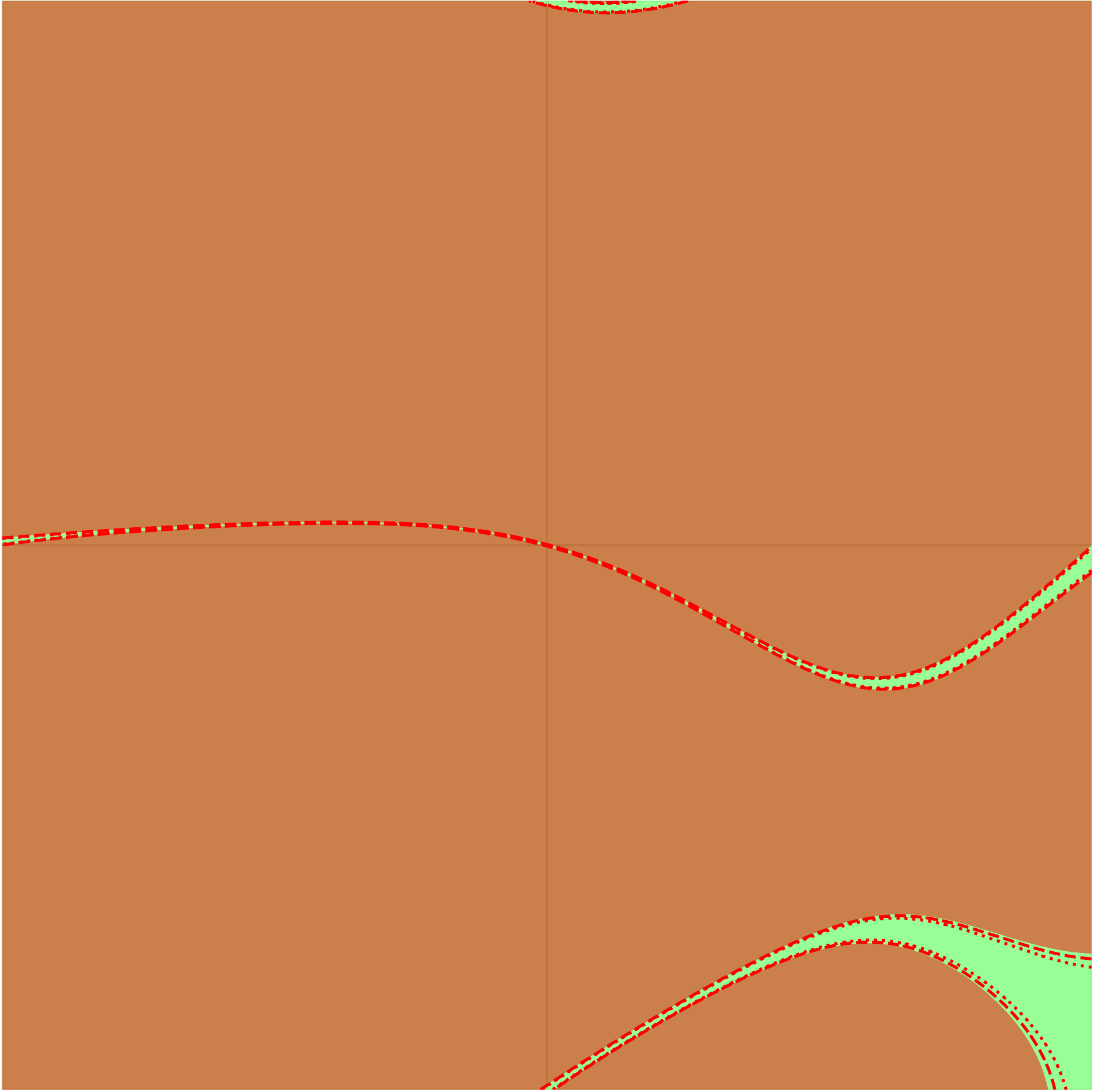}
      \caption{$|\calV_1|=9$}
    \end{subfigure}
    \begin{subfigure}[b]{0.22\textwidth}
      \includegraphics[width=\textwidth]{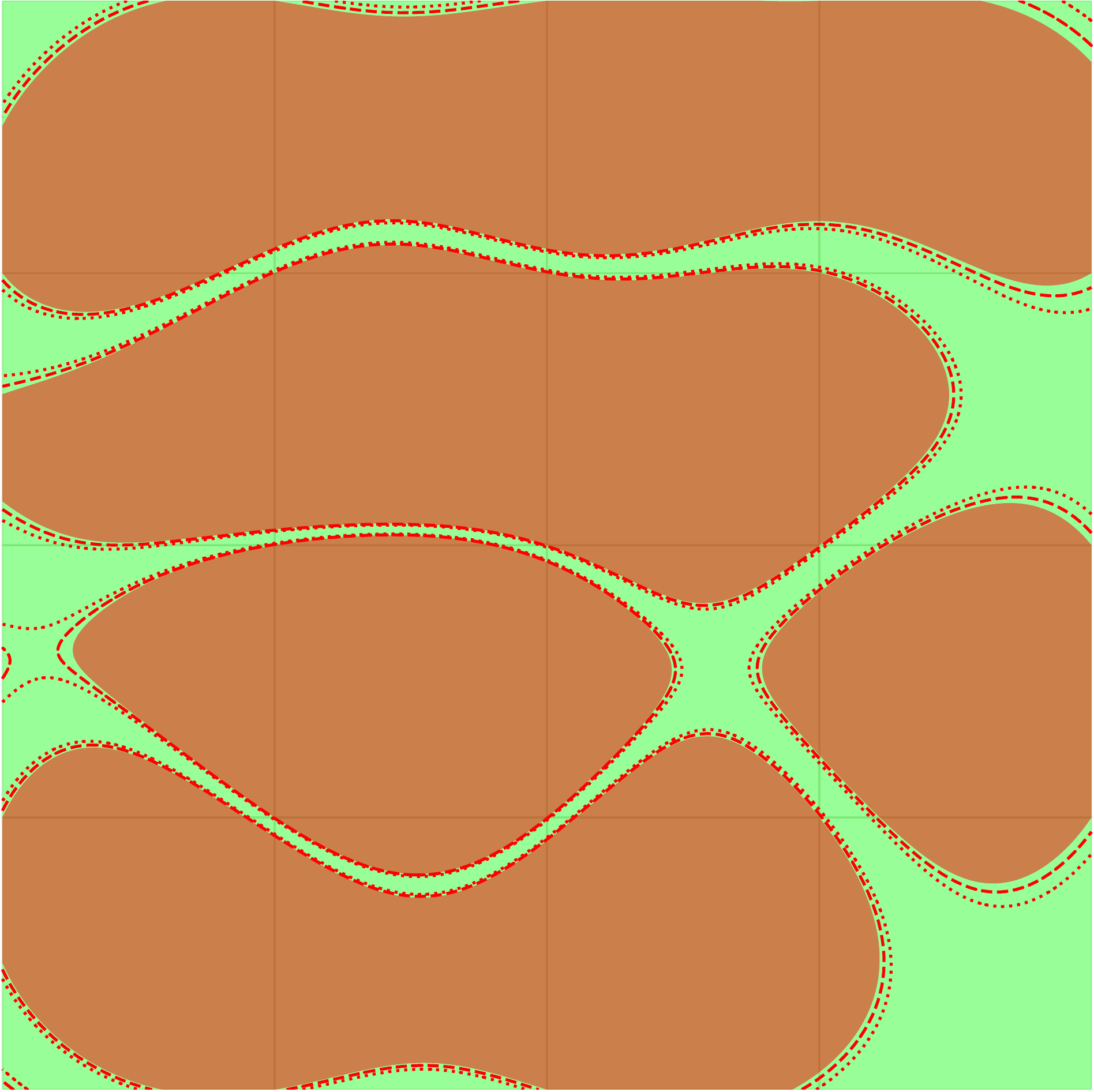}
      \caption{$|\calV_2|=25$}
    \end{subfigure}
    \begin{subfigure}[b]{0.22\textwidth}
      \includegraphics[width=\textwidth]{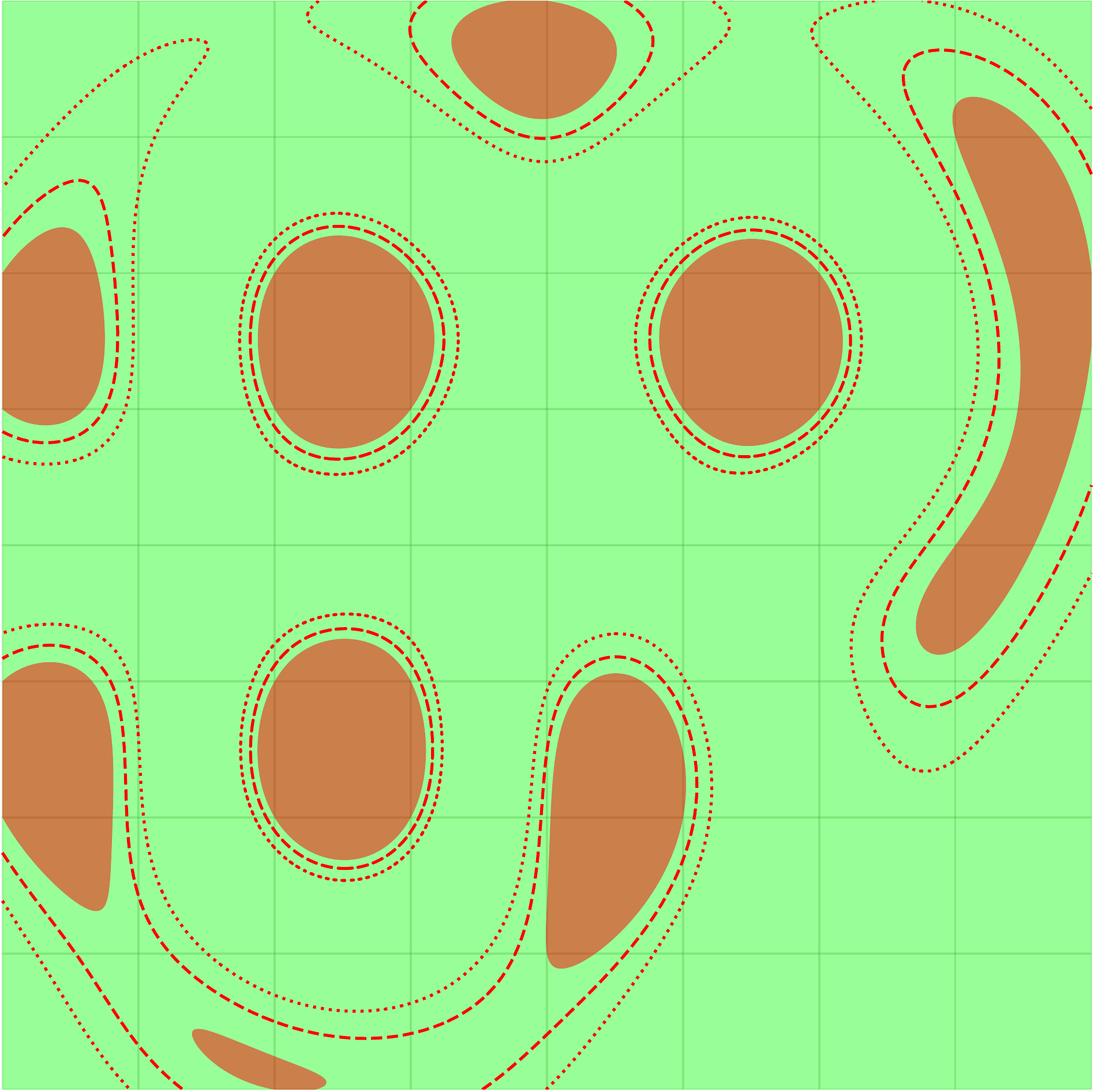}
      \caption{$|\calV_3|=81$}
    \end{subfigure}
    \begin{subfigure}[b]{0.22\textwidth}
      \includegraphics[width=\textwidth]{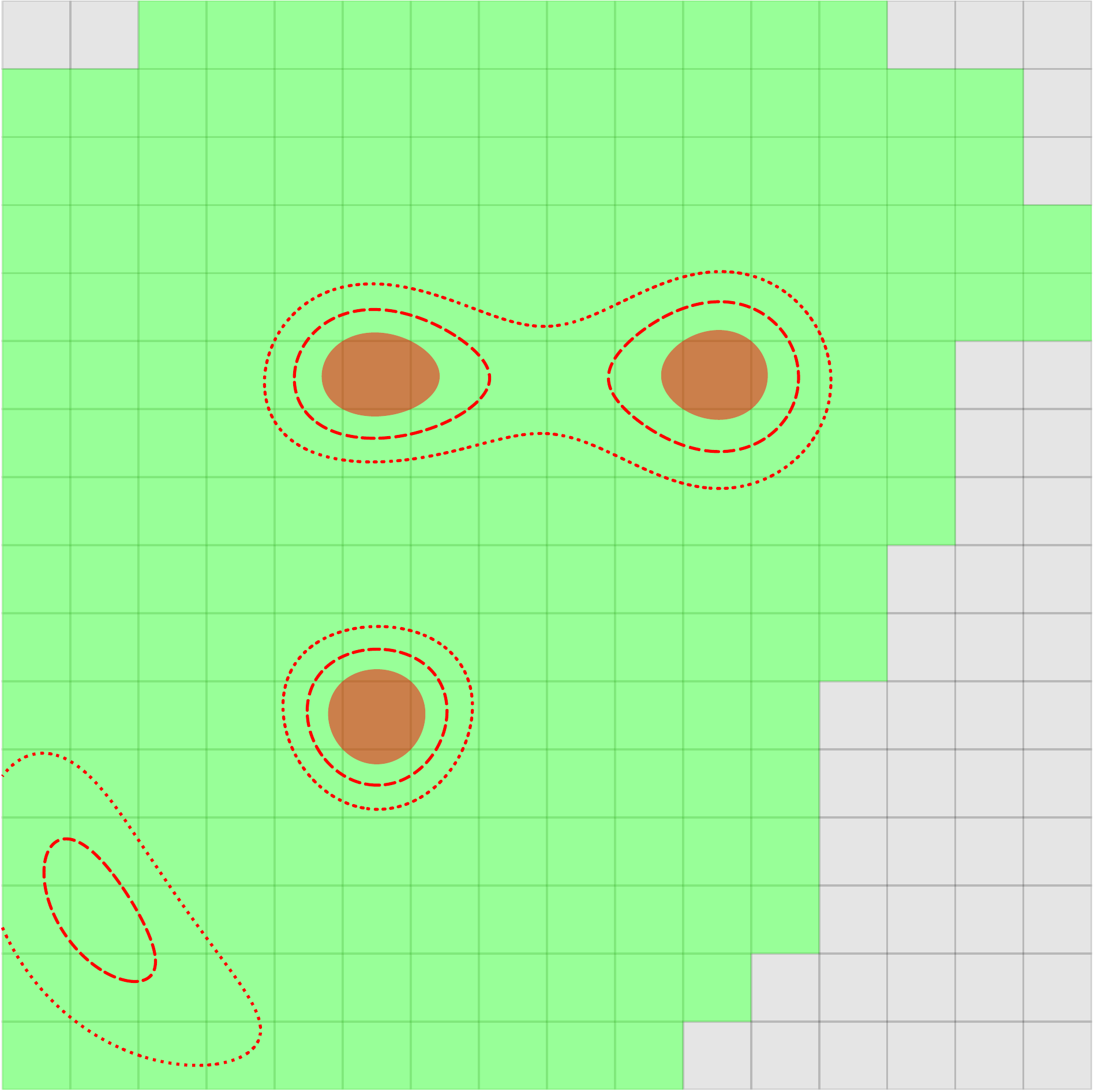}
      \caption{$|\calV_4|=289$}
    \end{subfigure}
    \begin{subfigure}[b]{0.22\textwidth}
      \includegraphics[width=\textwidth]{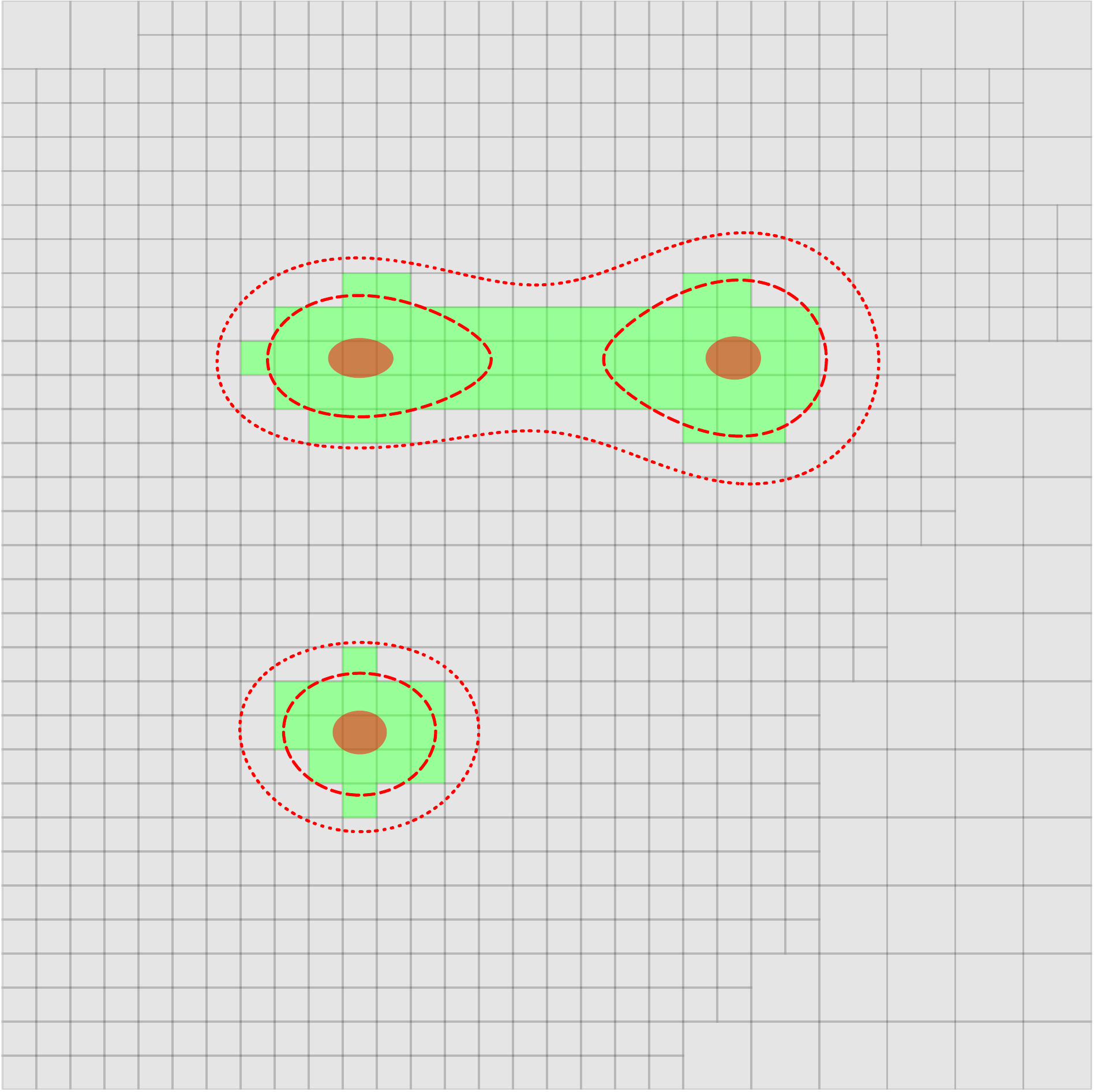}
      \caption{$|\calV_5|=951$}
    \end{subfigure}
    \begin{subfigure}[b]{0.22\textwidth}
      \includegraphics[width=\textwidth]{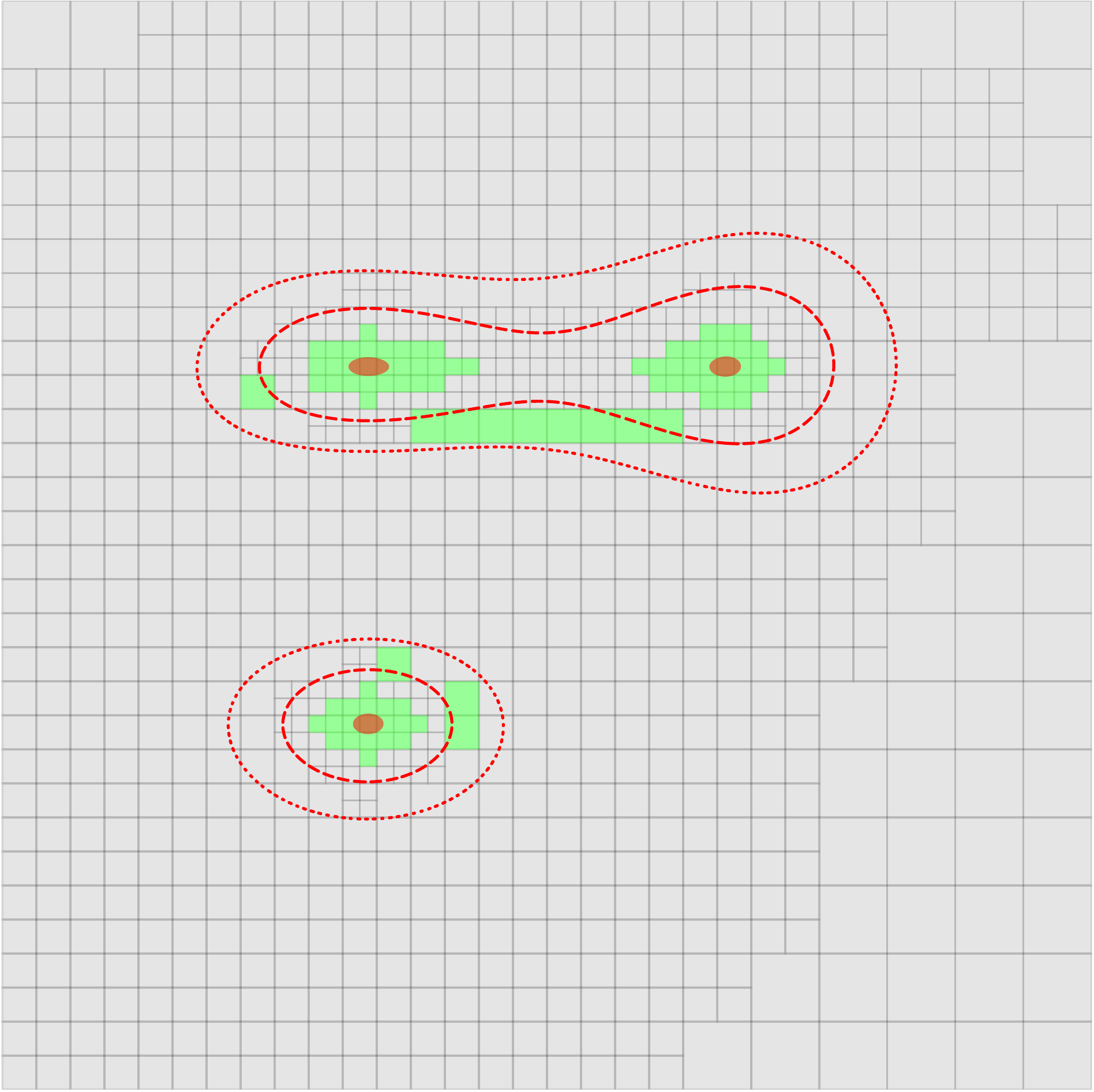}
      \caption{$|\calV_6|=1210$}
      \label{fig:2D1:6:grad}
    \end{subfigure}
    \begin{subfigure}[b]{0.22\textwidth}
      \includegraphics[width=\textwidth]{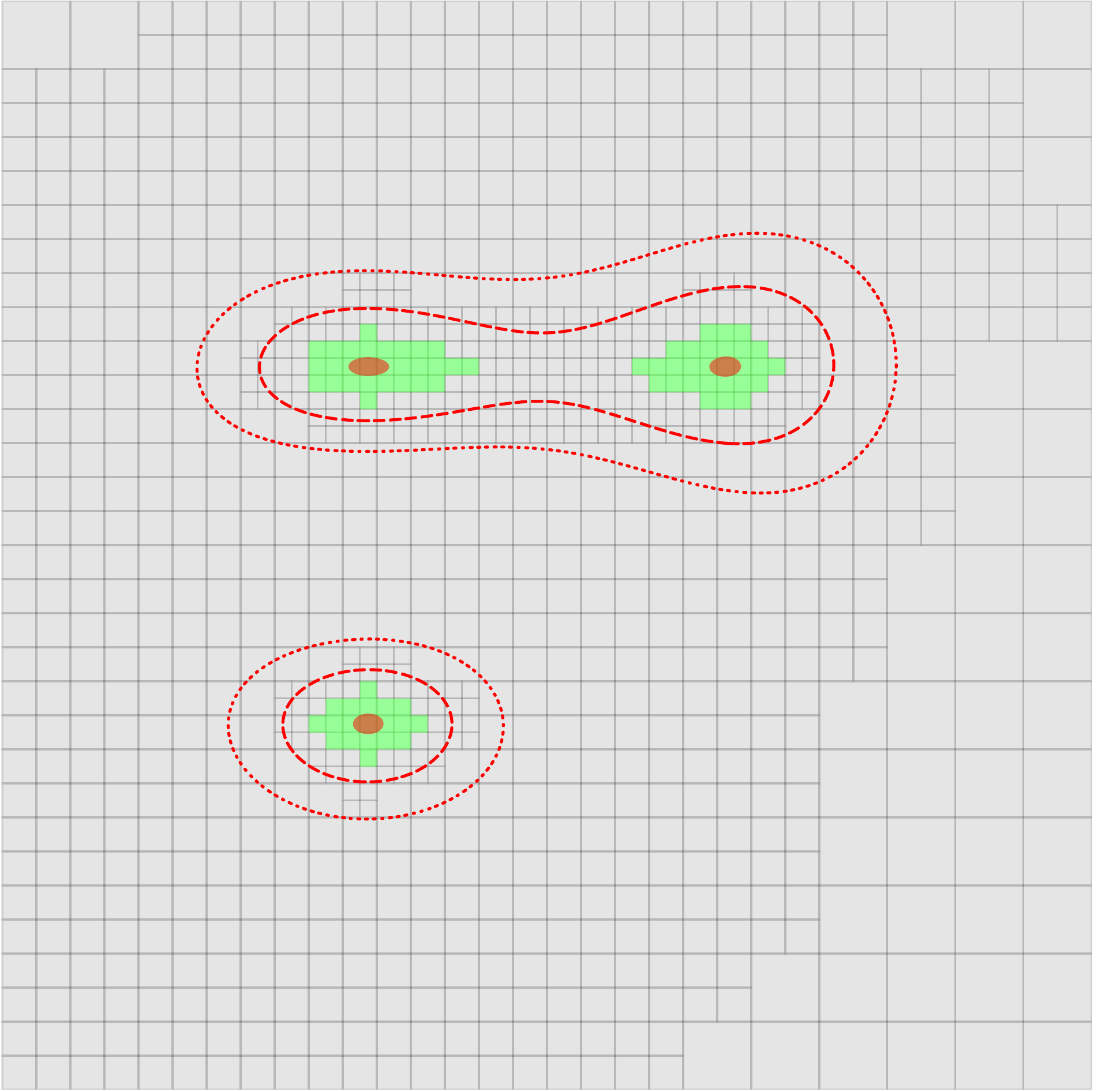}
      \caption{$|\calV_7|=1246$}
      \label{fig:2D1:7:grad}
    \end{subfigure}
    \begin{subfigure}[b]{0.22\textwidth}
      \includegraphics[width=\textwidth]{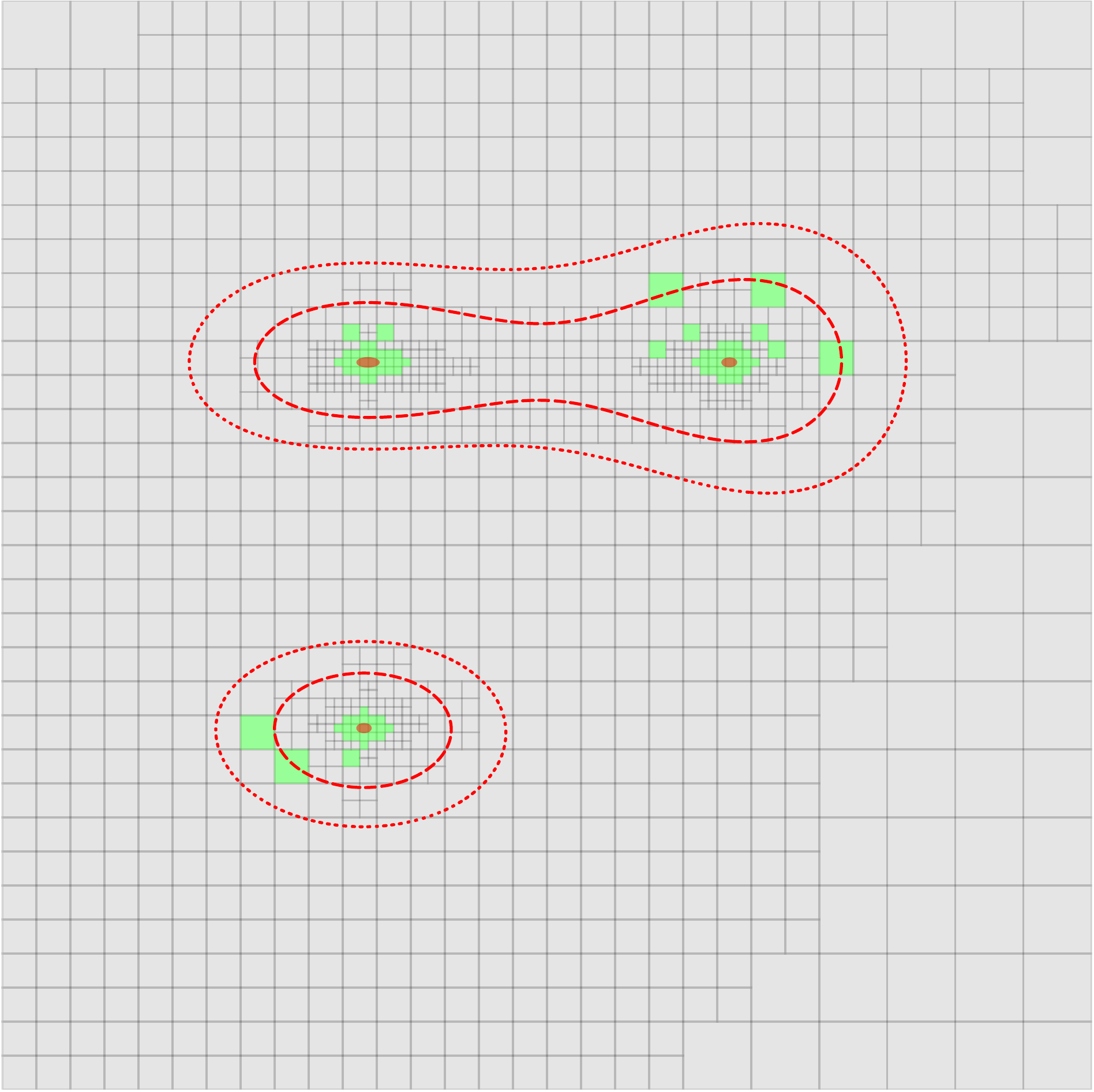}
      \caption{$|\calV_8|=1512$}
    \end{subfigure}
    \begin{subfigure}[b]{0.22\textwidth}
      \includegraphics[width=\textwidth]{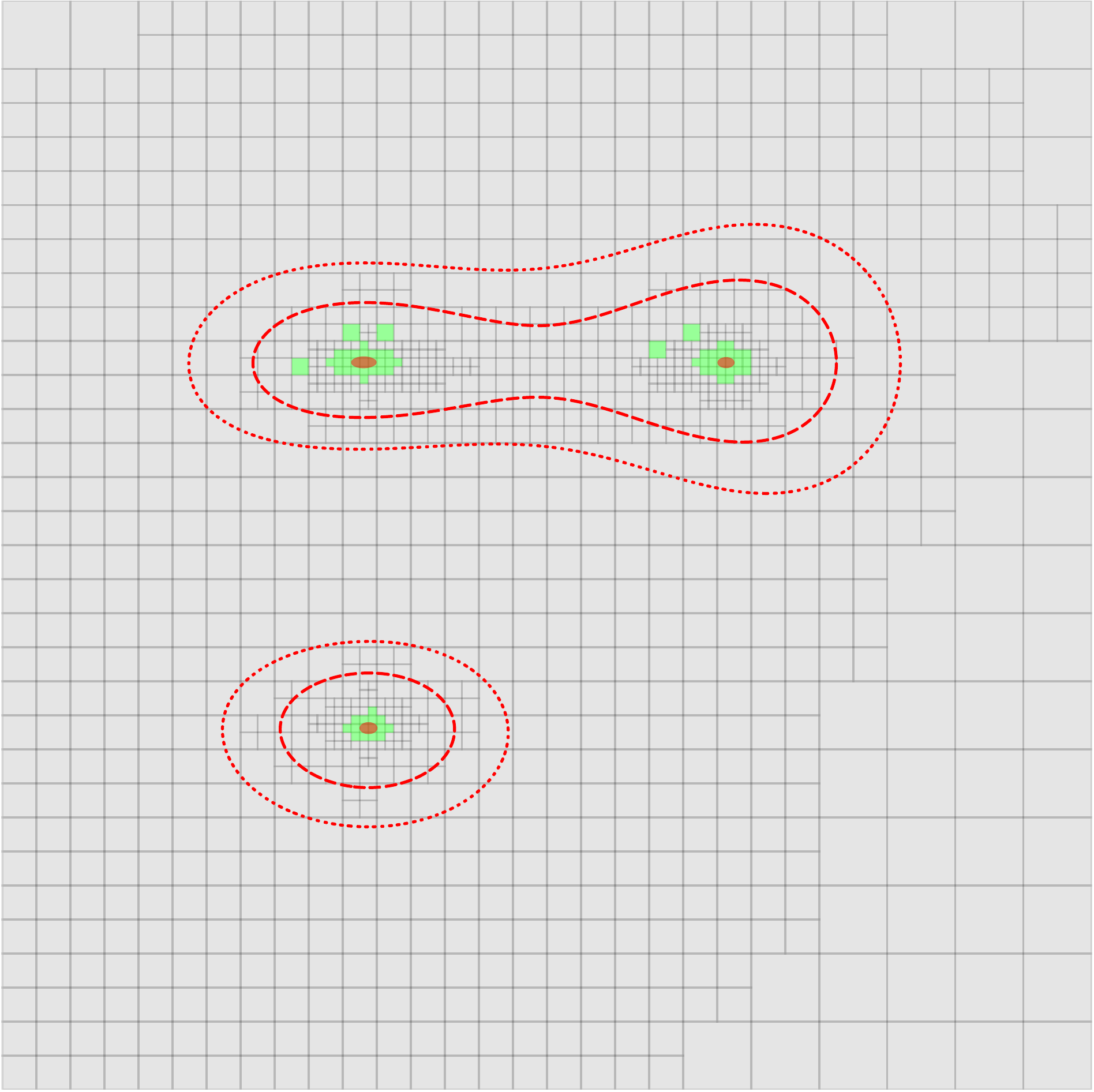}
      \caption{$|\calV_9|=1529$}
    \end{subfigure}
    \begin{subfigure}[b]{0.22\textwidth}
      \includegraphics[width=\textwidth]{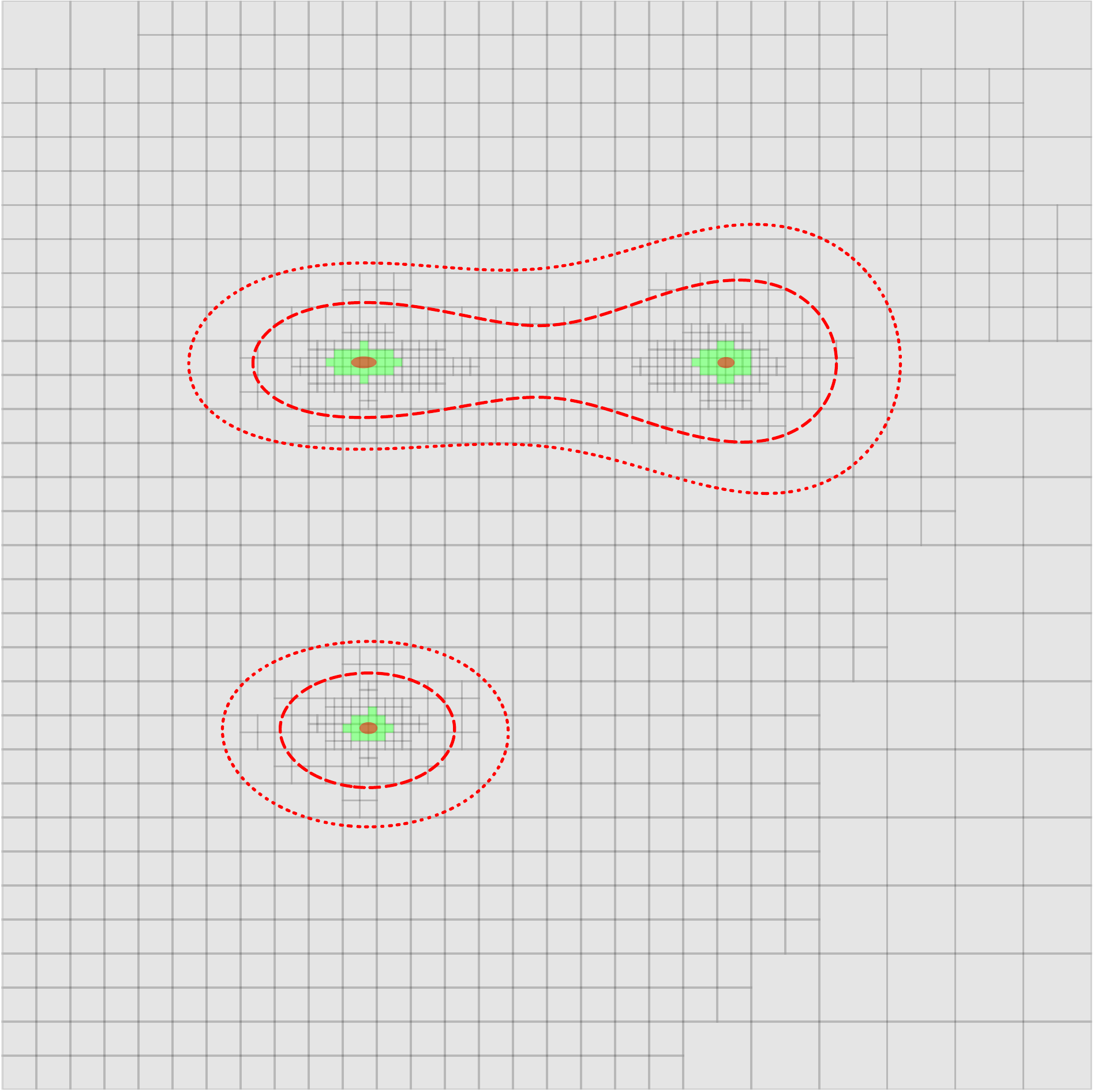}
      \caption{$|\calV_{10}|=1545$}
    \end{subfigure}
    \begin{subfigure}[b]{0.22\textwidth}
      \includegraphics[width=\textwidth]{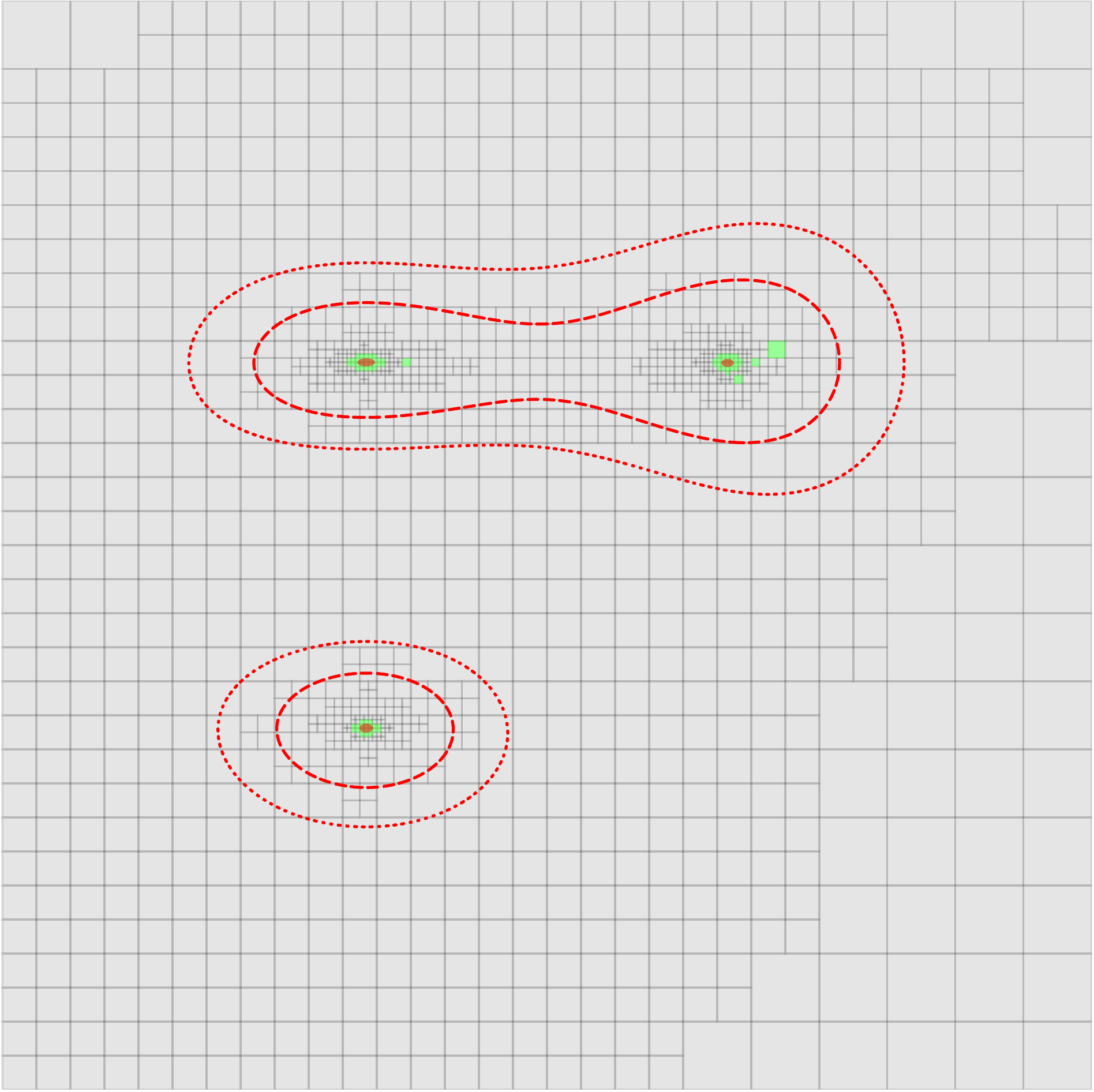}
      \caption{$|\calV_{11}|=1766$}
    \end{subfigure}      
    \begin{subfigure}[b]{0.22\textwidth}
      \includegraphics[width=\textwidth]{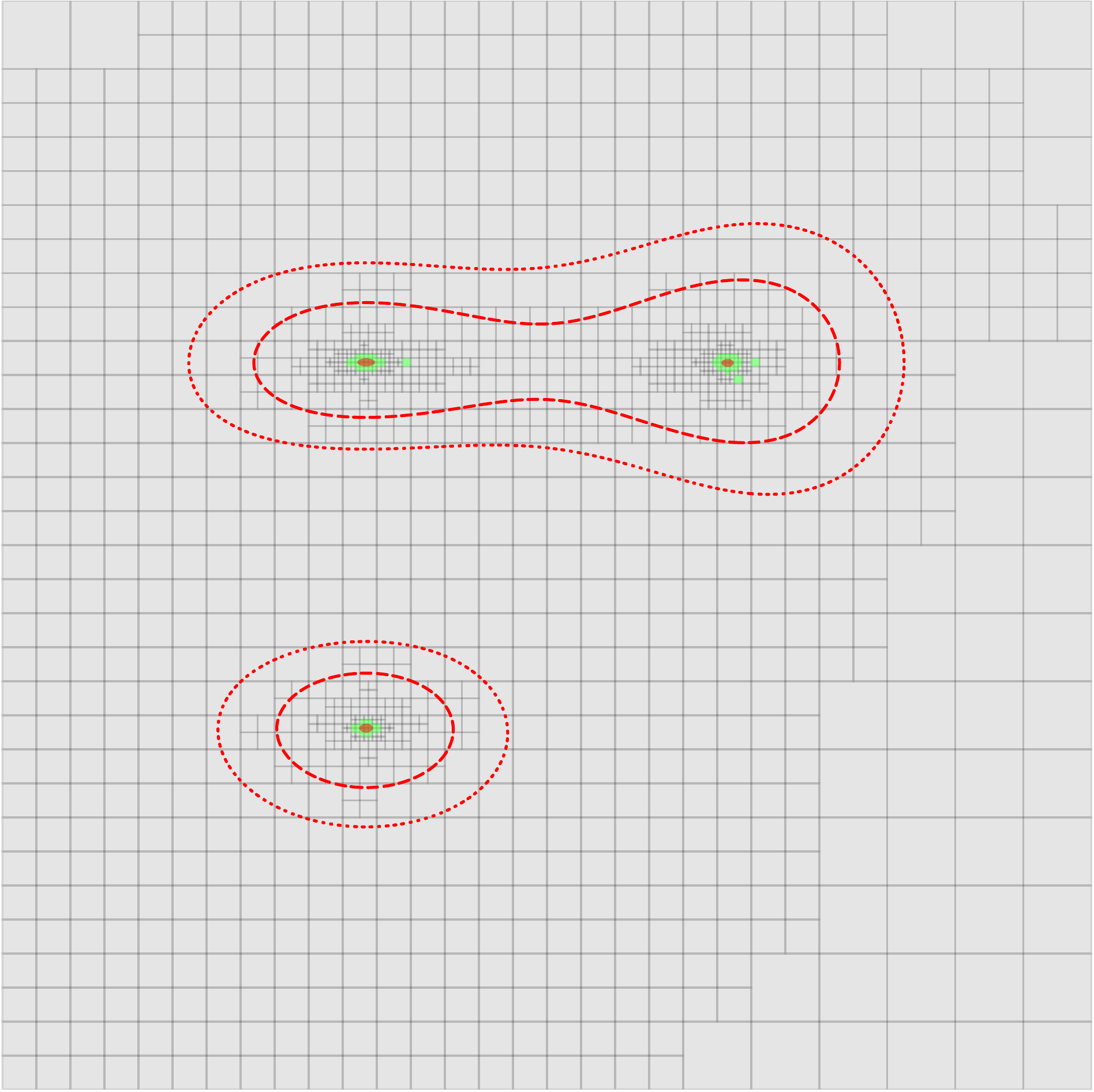}
      \caption{$|\calV_{12}|=1769$}
    \end{subfigure}
    \begin{subfigure}[b]{0.22\textwidth}
      \includegraphics[width=\textwidth]{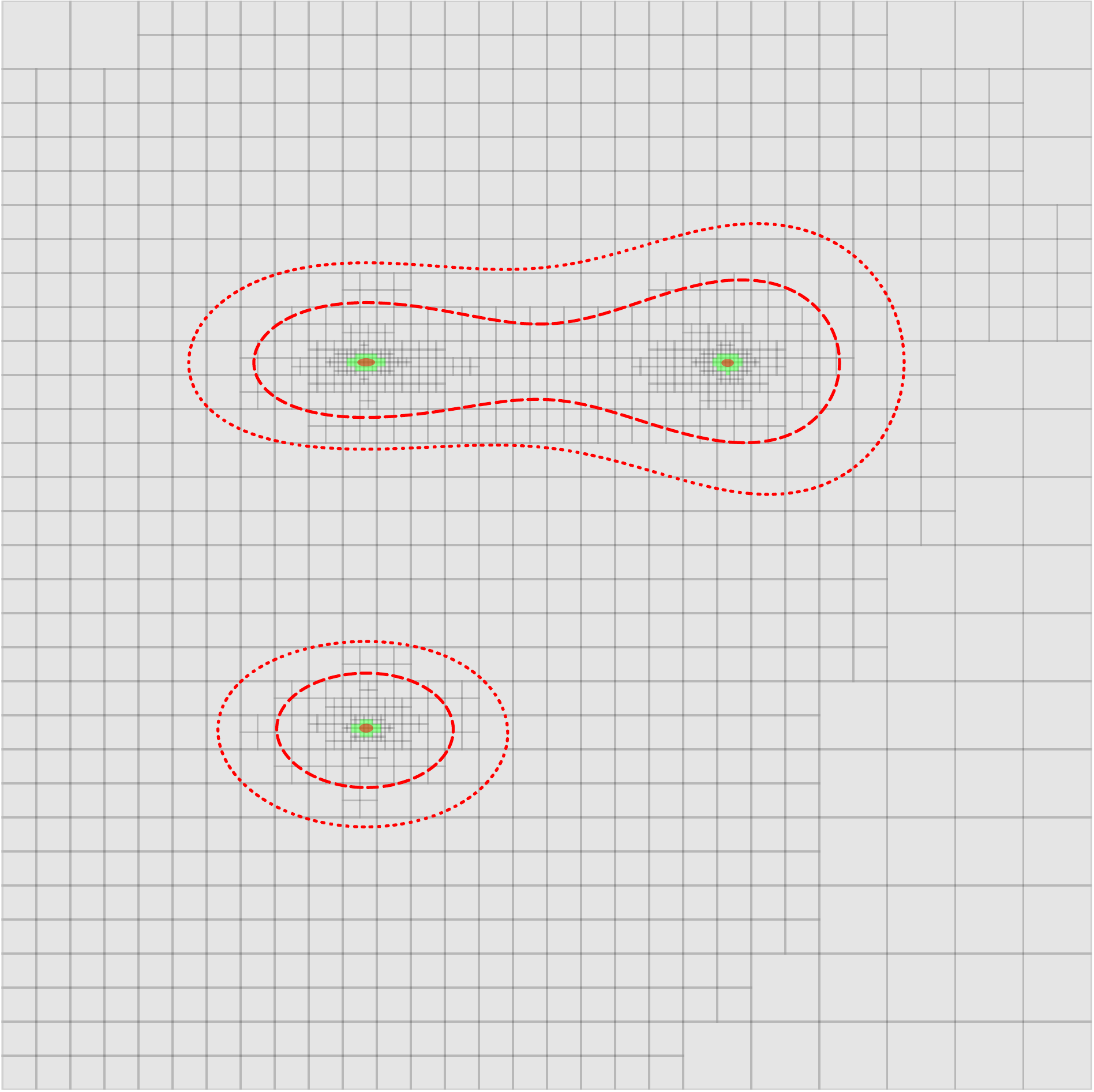}
      \caption{$|\calV_{13}|=1780$}
    \end{subfigure}
    \begin{subfigure}[b]{0.22\textwidth}
      \includegraphics[width=\textwidth]{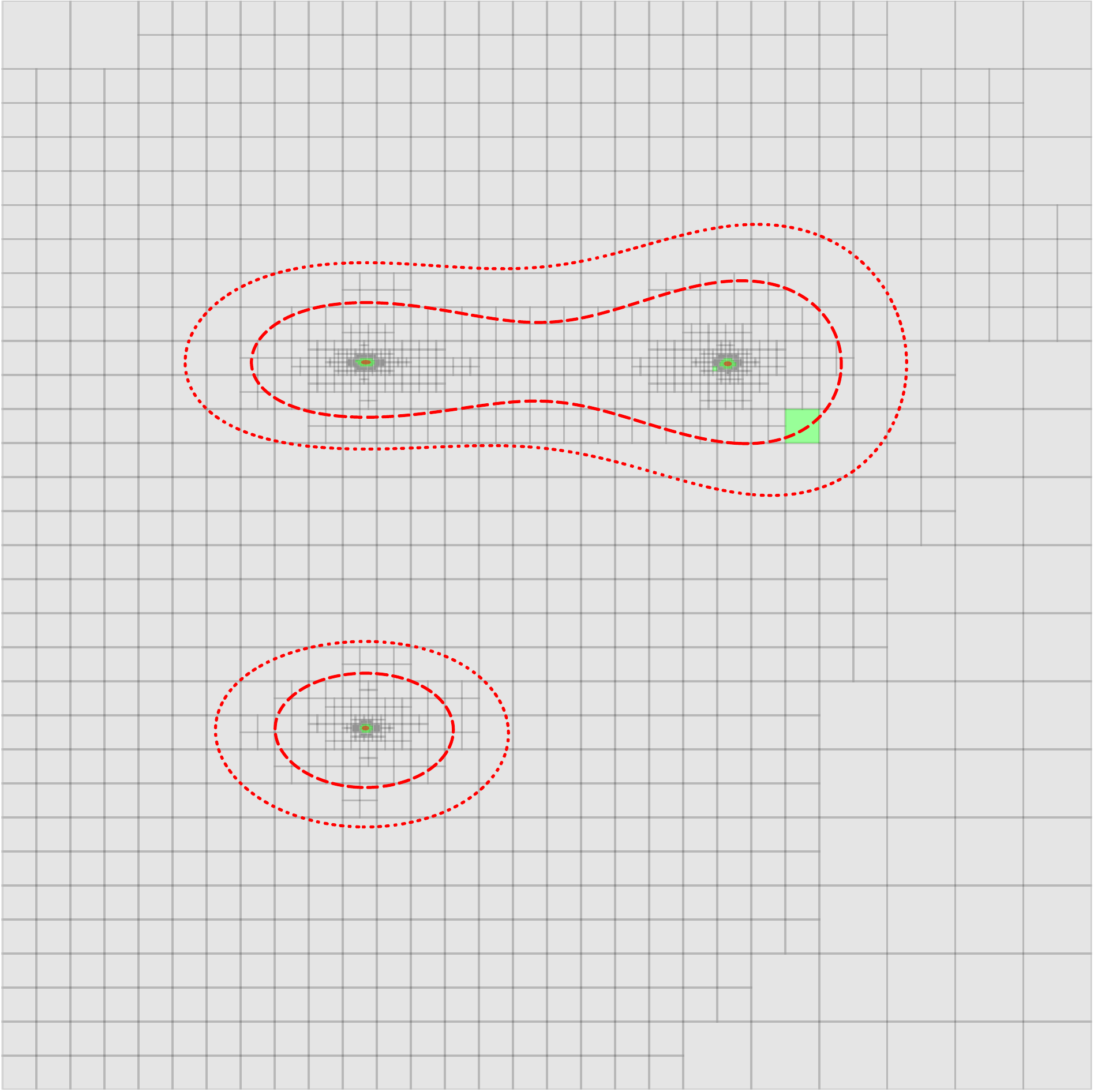}
      \caption{$|\calV_{14}|=2035$}
    \end{subfigure}
    \begin{subfigure}[b]{0.22\textwidth}
      \includegraphics[width=\textwidth]{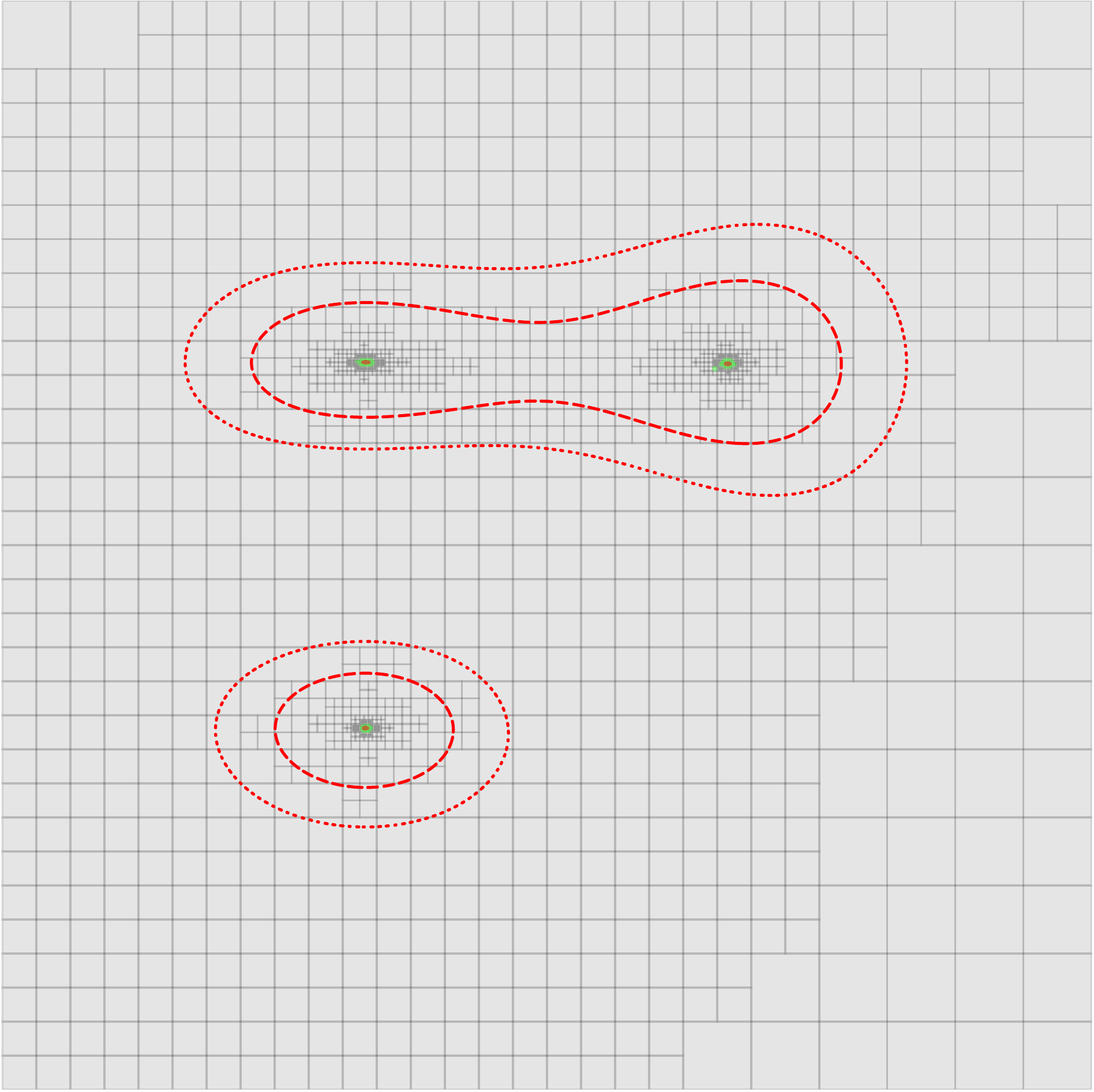}
      \caption{$|\calV_{15}|=2038$}
    \end{subfigure}
    \begin{subfigure}[b]{0.22\textwidth}
      \includegraphics[width=\textwidth]{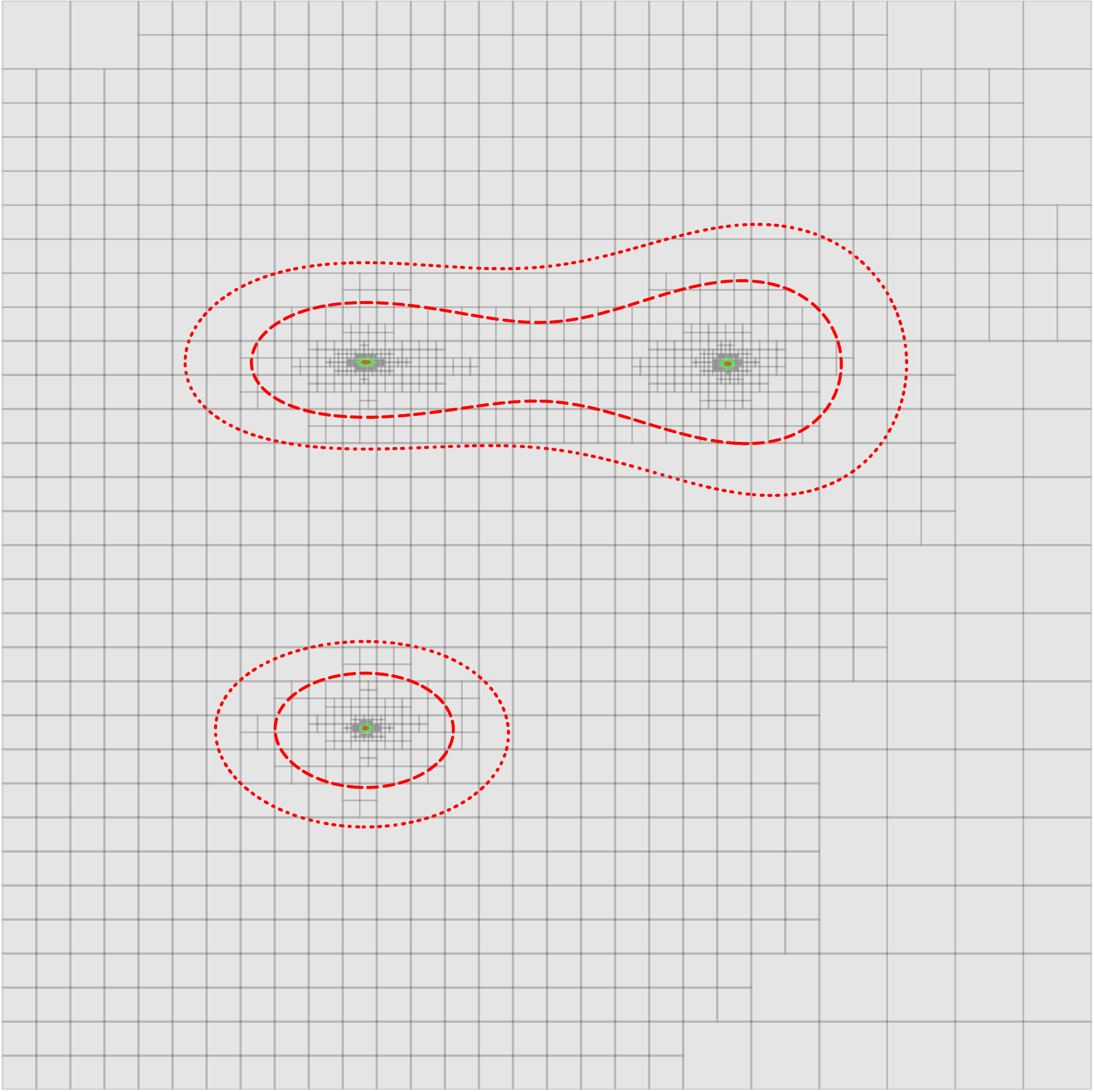}
      \caption{$|\calV_{16}|=2041$}
    \end{subfigure}
    \caption{The adaptive refinement algorithm's behavior on a 2D example for second-order selection rules with first order gradient. The behavior quadtrees are very similar to those in Figure~\ref{fig:super_resolution_2D}. \label{fig:super_resolution_2D_grad}}
  \end{figure}

\begin{table}[!ht]
\begin{subtable}[t]{0.45\textwidth}
\centering
\begin{tabular}[t]{lccccc}
\toprule
Iteration & $|\calV_k|$ & primal & $\distH(\calV_k, X^\star)$\\
\midrule
0 & 4 & 1.35942e+03 & 4.7e-01 \\
1 & 9 & 9.42990e+02 & 2.4e-01 \\
2 & 25 & 1.53313e+02 & 1.2e-01 \\
3 & 81 & 3.01429e+01 & 6.0e-02 \\
4 & 289 & 2.31285e+01 & 3.0e-02 \\
5 & 951 & 2.21082e+01 & 1.6e-02 \\
6 & 1210 & 2.19244e+01 & 7.7e-03 \\
7 & 1246 & 2.19244e+01 & 7.7e-03 \\
8 & 1512 & 2.18916e+01 & 4.6e-03 \\
9 & 1529 & 2.18955e+01 & 4.6e-03 \\
10 & 1545 & 2.18956e+01 & 4.6e-03 \\
11 & 1770 & 2.18836e+01 & 2.2e-03 \\
12 & 1773 & 2.18870e+01 & 2.2e-03 \\
13 & 1776 & 2.18870e+01 & 2.2e-03 \\
14 & 1787 & 2.18870e+01 & 2.2e-03 \\
15 & 2042 & 2.18795e+01 & 6.7e-04 \\
16 & 2045 & 2.18795e+01 & 6.7e-04 \\
17 & 2315 & 2.18778e+01 & 4.4e-04 \\
18 & 2647 & 2.18770e+01 & 2.7e-04 \\
19 & 3126 & 2.18766e+01 & 1.2e-04 \\
\bottomrule
\end{tabular}
\caption{Refinement rule with second-order bounds.\label{tab:super_resolution_2D_nograd}}
\end{subtable}
\hspace{\fill}
\begin{subtable}[t]{0.45\textwidth}
\centering
\begin{tabular}[t]{lccccc}
\toprule
Iteration & $|\calV_k|$ & primal & $\distH(\calV_k, X^\star)$\\
\midrule
0 & 4 & 1.35942e+03 & 4.7e-01 \\
1 & 9 & 9.42990e+02 & 2.4e-01 \\
2 & 25 & 1.53313e+02 & 1.2e-01 \\
3 & 81 & 3.01429e+01 & 6.0e-02 \\
4 & 289 & 2.31285e+01 & 3.0e-02 \\
5 & 951 & 2.21082e+01 & 1.6e-02 \\
6 & 1210 & 2.19244e+01 & 7.7e-03 \\
7 & 1246 & 2.19244e+01 & 7.7e-03 \\
8 & 1512 & 2.18916e+01 & 4.6e-03 \\
9 & 1529 & 2.18955e+01 & 4.6e-03 \\
10 & 1545 & 2.18956e+01 & 4.6e-03 \\
11 & 1766 & 2.18870e+01 & 2.2e-03 \\
12 & 1769 & 2.18870e+01 & 2.2e-03 \\
13 & 1780 & 2.18870e+01 & 2.2e-03 \\
14 & 2035 & 2.18795e+01 & 6.7e-04 \\
15 & 2038 & 2.18795e+01 & 6.7e-04 \\
16 & 2041 & 2.18795e+01 & 6.7e-04 \\
17 & 2318 & 2.18778e+01 & 4.4e-04 \\
18 & 2623 & 2.18770e+01 & 2.7e-04 \\
19 & 3007 & 2.18766e+01 & 1.2e-04 \\
\bottomrule
\end{tabular}
\caption{Refinement rule with second-order upper bounds and gradient lower bound.\label{tab:super_resolution_2D_grad}}
\end{subtable}
\caption{The adaptive refinement Algorithm's behavior for the 2D super-resolution problem.\label{tab:super_resolution_2D}}
\end{table}%


\section{Perspectives}

In this work, we proposed an alternative to the Frank-Wolfe algorithm for infinite dimensional total variation regularization. 
This adaptive refinement approach has a significant advantage: it does not require to search for the maximizers of a non-convex function at each iteration. Instead, it progressively discards regions of the space, in a certified manner, resembling a branch-and-bound approach. The only prerequisite to implement it is the computation of upper bounds on the largest eigenvalues of the measurement functions Hessians.
We proved that the method has great adaptivity properties. It converges generically under weak assumptions, and  its rate of convergence is linear under stronger regularity assumptions. 
To the best of our knowledge, this is as good as the best existing results for the Frank-Wolfe algorithm, and we cannot expect more from this dichotomic approach.

Despite these assets, some parts of the algorithm still require some analysis.
In particular, the solution of a finite dimensional convex problem needs to be computed at each iteration. 
In this work, we assumed that this could be achieved with an arbitrary accuracy. 
A complete theory should account for approximation errors and for the complexity of the sub-problems. 

On a more positive side, the scope of this approach is possibly significantly wider than total variation regularization. Up to some adjustements, we believe that the method could be extended to more general sparse measure optimization problems. In particular, we think of other regularizers that promote sparse solutions, such as problems defined over the cone of nonnegative measures, or over the set probability measures.

\appendix

\section{Proofs}

Here, we include the proofs ommited in the main text. We begin by proving the main results, i.e. the generic (Theorem \ref{th:generalConv}) and linear (Theorem \ref{th:linConv}) convergence results, and save the proofs of smaller, technical propositions to the end.

\subsection{Proof of Theorem \eqref{th:generalConv}\label{sec:proof_generic_convergence}}

The proof of Theorem \ref{th:generalConv} builds on similar arguments as the ones in the companion paper \cite[Thm 3.1]{FGW2019ExchangeI}, with some modifications related to the discretization and the fact that the assumptions have been slightly weakened. 

\begin{proof}[Proof of Theorem \ref{th:generalConv}]
\begin{enumerate}[label=\roman*)]
        \item \emph{Well-posedness} Under Assumptions \ref{ass:weakReg}, \ref{ass:A}, \ref{ass:coercivity} and \ref{ass:wellposedness}, we can apply Proposition \ref{prop:duality} to ensure the existence of the primal-dual pair $(\mu_0,q_0)$. For the next iterates $k\geq 1$, the measure $\mu$ in Assumption \ref{ass:wellposedness} still satisfies $\mu\in \calM(\calV_k)$ by nestedness of the sequence $(\calV_k)_{k\in \N}$. Hence we can apply Proposition \ref{prop:duality} again. 

        \item \emph{Existence of the limit of the primal solutions.} First remark that the sequence $(J(\mu_k))_{k \in \N}$ is non-increasing since the sets $\calV_k$ are nested. 
        Since $J$ is coercive, the sequence $(\norm{\mu_k}_{\calM(\Omega)})_{k\in \N}$ is bounded. Hence there exists a subsequence $(\mu_k)_{k\in \N}$, which we do not relabel, that weak-$*$ converges towards a measure $\mu_\infty$. 



\item \emph{Existence of the limit of the dual solutions.} 

Let $C_k \eqdef \left\{q\in \R^M, \|A^*q\|_{L^\infty(\calV_k)}\leq 1 \right\}$. 
The fact that $\calV_{k+1}\supseteq \calV_k$ implies that $C_{k+1}\subseteq C_k$.
Therefore any solution $q_k$ of $(\calD(\calV_k))$ belongs to $C_0$.
By Assumption \ref{ass:wellposedness}, $f^*$ is coercive on $C_0$ (see the proof \ref{prop:duality}).
Moreover, the sequence $(f^*(q_k))_{k\in \N}$ is nondecreasing, bounded above by $f^*(q^\star)$. 
Therefore all the vectors $q_k$ belong to the level set $\{q\in \R^M, f^*(q)\leq f^*(q^\star)\}$, which is bounded by coercivity of $f^*$. Up to a subsequence, $(q_k)_{k\in \N}$ converges to a limit point $q_\infty$.

\item \emph{Equicontinuity of $(A^*q_k)_{k\in \N}$}.
As another technical lemma, we prove that the set $(A^*q_k)_{k\in \N}$ is equicontinuous. 
Let $\epsilon>0$ be arbitrary. 
Since the functions $a_m \in \calC_0(\Omega)$ all are uniformly continuous, there exists a $\delta>0$ with the property
\begin{align*}
    \norm{x-y}_2 < \delta \, \Rightarrow \, \abs{a_m(x)-a_m(y)} < \frac{\epsilon}{\sup_{k} \norm{q_k}_1} \text{ for all } m.
\end{align*}
Consequently,
\begin{align}
 \norm{x-y}_2 < \delta \, \Rightarrow \, \abs{(A^*q_k)(x)- (A^*q_k)(y)} &= \abs{\sum_{m=1}^M (a_m(x)-a_m(y))q_k(m)} \leq \sum_{m=1}^M \abs{a_m(x)-a_m(y)}\abs{q_k(m)} \nonumber \\
 &< \frac{\epsilon}{\sup_{k} \norm{q_k}_1} \sum_{m=1}^M \abs{q_k(m)} \leq \epsilon. \label{eq:uniformcont}
\end{align}

    \item \emph{Feasibility of $q_\infty$.} Due to the convergence of $(q_k)_{k \in \N}$, the sequence $(A^*q_k)_{k \in \N}$ is converging strongly to $A^*q_\infty$. We will now prove that $\|A^*q_\infty\|_{L^\infty(\Omega)} \leq 1$. 
Towards a contradiction, assume that $\norm{A^*q_\infty}_{L^\infty(\Omega)} =1 + 2\epsilon$ for an $\epsilon>0$. By convergence of $(A^*q_k)$, we can conclude that there exists $k_0\in \N$ such that for $k\geq k_0$,  $\norm{A^*q_k}_{L^\infty(\Omega)} \geq 1 + \epsilon$. 
Set $\delta$ as in \eqref{eq:uniformcont}.
The set $X_k$ is  not empty and there exists a cell $\omega$ that contains a point of $X_k$ and  satisfies
$\norm{A^*q_k}_{L^\infty(\omega)} \geq 1 + \epsilon$. By Assumption~\ref{ass:generic_refinement} this cell belongs to $ \Omega_k^\star$.
It must further satisfy $\diam(\omega)\geq 2\delta$. If not, all points in $x$ have a distance to $\calV_k$ smaller than $\delta$. Since we have $|A^*q_k|(x)\leq 1$ for $x\in \vertx(\omega)$, the equicontinuity of the $A^*q_k$ implie that $\abs{A^*q_k(x)}\leq 1+\epsilon$ for all $x\in \omega$, which is a contradiction. Hence, for all $k\geq k_0$, there exists $\omega \in \Omega_k^\star$ such that $\diam(\omega)\geq 2\delta$. 
Let $(\omega_k)_{k\in \N}$ denote a sequence of refined cells in $\Omega_k$. Since we pick the active cells of largest diameter, we must have $\diam(\omega_k)\geq 2\delta$ for all $k\geq k_0$. Since all the $\omega_k$'s belong to a compact set $\Omega$, there is a finite number of cells with diameter larger than $2\delta$. Hence, we can extract a subsequence of $(\omega_k)$ which is constant. This is a contradiction, because the cells $(\omega_k)$ are refined and cannot appear twice.
    \item \emph{Convergence to a solution}. Overall, we proved that the primal-dual pair $(\mu_\infty,q_\infty)$ is feasible. It remains to prove that it is actually a solution. Here, we reproduce the argument of \cite{FGW2019ExchangeI} for completeness. Let us first remark that $\norm{\mu_\infty}_\calM + f(A\mu_\infty) \geq - f^*(q_\infty)$ by weak duality. To prove the second inequality, first notice that the weak-$*$-continuity of $A$ implies that $A\mu_k \to A\mu_\infty$. Assumption \ref{ass:weakReg} furthermore implies that $f$ is lower semi-continuous. As a supremum of linear functions, so is $f^*$. Since also $q_k \to q_\infty$, we conclude
\begin{align*}
    f^*(q_\infty) + f(A\mu_\infty) \leq \liminf_{k \to \infty} f^*(q_k) + f(A\mu_k).
\end{align*}
Assumptions \ref{ass:weakReg}, \ref{ass:A} together with Proposition \ref{prop:duality} imply exact duality of the discretized problems. This means $f^*(q_k) + f(A\mu_k) = -\norm{\mu_k}_\calM$. Since the norm is weak-$*$-l.s.c. , we thus obtain
\begin{align*}
    \liminf_{k \to \infty} f^*(q_k) + f(A\mu_k) = \liminf_{k \to \infty}  - \norm{\mu_k}_\calM \leq - \liminf_{k \to \infty} \norm{\mu_k}_\calM \leq -\norm{\mu_\infty}_\calM.
\end{align*}
Reshuffling these inequalities yields $\norm{\mu_\infty}_\calM + f(A\mu_\infty) \leq - f^*(q_\infty)$, i.e., the reverse inequality.
Thus, $\mu_\infty$ and $q_\infty$ fulfill the duality conditions, and are solutions. The final claim follows from a standard subsequence argument.
\end{enumerate}    
\end{proof}

\subsection{Proof of Theorem \eqref{th:linConv}\label{sec:proof_linear_convergence}}


In this section, we prove the main theoretical result of the paper, which is Theorem \eqref{th:linConv}.
To proceed to the final result, we begin by proving a set of intermediary results. 
The first one is a list of useful inequalities. Many of them are direct adaptions from \cite{FGW2019ExchangeI}. 
\begin{proposition}
    The following inequalities hold under Assumption \ref{ass:regularity}:
    \begin{align}
        &\|q_k - q^\star\|_2  \lesssim \distH(\calV_k | X_k) \label{eq:ineq1}\\
        &\|q_k - q^\star\|_2^2  \lesssim \max\left( \distH(X_k | X^\star), \distH(\calV_k | X^\star) \right)\cdot\distH(\calV_k | X^\star) \label{eq:ineq2} 
    \end{align}
\end{proposition}
\begin{proof}
All those inequalities come directly from \cite{FGW2019ExchangeI}. 
The first inequality \eqref{eq:ineq1} comes from Lemma 3.5, the second \eqref{eq:ineq2} from Lemma 3.6.
\end{proof}

\begin{proposition}[A list of useful inequalities\label{prop:inequalities}]
Under Assumptions \ref{ass:regularity} and \ref{ass:nondegenerate}, there exists a $k_0\in \N$ with the property that for $k\geq k_0$, the  following inequalities are true:
    \begin{align}
        \distH(X^\star | X_k) & \lesssim \|q_k - q^\star\|_2 \label{eq:ineq4}\\
         \distH(X_k | X^\star) & = \distH(X^\star | X_k) \label{eq:ineq3}\\
        \distH(X_k | X^\star) &\lesssim \distH(\calV_k | X_k) \label{eq:ineq5} \\
        \distH(\calV_k | X_k) &\asymp \distH(\calV_k | X^\star) \label{eq:ineq6} \\
        \|q_k - q^\star\|_2 & \lesssim \distH(\calV_k | X^\star) \label{eq:ineq7} \\ 
        f(\mu_k)-f(\mu^\star) & \lesssim \distH(\calV_k | X^\star)^2. \label{eq:ineq9}
    \end{align}
\end{proposition}
\begin{proof}
    The first inequalities \eqref{eq:ineq3} and \eqref{eq:ineq4} are simple consequences of Proposition 3.7 in the companion paper \cite{FGW2019ExchangeI}, together with the fact that, by the generic convergence result, $q_k$ converges to $q^\star$. 

    Inequality \eqref{eq:ineq5} is a combination of \eqref{eq:ineq1}, \eqref{eq:ineq3} and \eqref{eq:ineq4}.

    To prove inequality \eqref{eq:ineq6}, let us start by proving that $\distH(\calV_k | X_k) \gtrsim \distH(\calV_k | X^\star)$. We have by the triangular inequality
    \begin{equation*}
        \distH(\calV_k | X^\star) \leq \distH(\calV_k | X_k) + \distH(X_k | X^\star) \stackrel{\eqref{eq:ineq5}}{\lesssim}  \distH(\calV_k | X_k).
    \end{equation*}
    Let us prove the converse inequality $\distH(\calV_k | X_k) \lesssim \distH(\calV_k | X^\star)$.  
    To this end, first combine \eqref{eq:ineq2}, \eqref{eq:ineq3} and \eqref{eq:ineq4} to get
\begin{equation}
 \distH( X^\star|X_k)^2 \lesssim \max\left( \distH(X^\star | X_k), \distH(\calV_k | X^\star) \right)\cdot\distH(\calV_k | X^\star).
 \end{equation}
Regardless which of the expressions $\distH(X^\star | X_k)$ and $\distH(\calV_k | X^\star)$ is larger, this inequality yields $\distH( X^\star|X_k) \lesssim \distH(\calV_k | X^\star)$. Combining this and the triangular inequality, we get
    \begin{equation*}
        \distH(\calV_k | X_k) \leq \distH(\calV_k | X^\star) + \distH(X^\star | X_k) \lesssim  \distH(\calV_k | X^\star).
    \end{equation*}
The inequality \eqref{eq:ineq6} together with \eqref{eq:ineq6} now implies  \eqref{eq:ineq7}. Since the inequality \eqref{eq:ineq9} is a direct consequence of Proposition 3.12 in \cite{FGW2019ExchangeI}, we can conclude the proof.
\end{proof}

Let us introduce the following shorthand notation to design a neighborhood of $X^\star$ of width $r>0$:
\begin{equation*}
    \calB_r \eqdef \bigcup_{x\in X^\star} B(x,r).
\end{equation*}
\begin{proposition}[Approximate nondegeneracy of $|A^*q_k|$\label{prop:Aqknondegenerate}]
    Under Assumptions \ref{ass:regularity} and \ref{ass:nondegenerate}, there exists $k_1\in \N$, which we can assume to be larger than $k_0$, such that for all $k\geq k_1$, $A^*q_k$ satisfies the approximate nondegenerate source condition:
    \begin{enumerate}[label=\roman*)]
        \item The balls $B(x_s^\star,R)$ contain exactly one local maximizer $x_{k,s}$ of $|A^*q_k|$ for each $1\leq s\leq S$.
        \item Within these balls, $|A^*q_k|$ is strongly concave:
\begin{equation}
    |A^*q_k|''(x)\preccurlyeq -\frac{\gamma}{2}\Id, \ \forall x\in \calB_{R}.
\end{equation}
        \item Outside of these balls, we have:
\begin{equation}\label{eq:upperbound_outside}
    |A^*q_k|(x) \leq 1-\frac{\gamma R^2}{4}, \ \forall x\in \Omega \setminus \calB_{R}.
\end{equation}
    \item  Finally
    \begin{equation}\label{eq:bound_sup_Aqk}
        \sup_{x\in \Omega} |A^*q_k|(x)\leq  1 + c_2 \distH(\calV_k | X_k)^2.    
    \end{equation}
\end{enumerate}
\end{proposition}
\begin{proof}
The convergence of $(q_k)_{k\in \N}$ to $q^\star$ and  the fact that the functions $a_m\in C^2(\Omega)$ imply
\begin{equation*}
    A^*q_k\to A^*q^\star, \qquad (A^*q_k)'\to (A^*q^\star)' \quad \mbox{and} \quad (A^*q_k)''\to (A^*q^\star)'' \quad \mbox{uniformly}.
\end{equation*}
The conclusion of the three first points follows from the nondegeneracy of $A^*q^\star$ in Assumption \ref{ass:nondegenerate}.
To obtain the last, let $\omega_s$ denote the cell containing the point $x_{k,s}$ in $X_k$ closest to $x^\star_s$. We have by definition $(A^*q_k)'(x_{k,s})=0$. Let $v_s$ denote a vertex of $\omega_s$ closest to $x_{k,s}$. 
To conclude, we can use a second-order Taylor expansion with the mean-value form of the remainder. 
It reads:
\begin{equation*}
    |A^*q_k|(v_s) = |A^*q_k|(x_{k,s}) + \langle |A^*q_k|'(x_{k,s}), v_s - x_{k,s}\rangle + \frac{1}{2} \left\langle |A^*q_k|''(\xi) (v_s - x_{k,s}), (v_s - x_{k,s}) \right\rangle
\end{equation*}
for some point $\xi$ in the segment $[v,x_{k,s}]$.
By construction $|A^*q_k|(v_s)\leq 1$, $|A^*q_k|'(x_{k,s})=0$. Moreover $\left(|A^*q_k|''(\xi)\right)_{k\in \N}$ is uniformly bounded. This yields for all $s$:
\begin{equation*}
    1 \geq |A^*q_k|(x_{k,s}) - c_2 \|v_s - x_{k,s}\|_2^2 \geq |A^*q_k|(x_{k,s}) - c_2\distH(\calV_k | X_k)^2. 
\end{equation*}
Taking the maximum over the different $1\leq s \leq S$ gives the result \eqref{eq:bound_sup_Aqk}.
\end{proof}

The above proposition translates to the following result for the upper-bound. 
\begin{proposition}[Finite time behavior of the upper-bound\label{prop:behaviorupperbound}]
Under Assumptions \ref{ass:regularity} and \ref{ass:nondegenerate}, there exists $k_2 \geq k_1$ and some positive constants $c_1,c_2,c_3$ such that for all $k\geq k_2$ and for all cell $\omega$:

\begin{align}\label{eq:bound_on_sup}
    \sup_{x\in \omega} |A^*q_k|(x) \leq 
        \begin{cases}
            1 - c_1 \dist(\omega , X_k)^2 + c_2 \distH(\calV_k | X_k)^2  & \textrm{if } \dist(\omega,X^\star)\leq R, \\
            1 - c_3 R^2 & \textrm{if } \dist(\omega,X^\star)\geq R.
        \end{cases}
\end{align}

\end{proposition}
\begin{proof}


Take a cell $\omega$ with $\dist(\omega, X^\star)\geq R$. 
For $k\geq k_1$, the upper-bound \eqref{eq:upperbound_outside} is valid. 
Hence, we obtain the second bound in inequality \eqref{eq:bound_on_sup} for all $k\geq k_1$ and $c_3 = \frac{\gamma}{4}$.

To obtain the first inequality, consider a cell $\omega$ with $\dist(\omega,X^\star)\leq R$. 
Let  $s\in \llbracket 1,S\rrbracket$ denote any index such that $B(x_s^\star,R)\cap \omega \neq \emptyset$. 
Point i) in Proposition \ref{prop:Aqknondegenerate} implies the existence of a unique point $x_{k,s}$ in $X_k\cap B(x_s^\star,R)$. Proposition \ref{prop:Aqknondegenerate}, point iv) implies that $|A^*q_k(x_{k,s})| \leq 1 + c_2 \distH(\calV_k | X_k)^2$. Moreover, point ii) in Proposition \ref{prop:Aqknondegenerate} states that $|A^*q_k|$ is strongly concave in the balls $B(x_s^\star,R)$. Therefore: 
\begin{equation*}
    |A^*q_k|(x)\leq 1 + c_2 \distH(\calV_k | X_k)^2 - c_1 \dist(x,X_k)^2 \textrm{ with } c_1=\frac{\gamma}{4},\   \forall x\in B(x_s^\star,R)\cap \omega.
\end{equation*}
Using the above inequality and point iii) in Proposition \ref{prop:Aqknondegenerate} gives:
\begin{align}
    \sup_{x\in \omega \cap \calB_R} |A^*q_k|(x) &\leq  1 + c_2 \distH(\calV_k | X_k)^2 - c_1 \dist(\omega,X_k)^2, \nonumber \\
    \sup_{x\in \omega \cap \calB_R^c} |A^*q_k|(x)&\leq  1 - \frac{\gamma R^2}{4} \label{eq:strongbound}
\end{align}
We now need to show that \eqref{eq:strongbound} actually implies
\begin{align*}
|A^*q_k|(x)\leq  1 + c_2 \distH(\calV_k | X_k)^2 - c_1' \dist(\omega,X_k)^2.
\end{align*}
for $x\in \omega \cap B_R^c$. To this end, first notice that point i) in Proposition \ref{prop:Aqknondegenerate} shows that $\distH(X_k|X^\star)\leq R$. Therefore, if  $\dist(\omega,X^\star)\leq R$, we get
\begin{equation*}
\dist(\omega,X_k) \stackrel{\eqref{eq:set_triangle_inequality}}{\leq} \dist(\omega,X^\star) + \distH(X_k|X^\star) \leq 2R.     
\end{equation*}
Hence, we get $\dist(\omega,X_k)^2 \leqsim  R^2$, and by \eqref{eq:strongbound} we get 
\begin{align*}
 \sup_{x\in \omega \cap \calB_R^c} |A^*q_k|(x)&\leq  1 - \frac{\gamma R^2}{4} \leq 1- c_1'\dist(\omega,X_k)^2
\end{align*}
for some other constant $c_1'$. In particular, we get
\begin{equation*}
    \sup_{x\in \omega} |A^*q_k|(x)\leq  1 + c_2 \distH(\calV_k | X_k)^2 - c_1' \dist(\omega,X_k)^2.
\end{equation*}


\end{proof}

\begin{proposition}[Structural properties of the partitions $\Omega_k$\label{prop:structural_properties}]
For $k\geq k_2$, let $\Omega_k$ denote a partition generated by Algorithm \ref{alg:genericDisc}. 
There exists a radius $r>0$ such that any cell $\omega\in \Omega_k$ satisfies:
\begin{enumerate}[label=\roman*)]
    \item $\dist(\omega,X^\star) \geq R \Rightarrow |\omega|\geq r$.
    \item $\dist(\omega,X^\star) \leq R \Rightarrow |\omega|\geqsim \dist(\omega,X^\star)$.
    \item For $\ell<2^{-k_2}\cdot \min_{\omega \in \Omega_0} \abs{\omega}$, we have $\min_{\omega\in \Omega_k} |\omega| < \ell \Rightarrow \distH(\calV_k|X^\star)\lesssim \ell$.
\end{enumerate}
\end{proposition}
\begin{proof}[Proof of Proposition \ref{prop:structural_properties}]
Let $k_2$ be the number of iterations referenced in Proposition \ref{prop:behaviorupperbound}. It is clear that after $k_2$ iterations, all the cells have a diameter larger than $2^{-k_2}\cdot \min_{\omega \in \Omega_0}\abs{\omega}$.

Let us establish point i) first. Let $\omega$ denote a cell in $\Omega_k$ with $\dist(\omega,X^\star)\geq R$. To be refined by the algorithm, this cell needs to verify the second order approximation Assumption~\ref{ass:second_order_condition}:
\[ \Vert A^*q_k \Vert_{L^\infty(\omega)} \geq 1- \kappa |\omega|^2\]
On the other hand, Proposition \ref{prop:behaviorupperbound} states that 
\[ \Vert A^*q_k \Vert_{L^\infty(\omega)} \leq 1 -c_3R^2 \]
A necessary condition for refinement by Algorithm \ref{alg:genericDisc} is then  $|\omega|\ge cR$ for some $c$. Taking $r=\min\left(\frac{c}{2}R,2^{-k_2}\cdot \min_{\omega \in \Omega_0}\abs{\omega}\right)$ proves i).

Point ii) is more technical. It follows from the following arguments. 
\begin{enumerate}
    \item Suppose that the partition $\Omega_k$ contains a cell $\omega$ with an edge-length $\abs{\omega}\eqdef \ell$ with $\ell < 2^{-k_2}\cdot\min_{\omega \in \Omega_0} \abs{\omega}$. 
    \item The parent cell $\omega_p$ of $\omega$ must have been refined in some iteration $k_\ell$ before $k$ but after $k_2$. For this $k_{\ell}$, the size of the parent cell is $|\omega_p|=2\ell$. 
    \item Since we refine the largest cells in $\Omega_{k_\ell}^\star$, it means that every cell in $\Omega_{k_{\ell}}^\star$ has an edge-length smaller or equal than $2\ell$. By construction, every cell that contains an element of $X_k$ is in $\Omega_{k_{\ell}}^\star$, hence 
    \begin{equation*}
    \distH(\calV_k  | X_k) \leqsim \ell.
    \end{equation*}
    \item By Proposition \ref{prop:inequalities}, equation~\eqref{eq:ineq6}, we have $\distH(\calV_k|X_k)\asymp \distH(\calV_k|X^\star)$. We get $\distH(\calV_k|X^\star) \lesssim \ell$ for $k=k_\ell$  . By monotonicity of the sequence $(\distH(\calV_k|X^\star))_{k\in \N}$, we also get the inequality for all $k\geq k_{\ell}$.
    \item Noticing that $\dist(\omega_p,X^*)\leq \dist(\omega,X^*)\le R$, we apply Proposition \ref{prop:behaviorupperbound} to $\omega_p$ to get
    \[
         \Vert A^*q_k \Vert_{L^\infty(\omega_p)} \leq 1 - c_1 \dist(\omega_p,X_k)^2 + c_2\ell^2.
    \]
     Since $\omega_p$ is refined at iteration $k_\ell$, it belongs to $\Omega_{k_\ell}^\star$ and needs to verify the second order approximation Assumption~\ref{ass:second_order_condition}:
\[ \Vert A^*q_k \Vert_{L^\infty(\omega_p)} \geq 1- \kappa |\omega_p|^2\geq 1- 4\kappa \ell^2.\]
    From the two previous inequalities, we get $ \dist(\omega_p,X_k) \lesssim \ell$.
    \item  To conclude, remark that 
    \begin{align*}
        \ell &\gtrsim \dist(\omega_p,X_k) \\ 
             &\stackrel{\eqref{eq:set_triangle_inequality}}{\geq}  \dist(\omega,X_k) - \distH(\omega|\omega_p) \geq \dist(\omega,X_k) - \sqrt{D}\ell \\ 
             & \stackrel{\eqref{eq:set_triangle_inequality}}{\geq} \dist(\omega,X^\star) - \distH(X_k|X^\star) - \sqrt{D}\ell \geq \dist(\omega,X^\star) - c_5 \ell,
    \end{align*}
    for some $c_5>0$. This proves ii). 
\end{enumerate}
For point iii), repeat the first four arguments of point ii) to the cell that achieves $\min_{\omega\in \Omega_k} |\omega|$.

\end{proof}

Now that we established the geometrical structure of $\Omega_k$, the remaining task is to count the number of cells in $\Omega_k$.
\begin{proposition}[Counting cells\label{prop:counting_cells}]
Let $\Omega_k$ denote a cell partition generated by Algorithm \ref{alg:genericDisc}. 
Assume that 
\begin{equation*}
\min_{\omega\in \Omega_k} |\omega| =  2^{-J}
\end{equation*}
for some $J\in \N$ with $J\geq k_2$ in Proposition \ref{prop:structural_properties}. Then the number of cells in $\Omega_k$ satisfies:
\begin{equation*}
        |\Omega_k|\leq c_0 + c_1 S J
\end{equation*}
for some constants $c_0,c_1>0$ independent of $J$ and $S$.
\end{proposition}
\begin{proof}
We decompose $\Omega_k$ as $\bigcup_{s=0}^S \Omega_{k,s}$ with 
\begin{align*}
    \Omega_{k,0} &\eqdef \{\omega \in \Omega_k, \dist(\omega,X^\star)> R\}, \\
    \Omega_{k,s} & \eqdef \{\omega \in \Omega_k, \dist(\omega,x_s^\star) = \dist(\omega,X^\star), \dist(\omega,X^\star)\leq R\} \quad s\in \llbracket 1,S \rrbracket.
\end{align*}
In words, $\Omega_{k,s}$ is the set of cells in $\Omega_k$ closest to $x_s^\star$, and at a distance smaller than $R$ from $X^\star$. We have
\begin{equation}\label{eq:total_recall}
    |\Omega_k| \leq \sum_{s=0}^S |\Omega_{k,s}|.
\end{equation}

We first use point i) in Proposition \ref{prop:structural_properties} to control $|\Omega_{k,0}|$. It states that all cells in $\Omega_{k,0}$ have an edge-length larger than $r$. 
The volume of a cell of edge-length $r$ is $r^D$. Since all the cells are disjoint and contained in $\Omega$, we get 
\begin{equation*}
|\Omega_{k,0}| \leq \vol(\Omega)/r^D = r^{-D}.
\end{equation*}

Now let us derive a bound for $|\Omega_{k,s}|$. 
Let $\omega(x)$ denote the cell in $\Omega_k$ containing $x$ and $\ell\eqdef 2^{-J}$.
By assumption $|\omega(x)|\geq \ell$ for all $x\in \Omega$. 
Moreover for all $x\in \Omega$ such that $\omega(x)\in \Omega_{k,s}$ we have by Proposition \ref{prop:structural_properties}, point ii) $|\omega(x)|\geq c \dist(\omega(x), x_s^\star)$. Therefore 
\begin{equation*}
    |\omega(x)| \gtrsim \dist(\omega(x),x_s^\star) \stackrel{\eqref{eq:set_triangle_inequality}}{\geq} \dist(x,x_s^\star) - \distH(x|\omega(x)) \geq \|x-x_s^\star\|_2 - \sqrt{D}|\omega(x)|. 
\end{equation*}
This gives $|\omega(x)|\gtrsim \|x-x_s^\star\|_2$ for any $x \in \bigcup_{\omega\in \Omega_{k,s}} \omega$. Combining the two inequalities yields
\begin{equation*}
    |\omega(x)| \geq \max(\ell,  c\|x-x_s^\star\|_2), \quad \forall x\in \bigcup_{\omega\in \Omega_{k,s}} \omega.
\end{equation*}
for some $c\geq 0$. For each cell $\omega$, we have $|\omega|^{D}=\int_{\omega} \,dx$.
We continue as follows
\begin{eqnarray*}
|\Omega_{k,s}| &=&\sum_{\omega\in \Omega_{k,s} } 1= \sum_{\omega\in \Omega_{k,s} } \int_\omega |\omega(x)|^{-d}\,dx=
\int_{\bigcup_{\omega\in \Omega_{k,s}} \omega} |\omega(x)|^{-D}\,dx \\
& \leq & \int_{\bigcup_{\omega\in \Omega_{k,s}} \omega} \max(\ell,  c\|x-x_s^\star\|_2)^{-D} \,dx \\
& \leq  &\int_{\Omega} \max(\ell,  c\|x-x_s^\star\|_2)^{-D} \,dx \\
& \leq  &\int_{B(x_s^\star,\sqrt{D})} \max(\ell,  c \|x-x_s^\star\|_2)^{-D} \,dx  \\
& = &\int_{B(x_s^\star,\ell/c)} \ell^{-d} \,dx + \int_{B(x_s^\star,\sqrt{D})\setminus B(x_s^\star,\ell/c)} (c \|x-x_s^\star\|_2)^{-D} \,dx \\
&\lesssim & 1 + \int_{\rho=\ell/c}^{\sqrt{D}}\rho^{-D} \rho^{D-1} d\rho 
\lesssim 1 + |\log_2(\ell)| \lesssim 1+J.
\end{eqnarray*}
Summing up everything, we obtain $|\Omega_k| \lesssim c_0 + c_1 SJ$ for some constants $c_0,c_1\geq 0$.
\end{proof}

We now gathered all the necessary ingredient to prove the complexity result.

\begin{proof}[Proof of Theorem \ref{th:linConv}]
Take $J\geq k_2$. 
The Algorithm terminates whenever a cell of size $2^{-(J+1)}$ has to be refined.
When it stops, all the cells therefore have a size larger than $2^{-(J+1)}$ by construction and $\min_{\omega \in \Omega_k} |\omega| = 2^{-(J+1)}$.

Proposition \ref{prop:counting_cells} therefore indicates that $|\Omega_k|\leq c_0 + c_1 S J$. 
Since at least one cell is refined per iteration, we reached the termination criterion for a number of iterations $k\leq c_0 + c_1 S J$. Point iii) in Proposition \ref{prop:structural_properties} allows us to conclude that $\distH(\calV_k|X^\star)\lesssim 2^{-J}$. The list of inequalities in Proposition \ref{prop:inequalities} yield the conclusion.
\end{proof}

\subsection{Further proofs}
Here, we collect proofs of the remaining, smaller and more technical propositions. 
\subsubsection{Proof of Proposition \ref{prop:triangle_set}}
\begin{proof}
    For any $x_1\in X_1$, $x_2\in X_2$, $x_3\in X_3$, we have $\|x_1-x_2\|_2\leq\|x_1-x_3\|_2+\|x_3-x_2\|_2$. 
    Taking the infimum over $x_1\in X_1$ and the infimum over $x_2\in X_2$ yields
    \begin{align*}
        \dist(X_1,X_2) &\leq \inf_{x_1\in X_1}\|x_1-x_3\|_2 + \inf_{x_2\in X_2} \|x_3-x_2\|_2
                   \leq \sup_{x_3\in X_3}\inf_{x_1\in X_1}\|x_1-x_3\|_2 + \inf_{x_2\in X_2} \|x_3-x_2\|_2 \\
                   & = \distH(X_1|X_3) + \inf_{x_2\in X_2} \|x_3-x_2\|_2.
    \end{align*}
    Taking the infimum over $x_3\in X_3$, we obtain the claimed result.
\end{proof}

\subsubsection{Proof of Proposition \ref{prop:duality}}

\begin{proof}
Under Assumptions  \ref{ass:A} and \ref{ass:weakReg}, the function $J$ is lower semi-continuous for the weak-* topology. 
The existence of a measure $\mu$ supported on $\calV$ with $J(\mu)<+\infty$ and the coercivity of $J$ therefore ensures the existence of a primal solution. We then invoke Theorem \cite[9.8.1]{attouch2014variational} to conclude on the existence of a dual solution, the extremality relationships and on the fact that there is no duality gap. 

For the boundedness of the primal solution set in total variation norm, it suffices to use the fact that $J$ is coercive, ensuring boundedness of its sub-level sets.

Now let us prove the boundedness of the dual solution set. To this end, notice that by convexity, $f$ is continuous at any point in $\interior(\dom(f))$. In particular, $f$ is continuous at $A\mu$. Using Proposition 1.3.9 in \cite{urrutyII}, we conclude that 
\begin{equation}
   g^*(q) \eqdef  f^*(q) - \langle A\mu, q\rangle 
\end{equation}
is coercive. 
We have $f^*(q) \geq g^*(q) - \|\mu\|_{\calM(\calV)} \|A^*q\|_{L^\infty(\calV)}$.
Hence $f^*$ is coercive on the admissible set for which $f^*(q) \geq g^*(q) - \|\mu\|_{\calM(\calV)}$. This ensures the boundedness of the dual solution set.
\end{proof}

\subsubsection{Proof of Proposition \ref{prop:first_order_selection}}

\begin{proof}
Any choice of $\kappa_1(q_k,\omega)$ in \eqref{eq:first order} satisfying the Lipschitz inequality 
\begin{equation*}
    \sup_{x_1,x_2\in \omega} \frac{\Bigl | \left|A^*q_k(x_1)| - |A^*q_k(x_2)\right|\Bigr|}{\|x_1-x_2\|_2} \leq \kappa_1(q_k,\omega)
\end{equation*}    
also satisfies $\Uk(\omega) \geq \|A^*q_k\|_{L^\infty(\omega)}$, i.e. Assumption \ref{ass:generic_refinement}. 
We have 
\begin{equation*}
\sup_{x_1,x_2\in \omega} \frac{\Bigl| |A^*q_k(x_1)| - |A^*q_k(x_2)|\Bigr|}{\|x_1-x_2\|_2} = \sup_{x\in \omega} |A^*q_k|'(x),    
\end{equation*}
where we consider that $|A^*q_k|'(x)=0$ on the points of non differentiability of $|A^*q_k|$.  
To obtain the expression \eqref{eq:def_kappa1_loc}, we use a Hölder inequality:
\begin{equation}\label{eq:bound_local_k1}
    \sup_{x\in \omega} |A^*q_k|'(x) =  \sup_{x\in \omega} \left|\sum_{m=1}^M q_k[m] a_m'(x)\right| \leq \sum_{m=1}^M | q_k[m] | \sup_{x\in \omega} |a_m'|(x). 
\end{equation}

Showing that Assumption \ref{ass:second_order_condition} is not always valid stems from the fact that we use a $0$-th order Taylor expansion, with a remainder that is therefore of first order only.

\end{proof}

\subsubsection{Proofs of Proposition \ref{prop:second_order_selection} and Proposition \ref{prop:second_order_selection_with_grad}}

Our proofs rely on the following well-known (see e.g. \cite[Lemma 1.2.3, 1.2.4]{nesterov2003introductory}) statements about Taylor-expansions. If $f: C \to \R$ is a function on a convex domain $C\sse \R^n$ with a $\kappa_2$-Lipschitz continuous gradient, we have for $x,y$ arbitrary
\begin{align}
	\abs{f(x)-f(y)- \sprod{f'(y),x-y}} &\leq \kappa_2 \frac{\norm{x-y}^2_2}{2}\label{eq:taylor1} \\
	\|f'(x)-f'(y)\|_2 &\leq \kappa_2 \norm{x-y}_2\label{eq:taylor2grad}.
\end{align}

The value $\kappa_2(q_k,\omega)$ is an upper-bound on the Lipschitz constant of $(A^*q_k)'$ restricted to $\omega$. Indeed, 
\begin{equation}
    \sup_{x\in \omega} \|(A^*q_k)''(x)\|_{2\to 2} =  \sup_{x\in \omega} \left\|\sum_{m=1}^M q_k[m] a_m''(x)\right\|_{2\to 2} \leq \sum_{m=1}^M | q_k[m] | \sup_{x\in \omega} \|a_m''\|_{2\to 2}(x). 
\end{equation}

\paragraph{Proof of Proposition \ref{prop:second_order_selection}}
\begin{proof}
Now, let us prove that
\begin{equation}\label{eq:bound_order_2}
    \Uk(\omega) - \kappa_2(q_k,\omega) \diam(\omega)^2 \leq  \|A^*q_k\|_{L^\infty(\omega)} \leq \Uk(\omega).
\end{equation}
By equation \eqref{eq:taylor1}, we get  
    \begin{equation*}
        |A^*q(x)- A^*q(v) - \langle (A^*q)'(v), x-v\rangle| \le  \kappa_{2}(q_k,\omega) \frac{\|x-v\|_2^2}{2}.
    \end{equation*}
It follows
    \begin{equation*}
        |A^*q(x)| \leq |A^*q(v) + \langle (A^*q)'(v), x-v\rangle| +\kappa_{2}(q_k,\omega)  \frac{\|x-v\|_2^2}{2},
    \end{equation*}
    Taking first the supremum in $x$ and then the infimum over the vertices $v$ yields the right-hand side of \eqref{eq:bound_order_2}.
    We also have
    \begin{align*}
        |A^*q(x)| &\geq |A^*q(v) + \langle (A^*q)'(v), x-v\rangle| -  \kappa_{2}(q_k,\omega) \frac{\diam(\omega)^2}{2} \\
        &\geq |A^*q(v) + \langle  (A^*q)'(v), x-v\rangle| +\kappa_{2}(q_k,\omega) \frac{\|x-v\|_2^2}{2} -  \kappa_{2}(q_k,\omega)\diam(\omega)^2.
    \end{align*}
    Again, taking the supremum in $x$ and then the infimum over the vertices $v$ we obtain the left-hand side of \eqref{eq:bound_order_2}.
    The right hand-side of \eqref{eq:bound_order_2} proves that the second order selection process satisfies Assumption \ref{ass:generic_refinement}, whereas the left hand-side proves that it obeys Assumption \ref{ass:second_order_condition}.
\end{proof}


\paragraph{Proof of Proposition \ref{prop:second_order_selection_with_grad}}
\begin{proof}
We have for all $v \in \vertx{\omega}$ and $x\in \omega$
\begin{align*}
	 \norm{(A^*q_k)'(x) - (A^*q_k)'(v)}_2  \leq \kappa_2(q_k,\omega) \norm{x-v}_2 ,
\end{align*}
which implies
\begin{align*}
	\norm{(A^*q_k)'(x)} \geq  \norm{(A^*q_k)'(v)}_2  - \kappa_2(q_k,\omega) \norm{x-v}_2 \
	\Longrightarrow \ \inf_{x\in \omega} \norm{(A^*q_k)'(x)} \geq \norm{(A^*q_k)'(v)}_2  - \kappa_2(q_k,\omega) \diam(\omega).
\end{align*}
Since this is true for every $v \in \vertx(\omega)$, we get 
\begin{equation*}
    \Gk(\omega) \leq \inf_{x \in \omega} \norm{\nabla A^*q_k(x)}_2,
\end{equation*}
In other words, $\Gk$ is a lower bound of $\Vert (A^*q_k)'\Vert $.
Since $\kappa_2$ is an upper bound of the Lipschitz constant of $\Vert (A^*)'q_k\Vert $, $\Gk(\omega)$ is a lower bound of $\Vert (A^*q_k)'\Vert $. It follows that any cell $\omega$ that contains a point of $X_k$ will verify both $\Uk(\omega)\ge 1$ and $\Gk(\omega) \le 0$. Hence $\Omega_k^\star$ verifies Assumption \ref{ass:generic_refinement}. It is clear that the $\Omega_k^\star$ of Definition~\ref{def:second_order_selection_with_grad} is included in the $\Omega_k^\star$ of Definition~\ref{def:second_order_selection}. Because the latter verifies Assumption~\ref{ass:second_order_condition}, so does the former.
\end{proof}

%


\subsubsection{Proof of Proposition \ref{prop:k2_gaussian_2D}}

\begin{proof} We have
  \begin{align*}
      a_m'(x) &= -a_m(x) \left[\frac{x-z_m}{\sigma^2}\right] \\
    a_m''(x) &= a_m(x) \left[ -\frac{\sigma^2}{\sigma^4}\Id + \frac{(x-z_m)(x-z_m)^T}{\sigma^4}\right].  
 \end{align*}

 For any $u\in \R^D$ and any $x\in \Omega$, we have
 \begin{align*}
     \bigl|\langle a_m''(x) u, u \rangle\bigr|  & =\frac{a_m(x)}{\sigma^4}  \bigl| -\sigma^2 \|u\|_2^2 + \langle u,x-z_m\rangle^2 \bigr| \\ 
     & \leq \frac{a_m(x)}{\sigma^4}\max(\sigma^2,\|x-z_m\|_2^2)\|u\|_2^2.
 \end{align*}

To conclude, it suffices to notice that for $x\in \omega$
\begin{align*}
    a_m(x)& \leq a_m(\dist(z_m, \omega))  \\
    \|x-z_m\|_2 & \leq \left[ \dist(z_m,\omega) + \diam(\omega) \right].
\end{align*}

%

\end{proof}

\subsection*{Acknowledgement} AF acknowledges support from the Wallenberg AI, Autonomous Systems and Software Program (WASP) funded by the Knut and Alice Wallenberg Foundation.
FdG and PW were supported by the ANR Micro-Blind (grant ANR-21-CE48-0008), by the ANR LabEx CIMI (grant ANR-11-LABX-0040) and the support of AI Interdisciplinary Institute ANITI funding, through the French ``Investing for the Future PIA3'' program under the Grant Agreement ANR-19-PI3A-0004. 

\bibliographystyle{plain}
\bibliography{iterative_bio}

\end{document}